\documentclass[10pt]{amsart} 

\usepackage{amsmath}
\usepackage{amssymb}

\usepackage{graphicx}
\usepackage{xcolor}
\usepackage{subcaption}

\usepackage{multirow}

\usepackage{enumitem}

\newtheorem{theorem}{Theorem}[section]
\newtheorem{corollary}[theorem]{Corollary}

\theoremstyle{definition}

\theoremstyle{remark}

\newtheorem{example}[theorem]{Example}

\newcommand{\R}{\mathbb{R}}  
\newcommand{\e}{\mathrm e}
\newcommand{\bfu}{\mathbf u}
\newcommand{\bff}{\mathbf f}
\newcommand{\bfx}{\mathbf x}
\newcommand{\dx}{\mathrm d}
\newcommand{\FD}{\mathrm{FD}}
\newcommand{\JS}{\mathrm{JS}}
\newcommand{\M}{\mathrm M}
\newcommand{\Z}{\mathrm Z}
\newcommand{\ZR}{\mathrm{ZR}}

\title[5th-order WENO schemes]{Fifth-order weighted essentially non-oscillatory schemes with new Z-type nonlinear weights for hyperbolic conservation laws}
\author{Jiaxi Gu}
\address{Department of Mathematics $\&$ POSTECH MINDS (Mathematical Institute for Data Science), Pohang University of Science and Technology, Pohang 37673, Korea}
\email{jiaxigu@postech.ac.kr}

\author{Xinjuan Chen}
\address{Department of Mathematics, College of Science, Jimei University, Xiamen, Fujian 361021, China}
\email{sabrinachern@yahoo.com}

\author{Jae-Hun Jung}
\address{Department of Mathematics $\&$ POSTECH MINDS (Mathematical Institute for Data Science), Pohang University of Science and Technology, Pohang 37673, Korea}
\email{jung153@postech.ac.kr}

\makeatletter
\@namedef{subjclassname@2020}{\textup{}2020 Mathematics Subject Classification}
\makeatother
\subjclass[2020]{65M06}
\keywords{Weighted essentially non-oscillatory, Hyperbolic conservation laws, Smoothness indicators, Z-type nonlinear weights}

\begin{document}

\maketitle

\begin{abstract}
In this paper we propose new Z-type nonlinear weights of the fifth-order weighted essentially non-oscillatory (WENO) finite difference scheme for hyperbolic conservation laws.
Instead of employing the classical smoothness indicators for the nonlinear weights, we take the $p$th root of the smoothness indicators and follow the form of Z-type nonlinear weights, leading to fifth order accuracy in smooth regions, even at the critical points, and sharper approximations around the discontinuities. 
We also prove that the proposed nonlinear weights converge to the linear weights as $p \to \infty$, implying the convergence of the resulting WENO numerical flux to the finite difference numerical flux.
Finally, numerical examples are presented by comparing with other WENO schemes, such as WENO-JS, WENO-M and WENO-Z, to demonstrate that the proposed WENO scheme performs better in shock capturing.
%
\end{abstract}

\section{Introduction} \label{sec:intro}
We consider the system of hyperbolic conservation laws of the form
\begin{equation} \label{eq:hyperbolic}
\begin{aligned}
 \bfu_t + \sum_{i=1}^3 \bff_i(\bfu)_{x_i} &= 0, \\
                            \bfu (\bfx, 0) &= \bfu_0(\bfx), 
\end{aligned}
\end{equation}
where $\bfu(\bfx,t) = (u_1(\bfx,t), \cdots, u_m(\bfx,t))^T$ is a column vector of $m$ conserved unknowns, $\bfx = (x_1, x_2, x_3)$, and, for real $\xi = (\xi_1, \xi_2, \xi_3)$, the combination $\sum_{i=1}^3 \xi_i(\partial \bff / \partial \bfu)$ always has $m$ real eigenvalues and a complete set of eigenvectors.
As is well known, the solution to such system \eqref{eq:hyperbolic} may develop discontinuities even if the initial condition $\bfu_0(\bfx)$ is smooth.
The weighted essentially non-oscillatory (WENO) methods \cite{Balsara,Borges,GuoJungFV,GuoJungFD,Henrick,Liu,Jiang,ShuSpringer} gain the popularity among all numerical schemes since they can obtain the high order accurate approximations right up to discontinuities, while maintaining a sharp, essentially non-oscillatory shock transition.

Jiang and Shu \cite{Jiang} designed the fifth-order WENO scheme (WENO-JS) in the finite difference (FD) form with the smoothness indicators, which are given as a sum of the squares of scaled $L^2$ norms for all the derivatives of the interpolation polynomial over the desired cell.
However, the nonlinear weights $\omega_k$ in \cite{Jiang}, based on the smoothness indicators, failed to satisfy the sufficient condition for the fifth order accuracy at critical points.
To increase the accuracy of the nonlinear weights, Henrick et al. \cite{Henrick} introduced mapping functions that are applied to the nonlinear weights $\omega_k$. 
Not only do the resulting nonlinear weights $\omega_k^M$ meet the requirement of the fifth order accuracy, but also the mapped WENO scheme (WENO-M) produces sharper numerical solutions around discontinuities.
With a different approach, Borges et al. \cite{Borges} introduced the global smoothness indicator and devised Z-type nonlinear weights $\omega_k^Z$,
giving rise to the WENO-Z scheme. 
The WENO-Z scheme further decreases the dissipation near discontinuities, as well as maintains the essentially non-oscillatory behavior.
It is also observed that if the nonlinear weight(s), corresponding to the candidate substencil(s) with the discontinuity, is closer to the linear weight(s), then the WENO scheme generates sharper approximations around the discontinuity.
Nevertheless, to strike a balance between the order of accuracy at critical points and less dissipation near discontinuities, an appropriate value of $p$ has to be determined for the WENO-Z scheme.
For example, $p=1$ in \cite{Borges}.

We continue the path of Z-type nonlinear weights. 
By taking the $p$th root of the classical smoothness indicators and following the form of the Z-type nonlinear weights, we propose a new set of nonlinear weights $\omega_k^{ZR}$.
We show that the new nonlinear weights satisfy the sufficient condition for the fifth order accuracy, even at critical points.
It is also proved that the weights $\omega_k^{ZR}$ converge to the linear weights as $p \to \infty$.
As a result, the dissipation is reduced around the discontinuities with keeping the shock-capturing property if $p$ is properly chosen.

The next section will review the WENO-JS, WENO-M and WENO-Z schemes with an analysis of the nonlinear weights when a discontinuity is involved.
In Section 3, we introduce the new Z-type nonlinear weights $\omega_k^{ZR}$ and provide a detailed discussion about them.
Section 4 contains numerical experiments, including the 1D linear advection equation; 1D Riemann problem for Burgers' equation; 1D and 2D Euler equations of gas dynamics.
Concluding remarks are presented in Section 5.

\section{1D WENO finite difference schemes} \label{sec:1d_fd_WENO}
We begin with the description of fifth-order finite difference WENO scheme for the one-dimensional scalar hyperbolic conservation laws
\begin{equation} \label{eq:1D_hyperbolic}
 u_t(x,t) + f_x(u(x,t)) = 0.
\end{equation} 
Here we focus on the discussion of spatial discretization, while leaving the time variable $t$ continuous.
Assume that $\frac{\partial f}{\partial u} > 0$.
Then the wave is propagating to the right.
Given the computational domain $[a,b]$, we apply a uniform grid with $N+1$ points on the domain, 
$$
   x_i = a + i \Delta x,~i = 0, \cdots, N+1, 
$$
where $\Delta x = \frac{b-a}{N}$ is the grid spacing. 
Each grid point $x_i$ is also called the cell center for the $i$th cell $I_i = \left[ x_{i-1/2},~x_{i+1/2} \right]$, where the cell boundaries are defined as $x_{i \pm 1/2} = x_i \pm \Delta x/2$. 
If we fix the time $t$, the semi-discretization of Equation \eqref{eq:1D_hyperbolic}, by the method of lines, gives
\begin{equation} \label{eq:1D_hyperbolic_discrete}
 \frac{du_{i}(t)}{dt} = - \left. \frac{\partial f \left( u(x,t) \right)}{\partial x} \right |_{x=x_i},
\end{equation} 
where $u_i(t)$ is the numerical approximation to $u(x_i,t)$.
Define the flux function $h(x)$ implicitly by
$$
   f \left( u(x) \right) = \frac{1}{\Delta x} \int^{x+\Delta x/2}_{x-\Delta x/2} h(\xi) \dx \xi. 
$$
Here the time variable $t$ is dropped for the notational convenience. 
Differentiating both sides with respect to $x$ yields
\begin{equation} \label{eq:partial_f}
 \frac{\partial f}{\partial x} = \frac{h(x+\Delta x/2) - h(x-\Delta x/2)}{\Delta x}.
\end{equation}
Evaluating \eqref{eq:partial_f} at $x=x_i$ and plugging this into \eqref{eq:1D_hyperbolic_discrete}, we obtain 
\begin{equation} \label{eq:1D_hyperbolic_discrete_h}
 \frac{du_{i}(t)}{dt} = - \frac{h_{i+1/2} - h_{i-1/2}}{\Delta x},
\end{equation} 
with $h_{i \pm 1/2} = h(x_{i \pm 1/2})$.
Our goal is to reconstruct the fluxes $h_{i \pm 1/2}$ at cell boundaries to a high order accuracy, and simultaneously to avoid spurious oscillations
near shocks.

Since the advection speed is positive, a polynomial approximation $p(x)$ to $h(x)$ of degree at most four on the $5$-point stencil $S^5$, as shown in Figure \ref{fig:stencil}, is sufficient to approximate $h_{i+1/2}$.
\begin{figure}[h!]
\centering
\includegraphics[width=\textwidth]{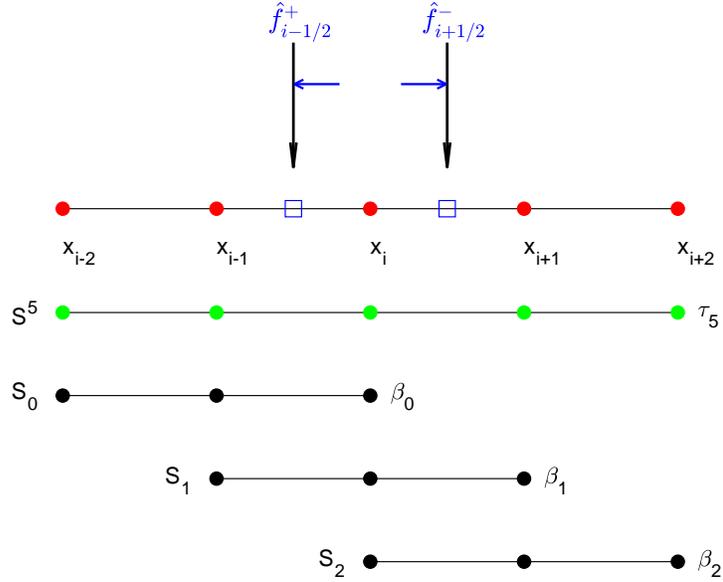}
\caption{The construction of numerical fluxes $\hat{f}^-_{i+1/2}$ and $\hat{f}^+_{i-1/2}$ depends on the stencil $S^5 = \{ x_{i-2}, \cdots, x_{i+2} \}$ with five uniform points, as well as three $3$-point substencils $S_0, S_1, S_2$.}
\label{fig:stencil}
\end{figure} 
Evaluating $p(x)$ at $x = x_{i+1/2}$ gives the finite difference numerical flux
\begin{equation} \label{eq:numerical_flux_approximation_plus}
 \hat{f}^{\FD}_{i+1/2} = \frac{1}{30} f_{i-2} - \frac{13}{60} f_{i-1} + \frac{47}{60} f_i + \frac{9}{20} f_{i+1} - \frac{1}{20} f_{i+2}
\end{equation}
with $f_j = f(u_j)$.
The numerical flux $\hat{f}_{i-1/2}$ is obtained by shifting one grid to the left, yielding
\begin{equation} \label{eq:numerical_flux_approximation_minus}
 \hat{f}^{\FD}_{i-1/2} = \frac{1}{30} f_{i-3} - \frac{13}{60} f_{i-2} + \frac{47}{60} f_{i-1} + \frac{9}{20} f_i - \frac{1}{20} f_{i+1}.
\end{equation}
Applying the Taylor expansions to $\hat{f}^{\FD}_{i \pm 1/2}$ \eqref{eq:numerical_flux_approximation_plus} and \eqref{eq:numerical_flux_approximation_minus} would give
\begin{align}
 \hat{f}^{\FD}_{i+1/2} &= h_{i+1/2} - \frac{1}{60} h_i^{(5)} \Delta x^5 + O(\Delta x^6), \label{eq:numerical_flux_approximation_plus_Taylor} \\
 \hat{f}^{\FD}_{i-1/2} &= h_{i-1/2} - \frac{1}{60} h_i^{(5)} \Delta x^5 + O(\Delta x^6). \label{eq:numerical_flux_approximation_minus_Taylor}
\end{align}
Substituting \eqref{eq:numerical_flux_approximation_plus_Taylor} and \eqref{eq:numerical_flux_approximation_minus_Taylor} for $h_{i \pm 1/2}$ in \eqref{eq:1D_hyperbolic_discrete_h}, we obtain the upwind scheme with fifth order accuracy in space
$$
   \frac{du_{i}(t)}{dt} = - \frac{\hat{f}^{\FD}_{i+1/2} - \hat{f}^{\FD}_{i-1/2}}{\Delta x} + O(\Delta x^5).
$$
The same procedure could be carried out on the $3$-point substencils $S_0, S_1$ and $S_2$ shown in Figure \ref{fig:stencil}, giving the numerical fluxes,
\begin{equation} \label{eq:numerical_flux_approximation_plus_substencil}
\begin{aligned}
 \hat{f}^0_{i+1/2} &=   \frac{1}{3} f_{i-2} - \frac{7}{6} f_{i-1} + \frac{11}{6} f_i, \\
 \hat{f}^1_{i+1/2} &= - \frac{1}{6} f_{i-1} + \frac{5}{6} f_i     + \frac{1}{3} f_{i+1}, \\
 \hat{f}^2_{i+1/2} &=   \frac{1}{3} f_i     + \frac{5}{6} f_{i+1} - \frac{1}{6} f_{i+2}.
\end{aligned}
\end{equation}
Taylor expansions of \eqref{eq:numerical_flux_approximation_plus_substencil} result in
\begin{equation*} 
\begin{aligned}
 \hat{f}^0_{i+1/2} &= h_{i+1/2} -  \frac{1}{4} h_i^{(3)} \Delta x^3 + O(\Delta x^4), \\
 \hat{f}^1_{i+1/2} &= h_{i+1/2} + \frac{1}{12} h_i^{(3)} \Delta x^3 + O(\Delta x^4), \\
 \hat{f}^2_{i+1/2} &= h_{i+1/2} - \frac{1}{12} h_i^{(3)} \Delta x^3 + O(\Delta x^4).
\end{aligned}
\end{equation*}
Then 
\begin{equation} \label{eq:numerical_flux_approximation_plus_substencil_Taylor}
 \hat{f}^k_{i+1/2} = h_{i+1/2} + A_k \Delta x^3 + O(\Delta x^4),~k=0,1,2.
\end{equation}
Similarly, an index shift by $-1$ returns the corresponding $\hat{f}^k_{i-1/2}$:
\begin{equation} \label{eq:numerical_flux_approximation_minua_substencil}
\begin{aligned}
 \hat{f}^0_{i-1/2} &=   \frac{1}{3} f_{i-3} - \frac{7}{6} f_{i-2} + \frac{11}{6} f_{i-1}, \\
 \hat{f}^1_{i-1/2} &= - \frac{1}{6} f_{i-2} + \frac{5}{6} f_{i-1} + \frac{1}{3}  f_i, \\
 \hat{f}^2_{i-1/2} &=   \frac{1}{3} f_{i-1} + \frac{5}{6} f_i     - \frac{1}{6}  f_{i+1},
\end{aligned}
\end{equation}
where, by Taylor expansion, 
\begin{equation} \label{eq:numerical_flux_approximation_minus_substencil_Taylor}
 \hat{f}^k_{i-1/2} = h_{i-1/2} + A_k \Delta x^3 + O(\Delta x^4),~k=0,1,2.
\end{equation}
It is clear that the linear combination of $\hat{f}^k_{i+1/2}$ in \eqref{eq:numerical_flux_approximation_plus_substencil} could reduce to $\hat{f}^{\FD}_{i+1/2}$ in \eqref{eq:numerical_flux_approximation_plus}. 
In other words, there are linear weights $d_0 = \frac{1}{10}, d_1 = \frac{3}{5}$ and $d_2 = \frac{3}{10}$ such that
\begin{equation} \label{eq:numerical_flux_approximation_linear_weights}
 \hat{f}^{\FD}_{i+1/2} = \sum_{k=0}^2 d_k \hat{f}^k_{i+1/2}.
\end{equation}

The WENO numerical flux $\hat{f}_{i+1/2}$, which is an approximation of $h_{i+1/2}$, takes a convex combination of three candidate numerical fluxes $\hat{f}^k_{i+1/2}$, 
\begin{equation} \label{eq:numerical_flux_approximation}
 \hat{f}_{i+1/2} = \sum_{k=0}^2 \omega_k \hat{f}^k_{i+1/2}.
\end{equation}
Note that the numerical flux $\hat{f}_{i+1/2}$ in \eqref{eq:numerical_flux_approximation} is actually $\hat{f}^-_{i+1/2}$ in Figure \ref{fig:stencil}.
We drop the superscript $-$ to simplify the notation.
The key to the success of WENO is the choice of the weights $\omega_k$, where it is required that
\begin{equation} \label{eq:weights}
 \omega_k \geqslant 0,~~\sum_{k=0}^2 \omega_k = 1
\end{equation}
for stability and consistency.
In \cite{Jiang}, the nonlinear weights $\omega^{\JS}_k$ in \eqref{eq:numerical_flux_approximation} are given by
\begin{equation} \label{eq:weights_JS}
 \omega^{\JS}_k = \frac{\alpha_k}{\sum^2_{s=0} \alpha_s},~~\alpha_k = \frac{d_k}{(\beta_k + \epsilon)^2},~~k=0,1,2.
\end{equation} 
Here $\epsilon>0$, e.g., $\epsilon = 10^{-6}$, is introduced to avoid the case that the denominator becomes zero.
The smoothness indicators $\beta_k$ in \eqref{eq:weights_JS} are defined as
$$
   \beta_k = \sum_{l=1}^2 \Delta x^{2l-1} \int_{x_{i-1/2}}^{x_{i+1/2}} \left( \frac{\dx^l}{\dx x^l} p_k(x) \right)^2 \dx x,
$$
and
\begin{equation} \label{eq:smooth_indicator_JS}
\begin{aligned}
 \beta_0 &= \frac{13}{12} \left( f_{i-2} - 2 f_{i-1} + f_i \right)^2 + \frac{1}{4} \left( f_{i-2} - 4 f_{i-1} + 3 f_i \right)^2, \\
 \beta_1 &= \frac{13}{12} \left( f_{i-1} - 2 f_i + f_{i+1} \right)^2 + \frac{1}{4} \left( f_{i-1} - f_{i+1} \right)^2, \\
 \beta_2 &= \frac{13}{12} \left( f_i - 2 f_{i+1} + f_{i+2} \right)^2 + \frac{1}{4} \left( 3 f_i - 4 f_{i+1} + f_{i+2} \right)^2.
\end{aligned}
\end{equation}
Expanding \eqref{eq:smooth_indicator_JS} in Taylor series around $x=x_i$ gives
\begin{equation} \label{eq:smooth_indicator_JS_Taylor}
\begin{aligned}
 \beta_0 &= f'^2_i \Delta x^2 + \left( \frac{13}{12} f''^2_i - \frac{2}{3} f'_i f'''_i \right) \Delta x^4 - \left( \frac{13}{6} f''_i f'''_i - \frac{1}{2} f'_i f^{(4)}_i \right) \Delta x^5 + O(\Delta x^6), \\
 \beta_1 &= f'^2_i \Delta x^2 + \left( \frac{13}{12} f''^2_i + \frac{1}{3} f'_i f'''_i \right) \Delta x^4 + O(\Delta x^6), \\
 \beta_2 &= f'^2_i \Delta x^2 + \left( \frac{13}{12} f''^2_i - \frac{2}{3} f'_i f'''_i \right) \Delta x^4 + \left( \frac{13}{6} f''_i f'''_i - \frac{1}{2} f'_i f^{(4)}_i \right) \Delta x^5 + O(\Delta x^6).
\end{aligned}
\end{equation}
In \cite{Henrick}, Henrick et al. introduced the mapping functions $g_k$ given by
$$
   g_k(\omega) = \frac{\omega (d_k + d_k^2 - 3 d_k \omega + \omega^2)}{d_k^2 + \omega (1 - 2d_k)},~k = 0,1,2,
$$
to define the mapped nonlinear weights $\omega^\M_k$ by
\begin{equation} \label{eq:weights_M}
 \omega^\M_k = \frac{\alpha^*_k}{\sum^2_{s=0} \alpha^*_s},~~\alpha^*_k = g_k(\omega^{\JS}_k),~~k=0,1,2,
\end{equation}
with $\omega^{\JS}_k$ the weights in \eqref{eq:weights_JS}.
In \cite{Borges}, Borges et al. introduced the global smoothness indicator $\tau_5$ as the absolute difference between $\beta_0$ and $\beta_2$,
$$
   \tau_5 = | \beta_0 - \beta_2 |.
$$
The Z-type nonlinear weights $\omega^\Z_k$ are defined as 
\begin{equation} \label{eq:weights_Z}
 \omega^\Z_k = \frac{\alpha_k}{\sum^2_{s=0} \alpha_s},~~\alpha_k = d_k \left( 1 + \left( \frac{\tau_5}{\beta_k + \epsilon} \right)^p \right),~~k=0,1,2.
\end{equation}

\subsection{Spatial fifth order accuracy in smooth regions} \label{sec:1d_fd_WENO_smooth}
Assume that the solution is smooth on the stencil $S^5$. 
In \cite{Henrick}, Henrick et al. derived the sufficient condition for the fifth order accuracy in space as follows.
Let
$$
   \hat{f}_{i \pm 1/2} = \sum_{k=0}^2 \omega^{\pm}_k \hat{f}^k_{i \pm 1/2}.
$$
Since
$$
   \hat{f}_{i+1/2} = \sum_{k=0}^2 d_k \hat{f}^k_{i+1/2} + \sum_{k=0}^2 ( \omega^+_k - d_k ) \hat{f}^k_{i+1/2} 
                   = \hat{f}^{\FD}_{i+1/2} + \sum_{k=0}^2 ( \omega^+_k - d_k ) \hat{f}^k_{i+1/2},
$$
and
\begin{align*}
 \sum_{k=0}^2 ( \omega^+_k - d_k ) \hat{f}^k_{i+1/2} &= \sum_{k=0}^2 ( \omega^+_k - d_k ) \left[ h_{i+1/2} + A_k \Delta x^3 + O(\Delta x^4) \right] \\
                                                     &= h_{i+1/2} \sum_{k=0}^2 ( \omega^+_k - d_k ) + \Delta x^3 \sum_{k=0}^2 A_k ( \omega^+_k - d_k ) + \sum_{k=0}^2 ( \omega^+_k - d_k ) O(\Delta x^4) \\
                                                     &= \Delta x^3 \sum_{k=0}^2 A_k ( \omega^+_k - d_k ) + \sum_{k=0}^2 ( \omega^+_k - d_k ) O(\Delta x^4), 
\end{align*}
then
$$
   \hat{f}_{i+1/2} = \hat{f}^{\FD}_{i+1/2} + \Delta x^3 \sum_{k=0}^2 A_k ( \omega^+_k - d_k ) + \sum_{k=0}^2 ( \omega^+_k - d_k ) O(\Delta x^4). 
$$
Similarly,
$$
 \hat{f}_{i-1/2} = \hat{f}^{\FD}_{i-1/2} + \Delta x^3 \sum_{k=0}^2 A_k ( \omega^-_k - d_k ) + \sum_{k=0}^2 ( \omega^-_k - d_k ) O(\Delta x^4). 
$$
By \eqref{eq:numerical_flux_approximation_plus_Taylor} and \eqref{eq:numerical_flux_approximation_minus_Taylor}, we have
\begin{align*}
 \frac{\hat{f}_{i+1/2} - \hat{f}_{i-1/2}}{\Delta x} = {} & \frac{h_{i+1/2} - h_{i-1/2}}{\Delta x} + O(\Delta x^5) + \Delta x^2 \sum_{k=0}^2 A_k ( \omega^+_k - \omega^-_k ) \\
                                                      {} & + \sum_{k=0}^2 ( \omega^+_k - d_k ) O(\Delta x^3) - \sum_{k=0}^2 ( \omega^-_k - d_k ) O(\Delta x^3).
\end{align*} 
Thus the sufficient condition for fifth order accuracy is given by
\begin{equation} \label{eq:WENO_condition}
 \omega_k - d_k = O(\Delta x^3),
\end{equation} 
where the superscripts are dropped, meaning that the nonlinear weights $\omega_k$ for each stencil $S^5$ should satisfy the condition \eqref{eq:WENO_condition} in order to attain the fifth order accuracy in space.

For the nonlinear weights $\omega^{\JS}_k$ in \eqref{eq:weights_JS}, it can be shown via Taylor expansion that $\omega^{\JS}_k = d_k + O(\Delta x^2)$ if $f'_i \ne 0$, while $\omega^{\JS}_k = d_k + O(\Delta x)$ if $f'_i = 0$.
So the weights $\omega^{\JS}_k$ do not satisfy the sufficient condition \eqref{eq:WENO_condition}. 
Taylor series expansion shows that the mapped nonlinear weights $\omega^\M_k$ satisfy $\omega^\M_k = d_k + O(\Delta x^3)$ regardless of the value of $f'_i$, which agrees with the sufficient condition \eqref{eq:WENO_condition}.
With $p=1$, the Z-type nonlinear weights $\omega^\Z_k$ implies that $\omega^\Z_k = d_k + O(\Delta x^3)$ if $f'_i \ne 0$, whereas $\omega^\Z_k = d_k + O(\Delta x)$ if $f'_i = 0$.
Hence the weights $\omega^\Z_k$ satisfy the sufficient condition \eqref{eq:WENO_condition} at non-critical points.

\subsection{Dissipation around discontinuities} \label{sec:1d_fd_WENO_discontinuity}
In \cite{Borges}, Borges et al. observed that if the WENO scheme assigns substantially large weight(s) to discontinuous stencil(s), the resulting dissipation near discontinuities is decreased, thus improving the behavior of shock-capturing.
In other words, there is a negative correlation between the nonlinear weights $\omega_k$ for the substencils with the discontinuity and the dissipation around the discontinuity.
We invoke the example of the linear advection equation in \cite{Borges}:
\begin{equation} \label{eq:weight_comparison}
\begin{aligned}
 u_t + u_x &= 0, \\
    u(x,0) &= \left\{ 
               \begin{array}{ll} 
                - \sin(\pi x) - \frac{1}{2} x^3,     & -1 < x < 0, \\ 
                - \sin(\pi x) - \frac{1}{2} x^3 + 1, & 0 \leqslant x \leqslant 1,
               \end{array} 
              \right.
\end{aligned}
\end{equation}
where the initial condition is a piecewise function with a jump discontinuity at $x=0$, to reproduce the nonlinear weights $\omega_k$ for the WENO-JS ($\epsilon = 10^{-6}$), WENO-M ($\epsilon = 10^{-40}$) and WENO-Z ($\epsilon = 10^{-40}$) schemes at the first step of the numerical solutions.
Figure \ref{fig:weight_comparison} and Table \ref{tab:weight_comparison} show those nonlinear weights $\omega_k$ around the discontinuity, compared with the linear weights $d_k$.
It is obvious that the nonlinear weights $\omega^\Z_k$ for the substencils containing the discontinuity are orders of magnitude larger than the corresponding nonlinear weights $\omega^{\JS}_k$.

\begin{figure}[h!]
\centering
 \begin{subfigure}[b]{0.32\textwidth}
 \centering
 \includegraphics[width=\textwidth]{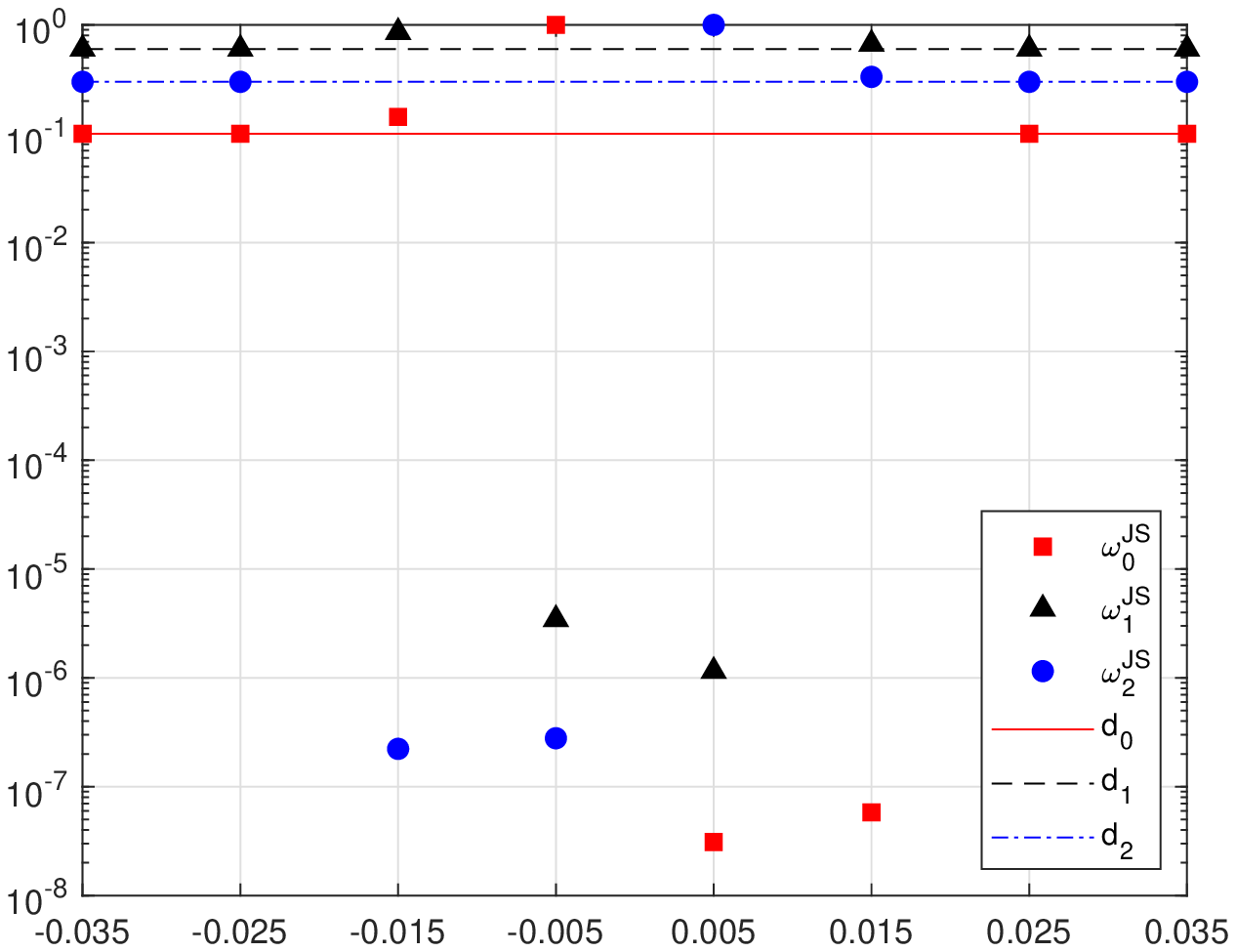}
 \caption{WENO-JS}
 \end{subfigure}
 \hfill
 \begin{subfigure}[b]{0.32\textwidth}
 \centering
 \includegraphics[width=\textwidth]{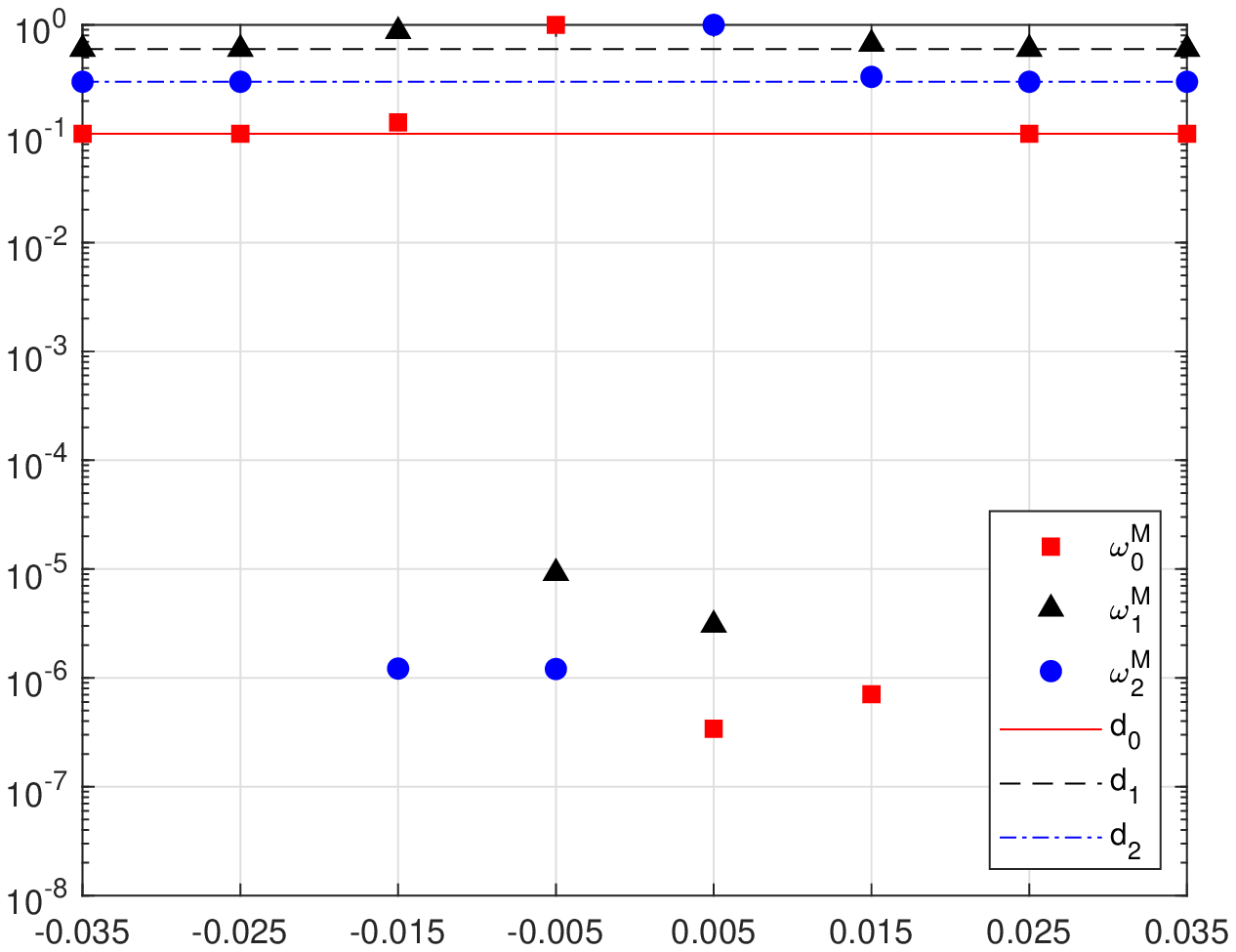}
 \caption{WENO-M}
 \end{subfigure}
 \hfill
 \begin{subfigure}[b]{0.32\textwidth}
 \centering
 \includegraphics[width=\textwidth]{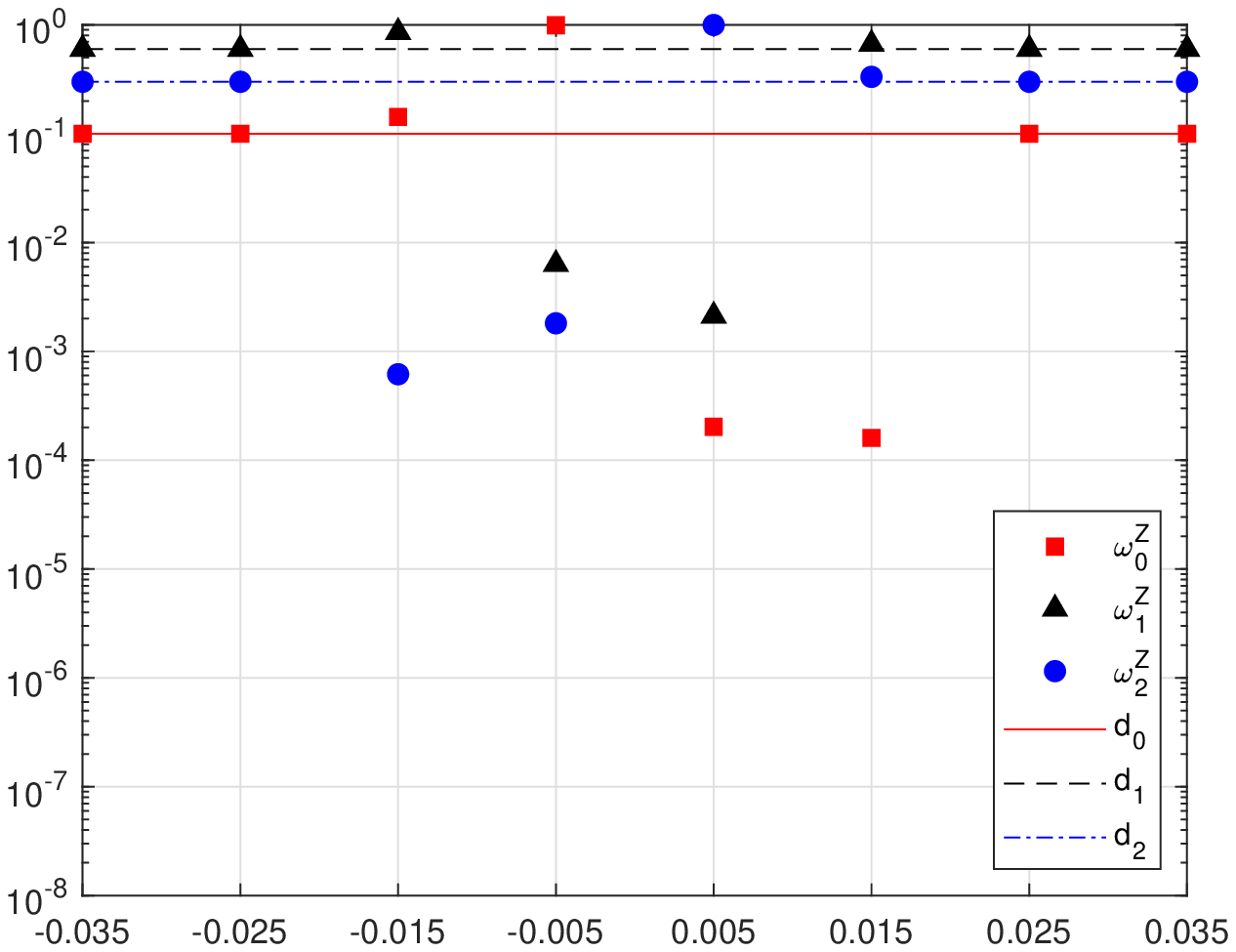}
 \caption{WENO-Z}
 \end{subfigure}
\caption{The distribution of the nonlinear weights $\omega_k$ near the discontinuity at $x=0$, compared with the linear weights $d_k$, for WENO-JS, WENO-M and WENO-Z at the first step of the numerical solutions of the test example \eqref{eq:weight_comparison} on a semilog plot with a $\log_{10}$ scale on the vertical axis. 
The linear weights $d_k$ are shown in lines while the nonlinear weights $\omega_k$ symbols.}
\label{fig:weight_comparison}
\end{figure}

\begin{table}[h!]
\renewcommand{\arraystretch}{1.2}
\centering
\begin{tabular}{ccccccccc|c} 
\hline
$x$ & -0.035 & -0.025 & -0.015 & -0.005 & 0.005 & 0.015 & 0.025 & 0.035 &  \\ 
\hline 
$\omega^{\JS}_0$ & 0.099892 & 0.099892 & 0.142639 & 0.999996 & 3.103e-8 & 5.804e-8 & 0.099894 & 0.099894 & \multirow{3}{*}{$d_0 = 0.1$} \\  
$\omega^\M_0$    & 0.100000 & 0.100000 & 0.127205 & 0.999990 & 3.413e-7 & 7.082e-7 & 0.100000 & 0.100000 & \\  
$\omega^\Z_0$    & 0.100000 & 0.100000 & 0.142660 & 0.991870 & 2.027e-4 & 1.604e-4 & 0.100000 & 0.100000 & \\
\hline 
$\omega^{\JS}_1$ & 0.600426 & 0.600426 & 0.857361 & 3.448e-6 & 1.151e-6 & 0.667063 & 0.600428 & 0.600428 & \multirow{3}{*}{$d_1 = 0.6$} \\  
$\omega^\M_1$    & 0.600000 & 0.600000 & 0.872794 & 9.195e-6 & 3.070e-6 & 0.667040 & 0.600000 & 0.600000 & \\  
$\omega^\Z_1$    & 0.600000 & 0.600000 & 0.856724 & 6.318e-3 & 2.120e-3 & 0.666758 & 0.600000 & 0.600000 & \\
\hline
$\omega^{\JS}_2$ & 0.299682 & 0.299681 & 2.226e-7 & 2.788e-7 & 0.999999 & 0.332937 & 0.299678 & 0.299678 & \multirow{3}{*}{$d_2 = 0.3$} \\  
$\omega^\M_2$    & 0.300000 & 0.300000 & 1.220e-6 & 1.208e-6 & 0.999997 & 0.332959 & 0.300000 & 0.300000 & \\  
$\omega^\Z_2$    & 0.300000 & 0.300000 & 6.166e-4 & 1.812e-3 & 0.997677 & 0.333082 & 0.300000 & 0.300000 & \\  
\hline
\end{tabular}
\caption{The values of the nonlinear weights $\omega_k$ near the discontinuity at $x=0$, compared with the linear weights $d_k$, for WENO-JS, WENO-M and WENO-Z at the first step of the numerical solutions of the test example \eqref{eq:weight_comparison}.}
\label{tab:weight_comparison}
\end{table}

Let us take a close look at how the nonlinear weights $\omega_k$ behave when there exists a discontinuity involved.
We assume that there is a discontinuity in the interval $[x_i, x_{i+1}]$, and except that, the solution is continuous at the nearby points.
Then there are four $5$-point stencils containing the discontinuity:
\begin{enumerate}[label=(\roman*)]
\item $\{ x_{i-3}, x_{i-2}, x_{i-1}, x_i, x_{i+1} \}$;
\item $\{ x_{i-2}, x_{i-1}, x_i, x_{i+1}, x_{i+2} \}$;
\item $\{ x_{i-1}, x_i, x_{i+1}, x_{i+2}, x_{i+3} \}$;
\item $\{ x_i, x_{i+1}, x_{i+2}, x_{i+3}, x_{i+4} \}$.
\end{enumerate}
By setting $\epsilon = 0$, \eqref{eq:weights_JS} implies that
\begin{align}
 \omega^{\JS}_0 &= \frac{1}{d_0 + d_1 \left( \frac{\beta_0}{\beta_1} \right)^2 + d_2 \left( \frac{\beta_0}{\beta_2} \right)^2} d_0, \label{eq:weights_JS_0} \\
 \omega^{\JS}_1 &= \frac{1}{d_0 \left( \frac{\beta_1}{\beta_0} \right)^2 + d_1 + d_2 \left( \frac{\beta_1}{\beta_2} \right)^2} d_1, \label{eq:weights_JS_1} \\
 \omega^{\JS}_2 &= \frac{1}{d_0 \left( \frac{\beta_2}{\beta_0} \right)^2 + d_1 \left( \frac{\beta_2}{\beta_1} \right)^2 + d_2} d_2. \label{eq:weights_JS_2}
\end{align}
Similarly, with \eqref{eq:weights_Z}, we have
\begin{align}
 \omega^\Z_0 &= \frac{1}{d_0 + d_1 \frac{1 + \frac{\tau_5}{\beta_1}}{1 + \frac{\tau_5}{\beta_0}} + d_2 \frac{1 + \frac{\tau_5}{\beta_2}}{1 + \frac{\tau_5}{\beta_0}}} d_0, \label{eq:weights_Z_0} \\
 \omega^\Z_1 &= \frac{1}{d_0 \frac{1 + \frac{\tau_5}{\beta_0}}{1 + \frac{\tau_5}{\beta_1}} + d_1 + d_2 \frac{1 + \frac{\tau_5}{\beta_2}}{1 + \frac{\tau_5}{\beta_1}}} d_1, \label{eq:weights_Z_1} \\
 \omega^\Z_2 &= \frac{1}{d_0 \frac{1 + \frac{\tau_5}{\beta_0}}{1 + \frac{\tau_5}{\beta_2}} + d_1 \frac{1 + \frac{\tau_5}{\beta_1}}{1 + \frac{\tau_5}{\beta_2}} + d_2} d_2. \label{eq:weights_Z_2}
\end{align}

For case (i), the smoothness indicators satisfy $\beta_0, \beta_1 = O(\Delta x^2)$ and $\beta_2 = O(1)$. 
Then $\tau_5 = | \beta_0 - \beta_2 | \approx \beta_2$ and 
\begin{equation} \label{eq:betak2_ineq}
 \beta_k \ll \beta_2,~~1 + \frac{\tau_5}{\beta_k} \gg 1 + \frac{\tau_5}{\beta_2},~~k = 0,1.
\end{equation}
It follows that
\begin{equation} \label{eq:betak2}
 \frac{\beta_k}{\beta_2} = o(1),~~\frac{1 + \frac{\tau_5}{\beta_2}}{1 + \frac{\tau_5}{\beta_k}} = o(1),~~k = 0,1.
\end{equation}
Since $\displaystyle{\lim_{\Delta x \to 0} \beta_1 / \beta_0 = 1}$, we see that 
$$
   \lim_{\Delta x \to 0} \frac{1 + \frac{\tau_5}{\beta_0}}{1 + \frac{\tau_5}{\beta_1}} - 1 
 = \lim_{\Delta x \to 0} \frac{\frac{\beta_1}{\beta_0} - 1}{1 + \frac{\beta_1}{\tau_5}} = 0,
$$
which gives
\begin{equation} \label{eq:beta01}
 \frac{\beta_0}{\beta_1} = 1 + o(1),~\frac{1 + \frac{\tau_5}{\beta_0}}{1 + \frac{\tau_5}{\beta_1}} = 1 + o(1).
\end{equation}
From \eqref{eq:betak2} and \eqref{eq:beta01}, we find
$$
   d_0 + d_1 \left( \frac{\beta_0}{\beta_1} \right)^2 + d_2 \left( \frac{\beta_0}{\beta_2} \right)^2 = d_0 + d_1 \left( 1 + o(1) \right)^2 + d_2 \, o(1) = d_0 + d_1 + o(1) < d_0 + d_1 + d_2 = 1,
$$
and
$$
   d_0 + d_1 \frac{1 + \frac{\tau_5}{\beta_1}}{1 + \frac{\tau_5}{\beta_0}} + d_2 \frac{1 + \frac{\tau_5}{\beta_2}}{1 + \frac{\tau_5}{\beta_0}} = d_0 + d_1 \left( 1 + o(1) \right) + d_2 \, o(1) = d_0 + d_1 + o(1) < d_0 + d_1 + d_2 = 1,
$$
and hence $\omega^{\JS}_0 > d_0$ and $\omega^\Z_0 > d_0$ by \eqref{eq:weights_JS_0} and \eqref{eq:weights_Z_0}, respectively.
Using the same approach, we obtain $\omega^{\JS}_1,~\omega^\Z_1 > d_1$.
Note that for $k=0,1$,
$$
   \lim_{\Delta x \to 0} \frac{\beta_k + \tau_5}{\beta_2 + \tau_5} = \lim_{\Delta x \to 0} \frac{\beta_k + \beta_2 - \beta_0}{\beta_2 + \beta_2 - \beta_0} = \lim_{\Delta x \to 0} \frac{\beta_k - \beta_0}{2 \beta_2 - \beta_0} + \lim_{\Delta x \to 0} \frac{\beta_2}{2 \beta_2 - \beta_0}= \frac{1}{2}.
$$
We thus expect that if $\Delta x$ is small enough,
\begin{equation} \label{eq:betak2_tau5_approx}
 \frac{\beta_k + \tau_5}{\beta_2 + \tau_5} \approx \frac{1}{2}.  
\end{equation}
Then we have from $\beta_2 \gg \beta_k$ and \eqref{eq:betak2_tau5_approx},
$$
   \frac{\beta_2}{\beta_k} \left( \frac{\beta_2}{\beta_k} - \frac{\beta_k + \tau_5}{\beta_1 + \tau_5} \right) \gg 1 > 0,
$$
and hence
$$
   \left( \frac{\beta_2}{\beta_k} \right)^2 \gg \frac{\beta_2}{\beta_k} \frac{\beta_k + \tau_5}{\beta_2 + \tau_5} = \frac{1 + \frac{\tau_5}{\beta_k}}{1 + \frac{\tau_5}{\beta_2}}.
$$
The above inequality and \eqref{eq:betak2_ineq} give
$$
   d_0 \left( \frac{\beta_2}{\beta_0} \right)^2 + d_1 \left( \frac{\beta_2}{\beta_1} \right)^2 + d_2 \gg d_0 \frac{1 + \frac{\tau_5}{\beta_0}}{1 + \frac{\tau_5}{\beta_2}} + d_1 \frac{1 + \frac{\tau_5}{\beta_1}}{1 + \frac{\tau_5}{\beta_2}} + d_2 \gg d_0 + d_1 + d_2 = 1,
$$
so $\omega^{\JS}_2 \ll \omega^\Z_2 \ll d_2$ by \eqref{eq:weights_JS_2} and \eqref{eq:weights_Z_2}.
Rewriting $\omega^{\JS}_2$ in \eqref{eq:weights_JS_2} and $\omega^\Z_2$ in \eqref{eq:weights_Z_2}, we find 
\begin{align*}
 \lim_{\Delta x \to 0} \omega^{\JS}_2 &= \lim_{\Delta x \to 0} \frac{d_2 \beta_0^2 \beta_1^2}{d_0 \beta_1^2 \beta_2^2 + d_1 \beta_0^2 \beta_2^2 + d_2 \beta_0^2 \beta_1^2} = 0, \\
 \lim_{\Delta x \to 0} \omega^\Z_2 &= \lim_{\Delta x \to 0} \frac{d_2 \beta_0 \beta_1}{d_0 \beta_1 \beta_2 \frac{\beta_0 + \tau_5}{\beta_2 + \tau_5} + d_1 \beta_0 \beta_2 \frac{\beta_1 + \tau_5}{\beta_2 + \tau_5} + d_2 \beta_0 \beta_1} = 0.
\end{align*}
Then $\omega^{\JS}_2, \omega^Z_2 \gtrapprox 0$, which follows from the above limits and the requirement \eqref{eq:weights}. 
Therefore, 
$$ 
   \omega^{\JS}_0,~\omega^\Z_0 > d_0,~~\omega^{\JS}_1,~\omega^\Z_1 > d_1,~~\omega^{\JS}_2 \ll \omega^\Z_2 \ll d_2,~~\omega^{\JS}_2, \omega^Z_2 \gtrapprox 0,
$$
which agree precisely with the results at $x=-0.015$ in Table \ref{tab:weight_comparison}.

Next we consider the case (ii), where $\beta_0= O(\Delta x^2)$ and $\beta_1, \beta_2 = O(1)$.
Then $\tau_5 = | \beta_0 - \beta_2 | \approx \beta_2$ and 
$$
   \beta_0 \ll \beta_k,~~1 + \frac{\tau_5}{\beta_0} \gg 1 + \frac{\tau_5}{\beta_k},~~k = 1,2,
$$
which in turn implies 
\begin{equation} \label{eq:beta0k}
 \frac{\beta_0}{\beta_k} = o(1),~~\frac{1 + \frac{\tau_5}{\beta_k}}{1 + \frac{\tau_5}{\beta_0}} = o(1),~~k = 1,2.
\end{equation}
Applying \eqref{eq:beta0k} to $\omega^{\JS}_0$ in \eqref{eq:weights_JS_0} and $\omega^\Z_0$ in \eqref{eq:weights_Z_0}, we find
\begin{align*}
 \omega^{\JS}_0 &= \frac{1}{d_0 + d_1 \, o(1) + d_2 \, o(1)} d_0 = 1 + o(1), \\
 \omega^Z_0 &= \frac{1}{d_0 + d_1 \, o(1) + d_2 \, o(1)} d_0 = 1 + o(1).
\end{align*}
Combining this with the requirement \eqref{eq:weights} gives $\omega^{\JS}_0, \omega^Z_0 \lessapprox 1$.
Since, for $k=1,2$,
$$
   \frac{\beta_0}{\beta_k} - \frac{\beta_k + \tau_5}{\beta_0 + \tau_5} = \frac{(\beta^2_0 - \beta^2_k) + \tau_5 (\beta_0 - \beta_k)}{\beta_k (\beta_0 + \tau_5)} < 0,
$$
we have 
$$
   \left( \frac{\beta_0}{\beta_k} \right)^2 < \frac{\beta_0}{\beta_k} \frac{\beta_k + \tau_5}{\beta_0 + \tau_5} = \frac{1 + \frac{\tau_5}{\beta_k}}{1 + \frac{\tau_5}{\beta_0}},
$$
and so
$$
   d_0 + d_1 \left( \frac{\beta_0}{\beta_1} \right)^2 + d_2 \left( \frac{\beta_0}{\beta_2} \right)^2 < d_0 + d_1 \frac{1 + \frac{\tau_5}{\beta_1}}{1 + \frac{\tau_5}{\beta_0}} + d_2 \frac{1 + \frac{\tau_5}{\beta_2}}{1 + \frac{\tau_5}{\beta_0}} < d_0 + d_1 + d_2 = 1,
$$
and hence $\omega^{\JS}_0 > \omega^\Z_0 > d_0$ according to \eqref{eq:weights_JS_0} and \eqref{eq:weights_Z_0}.
We can rewrite $\beta_0, \beta_1$ and $\beta_2$ in \eqref{eq:smooth_indicator_JS} as
\begin{align*}
 \beta_0 &= \frac{13}{12} \left( (f_{i-2}-f_{i-1}) + (f_i-f_{i-1}) \right)^2 + \frac{1}{4} \left( (f_{i-2}-f_{i-1}) + 3 (f_i-f_{i-1}) \right)^2, \\
 \beta_1 &= \frac{13}{12} \left( (f_{i-1} - f_i) + (f_{i+1} - f_i) \right)^2 + \frac{1}{4} \left( (f_{i-1} - f_i) + (f_i - f_{i+1}) \right)^2, \\
 \beta_2 &= \frac{13}{12} \left( (f_i-f_{i+1}) + (f_{i+2}-f_{i+1}) \right)^2 + \frac{1}{4} \left( 3 (f_i-f_{i+1}) + (f_{i+2}-f_{i+1}) \right)^2.
\end{align*}
Since $f_{i-2}-f_{i-1}, f_{i-1}-f_i, f_{i+1}-f_{i+2} = O(\Delta x)$ and $f_i-f_{i+1} = O(1)$, we obtain
$$
   \lim_{\Delta x \to 0} \frac{\beta_1}{\beta_2} = \lim_{\Delta x \to 0} \frac{\frac{13}{12} \left( O(\Delta x) + (f_{i+1} - f_i) \right)^2 + \frac{1}{4} \left( O(\Delta x) + (f_i - f_{i+1}) \right)^2}{\frac{13}{12} \left( (f_i - f_{i+1}) + O(\Delta x) \right)^2 + \frac{1}{4} \left( 3 (f_i - f_{i+1}) + O(\Delta x) \right)^2} = \frac{2}{5}.
$$
Using the same technique yields
$$
 \lim_{\Delta x \to 0} \frac{\beta_0 + \tau_5}{\beta_1 + \tau_5} = \frac{5}{7},~~\lim_{\Delta x \to 0} \frac{\beta_2 + \tau_5}{\beta_1 + \tau_5} = \frac{10}{7}. 
$$
It is expected that 
\begin{align}
 \frac{\beta_1}{\beta_2} & \approx \frac{2}{5},                    \label{eq:beta12_approx} \\
 \frac{\beta_0 + \tau_5}{\beta_1 + \tau_5} & \approx \frac{5}{7},  \label{eq:beta01_tau5_approx} \\
 \frac{\beta_2 + \tau_5}{\beta_1 + \tau_5} & \approx \frac{10}{7}, \label{eq:beta12_tau5_approx}
\end{align}
provided that $\Delta x$ is small.
Now we have from $\beta_1 \gg \beta_0$ and \eqref{eq:beta01_tau5_approx},
\begin{equation} \label{eq:omega1_JS_Z_cond1}
 \frac{\beta_1}{\beta_0} \left( \frac{\beta_1}{\beta_0} - \frac{\beta_0 + \tau_5}{\beta_1 + \tau_5} \right) \gg 1,
\end{equation}
and from \eqref{eq:beta12_approx} and \eqref{eq:beta12_tau5_approx},
\begin{equation} \label{eq:omega1_JS_Z_cond2}
 \frac{d_2}{d_0} \frac{\beta_1}{\beta_2} \left( \frac{\beta_2 + \tau_5}{\beta_1 + \tau_5} - \frac{\beta_1}{\beta_2} \right) \approx \frac{216}{175}.
\end{equation}
Then we conclude from \eqref{eq:omega1_JS_Z_cond1} and \eqref{eq:omega1_JS_Z_cond2} that
$$
   \frac{\beta_1}{\beta_0} \left( \frac{\beta_1}{\beta_0} - \frac{\beta_0 + \tau_5}{\beta_1 + \tau_5} \right) \gg \frac{d_2}{d_0} \frac{\beta_1}{\beta_2} \left( \frac{\beta_2 + \tau_5}{\beta_1 + \tau_5} - \frac{\beta_1}{\beta_2} \right)
$$
and hence
\begin{equation} \label{eq:omega1_JS_Z}
 d_0 \left( \frac{\beta_1}{\beta_0} \right)^2 + d_1 + d_2 \left( \frac{\beta_1}{\beta_2} \right)^2 \gg d_0 \frac{1 + \frac{\tau_5}{\beta_0}}{1 + \frac{\tau_5}{\beta_1}} + d_1 + d_2 \frac{1 + \frac{\tau_5}{\beta_2}}{1 + \frac{\tau_5}{\beta_1}}.
\end{equation}
Similarly, by $\beta_1 \gg \beta_0$ and \eqref{eq:beta01_tau5_approx},
\begin{equation} \label{eq:omega1_Z_cond1}
 d_0 \frac{1 + \frac{\tau_5}{\beta_0}}{1 + \frac{\tau_5}{\beta_1}} = d_0 \frac{\beta_1}{\beta_0} \frac{\beta_0 + \tau_5}{\beta_1 + \tau_5} \gg 1,
\end{equation}
and by \eqref{eq:beta12_approx} and \eqref{eq:beta12_tau5_approx},
\begin{equation} \label{eq:omega1_Z_cond2}
 d_2 \frac{1 + \frac{\tau_5}{\beta_2}}{1 + \frac{\tau_5}{\beta_1}} = d_2 \frac{\beta_1}{\beta_2} \frac{\beta_2 + \tau_5}{\beta_1 + \tau_5} \approx \frac{6}{35}.
\end{equation}
The inequality \eqref{eq:omega1_Z_cond1}, $d_1 = \frac{3}{5}$ and the approximation \eqref{eq:omega1_Z_cond2} give
\begin{equation} \label{eq:omega1_Z}
 d_0 \frac{1 + \frac{\tau_5}{\beta_0}}{1 + \frac{\tau_5}{\beta_1}} + d_1 + d_2 \frac{1 + \frac{\tau_5}{\beta_2}}{1 + \frac{\tau_5}{\beta_1}} \gg 1.
\end{equation}
Combining the inequalities \eqref{eq:omega1_JS_Z} and \eqref{eq:omega1_Z} above with \eqref{eq:weights_JS_1} and \eqref{eq:weights_Z_1}, we arrive at the conclusion that $\omega^{\JS}_1 \ll \omega^\Z_1 \ll d_1$.
A similar argument leads to the inequality $\omega^{\JS}_2 \ll \omega^\Z_2 \ll d_2$.
In summary,
$$
   \omega^{\JS}_0, \omega^Z_0 \lessapprox 1,~~\omega^{\JS}_0 > \omega^\Z_0 > d_0,~~\omega^{\JS}_1 \ll \omega^\Z_1 \ll d_1,~~\omega^{\JS}_2 \ll \omega^\Z_2 \ll d_2,
$$
matching the observations at $x=-0.005$ in Table \ref{tab:weight_comparison}.

The analysis of case (iii) is analogous to what is done in case (ii), with the following results
$$
   \omega^{\JS}_0 \ll \omega^\Z_0 \ll d_0,~~\omega^{\JS}_1 \ll \omega^\Z_1 \ll d_1,~~\omega^{\JS}_2 > \omega^\Z_2 > d_2,~~\omega^{\JS}_2, \omega^Z_2 \lessapprox 1.
$$
The results for case (iv) can be obtained in the same manner as the analysis in case (i), which gives
$$ 
   \omega^{\JS}_0, \omega^Z_0 \gtrapprox 0,~~\omega^{\JS}_0 \ll \omega^\Z_0 \ll d_0,~~\omega^{\JS}_1,~\omega^\Z_1 > d_1,~~\omega^{\JS}_2,~\omega^\Z_2 > d_2. 
$$

%
These inequalities concerning the weights $\omega^{\JS}_k,~ \omega^\Z_k$ and $d_k$, which are consistent with the numerical results of the test example \eqref{eq:weight_comparison}, are summarized in Table \ref{tab:S5_weight_comparison} below.
We will generalize this approach to the new nonlinear weights in the next section, where a very similar analysis can be applied.

\begin{table}[ht]
\renewcommand{\arraystretch}{1.2}
\centering
\begin{tabular}{c|cccc} 
\hline
$S^5$ & \multicolumn{4}{c}{$\omega^{\JS}_k,~\omega^\Z_k$ and $d_k$} \\ 
\hline 
$\{ x_{i-3}, x_{i-2}, x_{i-1}, {\color{red} x_i}, {\color{red} x_{i+1}} \}$ & $\omega^{\JS}_0,~\omega^\Z_0 > d_0$, & $\omega^{\JS}_1,~\omega^\Z_1 > d_1$, & $\omega^{\JS}_2 \ll \omega^\Z_2 \ll d_2$, & $\omega^{\JS}_2,~\omega^Z_2 \gtrapprox 0$ \\
$\{ x_{i-2}, x_{i-1}, {\color{red} x_i}, {\color{red} x_{i+1}}, x_{i+2} \}$ & $\omega^{\JS}_0 > \omega^\Z_0 > d_0$, & $\omega^{\JS}_1 \ll \omega^\Z_1 \ll d_1$, & $\omega^{\JS}_2 \ll \omega^\Z_2 \ll d_2$, & $\omega^{\JS}_0,~\omega^Z_0 \lessapprox 1$ \\
$\{ x_{i-1}, {\color{red} x_i}, {\color{red} x_{i+1}}, x_{i+2}, x_{i+3} \}$ & $\omega^{\JS}_0 \ll \omega^\Z_0 \ll d_0$, & $\omega^{\JS}_1 \ll \omega^\Z_1 \ll d_1$, & $\omega^{\JS}_2 > \omega^\Z_2 > d_2$, & $\omega^{\JS}_2,~\omega^Z_2 \lessapprox 1$ \\
$\{ {\color{red} x_i}, {\color{red} x_{i+1}}, x_{i+2}, x_{i+3}, x_{i+4} \}$ & $\omega^{\JS}_0 \ll \omega^\Z_0 \ll d_0$, & $\omega^{\JS}_1,~\omega^\Z_1 > d_1$, & $\omega^{\JS}_2,~\omega^\Z_2 > d_2$, & $\omega^{\JS}_0,~\omega^Z_0 \gtrapprox 0$ \\  
\hline
\end{tabular}
\caption{The deviation of the nonlinear weights $\omega_k$ for WENO-JS and WENO-Z from the linear weights $d_k$ when there is a discontinuity between the cell centers ${\color{red} x_i}$ and ${\color{red} x_{i+1}}$.}
\label{tab:S5_weight_comparison}
\end{table}

\section{New Z-type nonlinear weights} \label{sec:new_Z_weight}
For the nonlinear weights $\omega^\Z_k$ in \cite{Borges}, there is a dilemma between the order of accuracy at critical points in smooth regions and the numerical dissipation around discontinuities. 
On the one hand, increasing the value of $p$ in \eqref{eq:weights_Z} would recover the fifth order accuracy at first-order critical points, where the first derivative is zero, in smooth regions.
On the other hand, increasing the value of $p$ decreases the relative importance of the discontinuous substencil, leading to more dissipation around discontinuities, which implies that the approximation with $p>1$ is not as sharp as the one with $p=1$ for WENO-Z.
It is also pointed out that scaling down the numerical dissipation around discontinuities far outweighs attaining the order of accuracy at critical points in smooth regions when solving problems with shocks.
Then $p=1$ is chosen for practical use.

We introduce the new Z-type nonlinear weights $\omega_k$ so that the dilemma above could be conquered, meaning that increasing the value of $p$ can result in both the fifth order accuracy at critical points in smooth regions and the sharper approximation near discontinuities.
The key is to take the $p$th root of the smoothness indicators $\beta_k$ in \eqref{eq:smooth_indicator_JS}.
Then the novel global smoothness indicator $\tau^{\ZR}_5$ on the stencil $S^5$ is given by 
\begin{equation} \label{eq:global_smooth_indicator_ZR}
 \tau^{\ZR}_5 = \left| \sqrt[p]{\beta_0} - \sqrt[p]{\beta_2} \right|, 
\end{equation}
and the nonlinear weights, which is denoted by $\omega^{\ZR}_k$, are defined as
\begin{equation} \label{eq:weights_ZR}
 \omega^{\ZR}_k = \frac{\alpha_k}{\sum^2_{s=0} \alpha_s},~~\alpha_k = d_k \left( 1 + \left( \frac{\tau^{\ZR}_5}{\sqrt[p]{\beta_k} + \epsilon} \right)^p \right),~~k=0,1,2.
\end{equation}
Note that for $p=1$, the nonlinear weights $\omega^{\ZR}_k$ in \eqref{eq:weights_ZR} coincide with $\omega^\Z_k$ in \eqref{eq:weights_Z}.
We define $T_5: (0, \infty) \to \R$ by
\begin{equation} \label{eq:global_smooth_indicator_ZR_p}
 T_5(p) = \left( \tau^{\ZR}_5 \right)^p = \left| \sqrt[p]{\beta_0} - \sqrt[p]{\beta_2} \right|^p 
\end{equation}
for the analysis in the rest of the section, where $\epsilon$ is taken to be $0$.
If $\beta_0 \ne \beta_2$, by Corollary \ref{cor:phi} in the Appendix, $T_5$ is strictly decreasing on $(0, \infty)$ and
\begin{align}
 & \lim_{p \to 0^+} T_5(p) = \max \{ \beta_0, \beta_2 \}, \label{eq:T5_limit_0} \\
 & \lim_{p \to \infty} T_5(p) = 0, \label{eq:T5_limit_infty} \\
 & 0 < T_5(p) < \max \{ \beta_0, \beta_2 \}. \label{eq:T5_range}
\end{align}
See Appendix for the proof in detail.
Otherwise $T_5(p) = 0$ for all $p>0$.
Note that, in the limit $p \to \infty$, we have from \eqref{eq:T5_limit_infty} 
$$
   \lim_{p \to \infty} \alpha_k = \lim_{p \to \infty} d_k \left( 1 + \frac{T_5(p)}{\beta_k} \right) = d_k,
$$
and hence
\begin{equation} \label{eq:weights_ZR_infty}
 \lim_{p \to \infty} \omega^{\ZR}_k = \lim_{p \to \infty} \frac{\alpha_k}{\sum^2_{s=0} \alpha_s} = d_k.
\end{equation}
Accordingly, $\omega^{\ZR}_k \to d_k$ as $p \to \infty$, which returns to the classical finite difference scheme with Gibbs oscillations.

\subsection{Spatial fifth order accuracy in smooth regions for $\omega^{\ZR}_k$} \label{sec:new_Z_weight_smooth}
As in Section \ref{sec:1d_fd_WENO_smooth}, we assume the smooth solution over the stencil $S^5$. 
We want to explore the convergence of nonlinear weights $\omega^{\ZR}_k$ to linear weights $d_k$ in the smooth region.

We first consider the case where there is no critical point at $x = x_i$, that is, $f'_i \ne 0$.
The $p$th root of each smoothness indicator, $\sqrt[p]{\beta_k}$, in Taylor series at $x = x_i$ is
\begin{equation} \label{eq:smooth_indicator_ZR_Taylor}
\begin{aligned}
 \sqrt[p]{\beta_0} = f'^{2/p}_i \Delta x^{2/p} + \left( \frac{13}{12p} f''^2_i - \frac{2}{3p} f'_i f'''_i \right) f'^{2/p-2}_i \Delta x^{2+2/p} &- \left( \frac{13}{6p} f''_i f'''_i - \frac{1}{2p} f'_i f^{(4)}_i \right) f'^{2/p-2}_i \Delta x^{3+2/p} \\
  &+ O \left( \Delta x^{4+2/p} \right), \\
 \sqrt[p]{\beta_1} = f'^{2/p}_i \Delta x^{2/p} + \left( \frac{13}{12p} f''^2_i + \frac{1}{3p} f'_i f'''_i \right) f'^{2/p-2}_i \Delta x^{2+2/p} &+ O \left( \Delta x^{4+2/p} \right), \\
 \sqrt[p]{\beta_2} = f'^{2/p}_i \Delta x^{2/p} + \left( \frac{13}{12p} f''^2_i - \frac{2}{3p} f'_i f'''_i \right) f'^{2/p-2}_i \Delta x^{2+2/p} &+ \left( \frac{13}{6p} f''_i f'''_i - \frac{1}{2p} f'_i f^{(4)}_i \right) f'^{2/p-2}_i \Delta x^{3+2/p} \\
 &+ O \left( \Delta x^{4+2/p} \right).
\end{aligned}
\end{equation}
Then the global smoothness indicator $\tau^{\ZR}_5$ in \eqref{eq:global_smooth_indicator_ZR} is
\begin{equation} \label{eq:global_smooth_indicator_ZR_Taylor}
 \tau^{\ZR}_5 = \left| \left( \frac{13}{3p} f''_i f'''_i - \frac{1}{p} f'_i f^{(4)}_i \right) f'^{2/p-2}_i \right| \Delta x^{3+2/p} + O \left( \Delta x^{4+2/p} \right).
\end{equation}
With $\epsilon = 0$, we have from \eqref{eq:smooth_indicator_ZR_Taylor} and \eqref{eq:global_smooth_indicator_ZR_Taylor}
$$
   \frac{\tau^{\ZR}_5}{\sqrt[p]{\beta_k}} = O(\Delta x^3),
$$
and by \eqref{eq:weights_ZR},
\begin{equation} \label{eq:ZR_condition}
 \omega^{\ZR}_k = d_k + O(\Delta x^{3p}).
\end{equation}
Thus the sufficient condition \eqref{eq:WENO_condition} to achieve the spatial fifth order accuracy is clearly satisfied for $p \geqslant 1$, provided that the critical point does not exist at $x = x_i$.

In the case where $f'_i = 0$ but $f''_i \ne 0$, we cannot use \eqref{eq:smooth_indicator_ZR_Taylor} directly. 
Instead the Taylor expansions of $\sqrt[p]{\beta_k}$ at $x = x_i$ take the form
\begin{equation} \label{eq:smooth_indicator_ZR_Taylor_critical}
\begin{aligned}
 \sqrt[p]{\beta_0} &= \left( \frac{13}{12} f''^2_i \right)^{1/p} \Delta x^{4/p} - \left( \frac{13}{12} \right)^{1/p} \frac{2}{p} f'''_i f''^{2/p-1}_i \Delta x^{1+4/p} + O \left( \Delta x^{2+4/p} \right), \\
 \sqrt[p]{\beta_1} &= \left( \frac{13}{12} f''^2_i \right)^{1/p} \Delta x^{4/p} + O \left( \Delta x^{2+4/p} \right), \\
 \sqrt[p]{\beta_2} &= \left( \frac{13}{12} f''^2_i \right)^{1/p} \Delta x^{4/p} + \left( \frac{13}{12} \right)^{1/p} \frac{2}{p} f'''_i f''^{2/p-1}_i \Delta x^{1+4/p} + O \left( \Delta x^{2+4/p} \right).
\end{aligned}
\end{equation}
Then \eqref{eq:global_smooth_indicator_ZR} gives
$$
   \tau_5 = \left( \frac{13}{12} \right)^{1/p} \frac{4}{p} \left| f'''_i f''^{2/p-1}_i \right| \Delta x^{1+4/p} + O \left( \Delta x^{3+4/p} \right).
$$
By the technique of Taylor expansion again, we obtain
$$
   \omega^{\ZR}_k = d_k + O(\Delta x^p).
$$
So the nonlinear weights $\omega^{\ZR}_k$ with $p \geqslant 3$ satisfies the condition \eqref{eq:WENO_condition} to maintain fifth order accuracy in space if there is a first-order critical point at $x = x_i$.
Note that we could apply the similar analysis to the smooth region containing a higher-order critical point, which demonstrates that increasing the value of $p$ suffices for the condition \eqref{eq:WENO_condition}, attaining fifth order accuracy. 

Finally, we investigate the limit case of nonlinear weights $\omega^{\ZR}_k$ in smooth regions as $p \to 0^+$.
Combining \eqref{eq:smooth_indicator_JS_Taylor} and \eqref{eq:T5_limit_0} shows that
$$
   \lim_{\Delta x \to 0} \lim_{p \to 0^+} \frac{T_5(p)}{\beta_k} = \lim_{\Delta x \to 0} \frac{\max \{ \beta_0, \beta_2 \}}{\beta_k} = 1
$$
so that
$$
   \lim_{\Delta x \to 0} \lim_{p \to 0^+} \omega^{\ZR}_k = \lim_{\Delta x \to 0} \lim_{p \to 0^+} \frac{d_k \left( 1 + \frac{T_5(p)}{\beta_k} \right)}{\sum^2_{s=0} d_s \left( 1 + \frac{T_5(p)}{\beta_s} \right)} = \frac{d_k \left( 1 + 1 \right)}{\sum^2_{s=0} d_s \left( 1 + 1 \right)} = d_k.
$$
Consequently, if $\Delta x$ is small enough, $\omega^{\ZR}_k \approx d_k$ as $p \to 0^+$.

\subsection{Dissipation around discontinuities for $\omega^{\ZR}_k$} \label{sec:new_Z_weight_discontinuity}
The tool developed in Section \ref{sec:1d_fd_WENO_discontinuity} is effective to better understand the behavior of nonlinear weights, which are related to the dissipation around discontinuities.
We can apply this same approach to examine the performance of the proposed nonlinear weights $\omega^{\ZR}_k$ in \eqref{eq:weights_ZR} on the stencil $S^5$ with the discontinuity as $p$ increases.
The assumption remains the same as in Section \ref{sec:1d_fd_WENO_discontinuity}.
From \eqref{eq:weights_ZR}, we have
\begin{align}
 \omega^{\ZR}_0(p) &= \frac{1}{d_0 + d_1 \frac{1 + \frac{T_5(p)}{\beta_1}}{1 + \frac{T_5(p)}{\beta_0}} + d_2 \frac{1 + \frac{T_5(p)}{\beta_2}}{1 + \frac{T_5(p)}{\beta_0}}} d_0, \label{eq:weights_ZR_0} \\
 \omega^{\ZR}_1(p) &= \frac{1}{d_0 \frac{1 + \frac{T_5(p)}{\beta_0}}{1 + \frac{T_5(p)}{\beta_1}} + d_1 + d_2 \frac{1 + \frac{T_5(p)}{\beta_2}}{1 + \frac{T_5(p)}{\beta_1}}} d_1, \label{eq:weights_ZR_1} \\
 \omega^{\ZR}_2(p) &= \frac{1}{d_0 \frac{1 + \frac{T_5(p)}{\beta_0}}{1 + \frac{T_5(p)}{\beta_2}} + d_1 \frac{1 + \frac{T_5(p)}{\beta_1}}{1 + \frac{T_5(p)}{\beta_2}} + d_2} d_2. \label{eq:weights_ZR_2}
\end{align}
Suppose $p>q>0$.
According to \eqref{eq:T5_range} and the fact that $T_5$ is strictly decreasing on $(0, \infty)$, the inequality 
\begin{equation} \label{eq:pq}
 0 < T_5(p) < T_5(q) < \max \{ \beta_0, \beta_2 \} 
\end{equation}
holds.

In the case (i), we have $\beta_0, \beta_1 = O(\Delta x^2)$ and $\beta_2 = O(1)$. 
Then for $k = 0,1$, 
\begin{gather}
 \beta_k \ll \beta_2, \label{eq:betak2_ll} \\
 1 + \frac{T_5(p)}{\beta_k} > 1 + \frac{T_5(p)}{\beta_2}. \label{eq:betak2_grt}
\end{gather}
Combining $\displaystyle{\lim_{\Delta x \to 0} \beta_1 / \beta_0 = 1}$ and \eqref{eq:betak2_ll} gives
$$
   \frac{1}{\beta_1} - \frac{1}{\beta_0} \ll \frac{1}{\beta_0} - \frac{1}{\beta_2},
$$
and it is natural to assume that 
\begin{equation} \label{eq:beta1002}
 \frac{1}{\beta_1} - \frac{1}{\beta_0} < \frac{d_2}{d_1} \left( \frac{1}{\beta_0} - \frac{1}{\beta_2} \right).
\end{equation}
It follows that 
$$
   d_1 \left( \frac{1 + \frac{T_5(p)}{\beta_1}}{1 + \frac{T_5(p)}{\beta_0}} - 1 \right) 
 = d_1 \frac{T_5(p) \left( \frac{1}{\beta_1} - \frac{1}{\beta_0} \right)}{1 + \frac{T_5(p)}{\beta_0}} 
 < d_2 \frac{T_5(p) \left( \frac{1}{\beta_0} - \frac{1}{\beta_2} \right)}{1 + \frac{T_5(p)}{\beta_0}}
 = d_2 \left( 1 - \frac{1 + \frac{T_5(p)}{\beta_2}}{1 + \frac{T_5(p)}{\beta_0}} \right), 
$$
and hence
\begin{equation} \label{eq:t5p_1}
 d_0 + d_1 \frac{1 + \frac{T_5(p)}{\beta_1}}{1 + \frac{T_5(p)}{\beta_0}} + d_2 \frac{1 + \frac{T_5(p)}{\beta_2}}{1 + \frac{T_5(p)}{\beta_0}} < d_0 + d_1 + d_2 = 1.
\end{equation}
Multiplying both sides of \eqref{eq:beta1002} by $T_5(q) - T_5(p)$, along with \eqref{eq:pq}, leads to
$$
   \left( T_5(q) - T_5(p) \right) \left( \frac{1}{\beta_1} - \frac{1}{\beta_0} \right) < \frac{d_2}{d_1} \left( T_5(q) - T_5(p) \right) \left( \frac{1}{\beta_0} - \frac{1}{\beta_2} \right).
$$
After quite extensive algebra, we obtain 
\begin{equation} \label{eq:t5q_t5p}
 d_0 + d_1 \frac{1 + \frac{T_5(q)}{\beta_1}}{1 + \frac{T_5(q)}{\beta_0}} + d_2 \frac{1 + \frac{T_5(q)}{\beta_2}}{1 + \frac{T_5(q)}{\beta_0}}
<d_0 + d_1 \frac{1 + \frac{T_5(p)}{\beta_1}}{1 + \frac{T_5(p)}{\beta_0}} + d_2 \frac{1 + \frac{T_5(p)}{\beta_2}}{1 + \frac{T_5(p)}{\beta_0}}.
\end{equation}
From \eqref{eq:weights_ZR_0}, \eqref{eq:t5p_1} and \eqref{eq:t5q_t5p}, we have $\omega^{\ZR}_0(q) > \omega^{\ZR}_0(p) > d_0$.
Similarly, $\omega^{\ZR}_1(q) > \omega^{\ZR}_1(p) > d_1$.
According to \eqref{eq:pq} and \eqref{eq:betak2_ll}, we find, for $k = 0,1$,
$$
   \frac{\beta_k - \beta_2}{\beta_0 + T_5(q)} > \frac{\beta_k - \beta_2}{\beta_0 + T_5(p)},
$$
which implies
$$
   \frac{\beta_k + T_5(q)}{\beta_2 + T_5(q)} = 1 + \frac{\beta_k - \beta_2}{\beta_2 + T_5(q)} > 1 + \frac{\beta_k - \beta_2}{\beta_2 + T_5(p)} = \frac{\beta_k + T_5(p)}{\beta_2 + T_5(p)}.
$$
Multiplication by $\frac{\beta_2}{\beta_k}$ gives
$$
   \frac{1 + \frac{T_5(q)}{\beta_k}}{1 + \frac{T_5(q)}{\beta_2}} =  \frac{\beta_2}{\beta_k} \frac{\beta_k + T_5(q)}{\beta_2 + T_5(q)} 
 > \frac{\beta_2}{\beta_k} \frac{\beta_k + T_5(p)}{\beta_2 + T_5(p)} = \frac{1 + \frac{T_5(p)}{\beta_k}}{1 + \frac{T_5(p)}{\beta_2}}.
$$
We thus have 
$$
   d_0 \frac{1 + \frac{T_5(q)}{\beta_0}}{1 + \frac{T_5(q)}{\beta_2}} + d_1 \frac{1 + \frac{T_5(q)}{\beta_1}}{1 + \frac{T_5(q)}{\beta_2}} + d_2 
 > d_0 \frac{1 + \frac{T_5(p)}{\beta_0}}{1 + \frac{T_5(p)}{\beta_2}} + d_1 \frac{1 + \frac{T_5(p)}{\beta_1}}{1 + \frac{T_5(p)}{\beta_2}} + d_2
 > d_0 + d_1 + d_2 = 1,
$$
where the penultimate inequality holds from \eqref{eq:betak2_grt}.
By \eqref{eq:weights_ZR_2}, we obtain the relation $\omega^{\ZR}_2(q) < \omega^{\ZR}_2(p) < d_2$.
Therefore, 
$$
   \omega^{\ZR}_0(q) > \omega^{\ZR}_0(p) > d_0,~~\omega^{\ZR}_1(q) > \omega^{\ZR}_1(p) > d_1,~~\omega^{\ZR}_2(q) < \omega^{\ZR}_2(p) < d_2.
$$

For the case (ii), employing \eqref{eq:beta12_approx} and $\beta_0 \ll \beta_k, k = 1,2$, gives
$$
   \frac{1}{\beta_0} - \frac{1}{\beta_k} \gg \left( \frac{1}{\beta_1} - \frac{1}{\beta_2} \right).
$$
We can derive 
$$
   \omega^{\ZR}_0(q) > \omega^{\ZR}_0(p) > d_0,~~\omega^{\ZR}_1(q) < \omega^{\ZR}_1(p) < d_1,~~\omega^{\ZR}_2(q) < \omega^{\ZR}_2(p) < d_2
$$
in much the same way we did in the case (i) under the assumptions 
$$
   \frac{1}{\beta_0} - \frac{1}{\beta_1} < \frac{d_2}{d_0} \left( \frac{1}{\beta_1} - \frac{1}{\beta_2} \right) \text{ and } \,
   \frac{1}{\beta_0} - \frac{1}{\beta_2} < \frac{d_1}{d_0} \left( \frac{1}{\beta_1} - \frac{1}{\beta_2} \right).
$$
The same techniques would naturally lead us to the conclusions
$$
   \omega^{\ZR}_0(q) < \omega^{\ZR}_0(p) < d_0,~~\omega^{\ZR}_1(q) < \omega^{\ZR}_1(p) < d_1,~~\omega^{\ZR}_2(q) > \omega^{\ZR}_2(p) > d_2,
$$
in the case (iii), and 
$$
   \omega^{\ZR}_0(q) < \omega^{\ZR}_0(p) < d_0,~~\omega^{\ZR}_1(q) > \omega^{\ZR}_1(p) > d_1,~~\omega^{\ZR}_2(q) > \omega^{\ZR}_2(p) > d_2,
$$
in the case (iv).
%
Therefore, we show that each nonlinear weight $\omega^{\ZR}_k$ moves closer to the corresponding linear weight $d_k$ as $p$ increases, as summarized in Table \ref{tab:S5_weight_comparison_ZR}.
In the limit $p \to \infty$, the nonlinear weight $\omega^{\ZR}_k$ converges to $d_k$ as shown in \eqref{eq:weights_ZR_infty}.

\begin{table}[h!]
\renewcommand{\arraystretch}{1.2}
\centering
\begin{tabular}{c|c} 
\hline
$S^5$ & $\omega^{\ZR}_k(p),~\omega^\Z_k(q)$ and $d_k$ \\ 
\hline 
$\{ x_{i-3}, x_{i-2}, x_{i-1}, {\color{red} x_i}, {\color{red} x_{i+1}} \}$ & $\omega^{\ZR}_0(q) > \omega^{\ZR}_0(p) > d_0,~~\omega^{\ZR}_1(q) > \omega^{\ZR}_1(p) > d_1,~~\omega^{\ZR}_2(q) < \omega^{\ZR}_2(p) < d_2$ \\
$\{ x_{i-2}, x_{i-1}, {\color{red} x_i}, {\color{red} x_{i+1}}, x_{i+2} \}$ & $\omega^{\ZR}_0(q) > \omega^{\ZR}_0(p) > d_0,~~\omega^{\ZR}_1(q) < \omega^{\ZR}_1(p) < d_1,~~\omega^{\ZR}_2(q) < \omega^{\ZR}_2(p) < d_2$ \\
$\{ x_{i-1}, {\color{red} x_i}, {\color{red} x_{i+1}}, x_{i+2}, x_{i+3} \}$ & $\omega^{\ZR}_0(q) < \omega^{\ZR}_0(p) < d_0,~~\omega^{\ZR}_1(q) < \omega^{\ZR}_1(p) < d_1,~~\omega^{\ZR}_2(q) > \omega^{\ZR}_2(p) > d_2$ \\
$\{ {\color{red} x_i}, {\color{red} x_{i+1}}, x_{i+2}, x_{i+3}, x_{i+4} \}$ & $\omega^{\ZR}_0(q) < \omega^{\ZR}_0(p) < d_0,~~\omega^{\ZR}_1(q) > \omega^{\ZR}_1(p) > d_1,~~\omega^{\ZR}_2(q) > \omega^{\ZR}_2(p) > d_2$ \\  
\hline
\end{tabular}
\caption{The deviation of the nonlinear weights $\omega^{\ZR}_k$ with $p>q>0$ from the linear weights $d_k$ when there is a discontinuity between the cell centers ${\color{red} x_i}$ and ${\color{red} x_{i+1}}$.}
\label{tab:S5_weight_comparison_ZR}
\end{table}

We use the example \eqref{eq:weight_comparison} of the linear advection equation, where the initial condition contains a discontinuity, again to compare the nonlinear weights $\omega^{\ZR}_k$ ($\epsilon = 10^{-40}$) for $p=1,3$ and $6$ at the first step of the numerical solutions.
Note that the weights $\omega^{\ZR}_k$ with $p=1$ is equivalent to $\omega^\Z_k$ in \eqref{eq:weights_Z}.
Figure \ref{fig:weight_comparison_ZR} and Table \ref{tab:weight_comparison_ZR} confirm that the increase in $p$ narrows the difference between nonlinear and linear weights around the discontinuity.

\begin{figure}[h!]
\centering
 \begin{subfigure}[b]{0.32\textwidth}
 \centering
 \includegraphics[width=\textwidth]{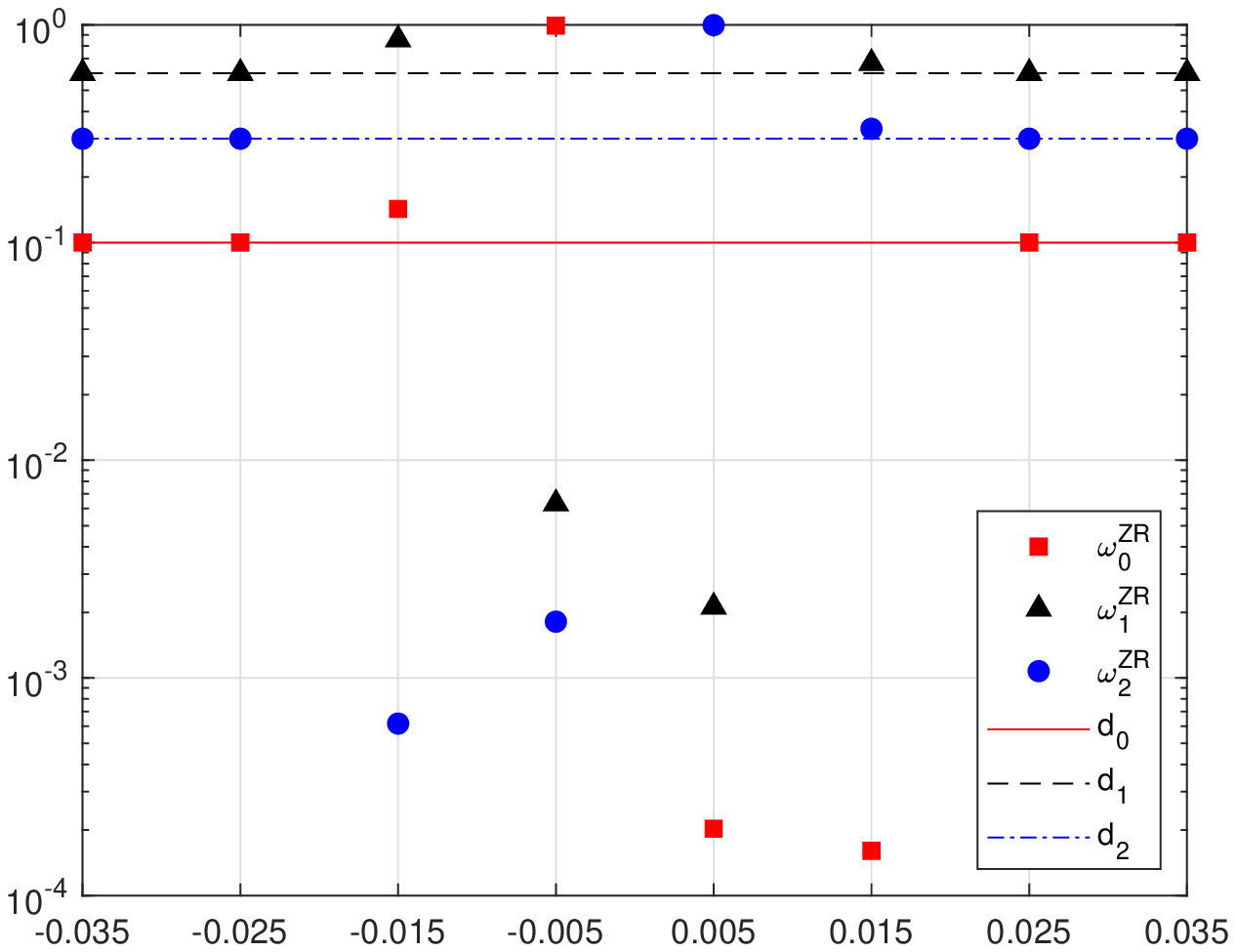}
 \caption{$p=1$}
 \end{subfigure}
 \hfill
 \begin{subfigure}[b]{0.32\textwidth}
 \centering
 \includegraphics[width=\textwidth]{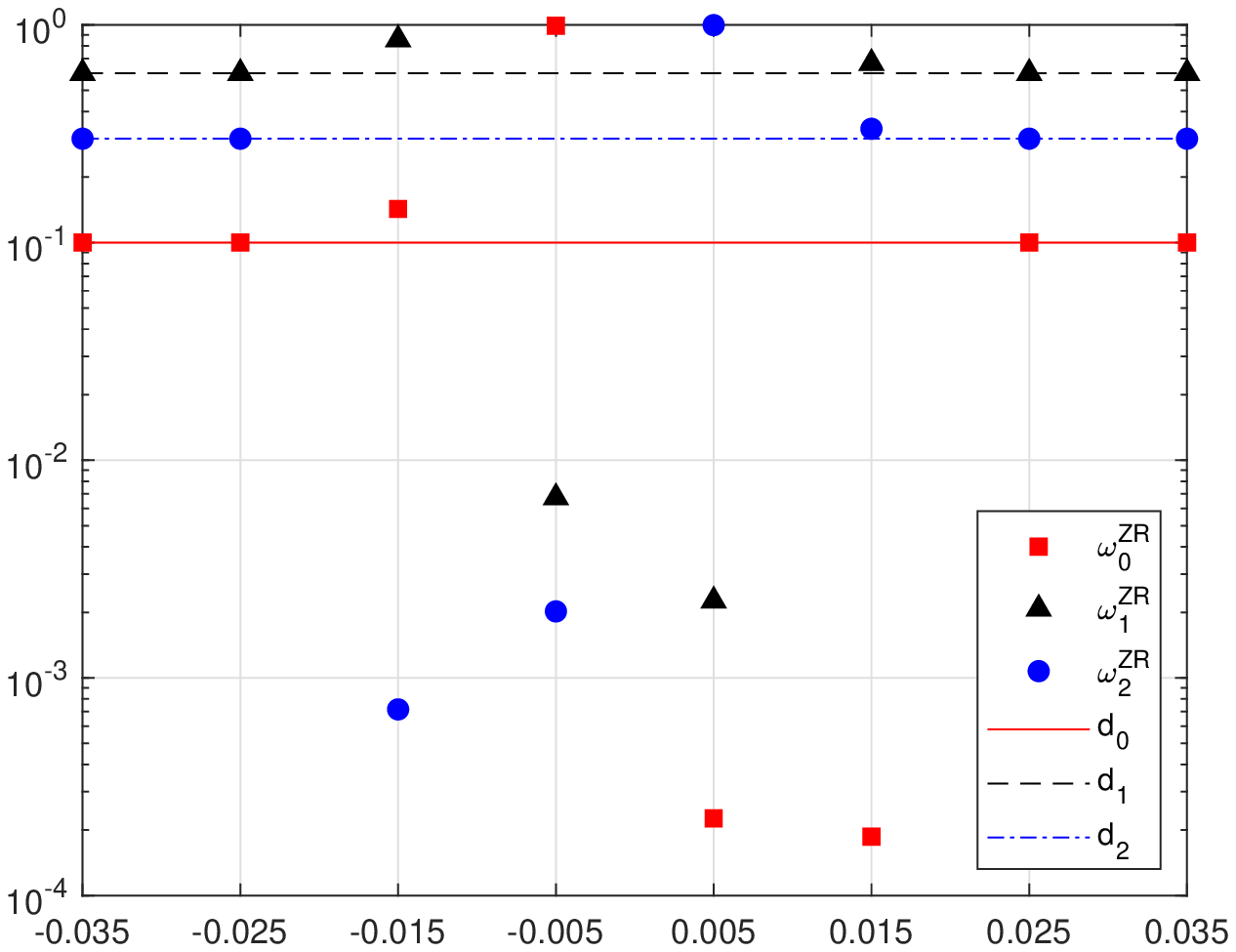}
 \caption{$p=3$}
 \end{subfigure}
 \hfill
 \begin{subfigure}[b]{0.32\textwidth}
 \centering
 \includegraphics[width=\textwidth]{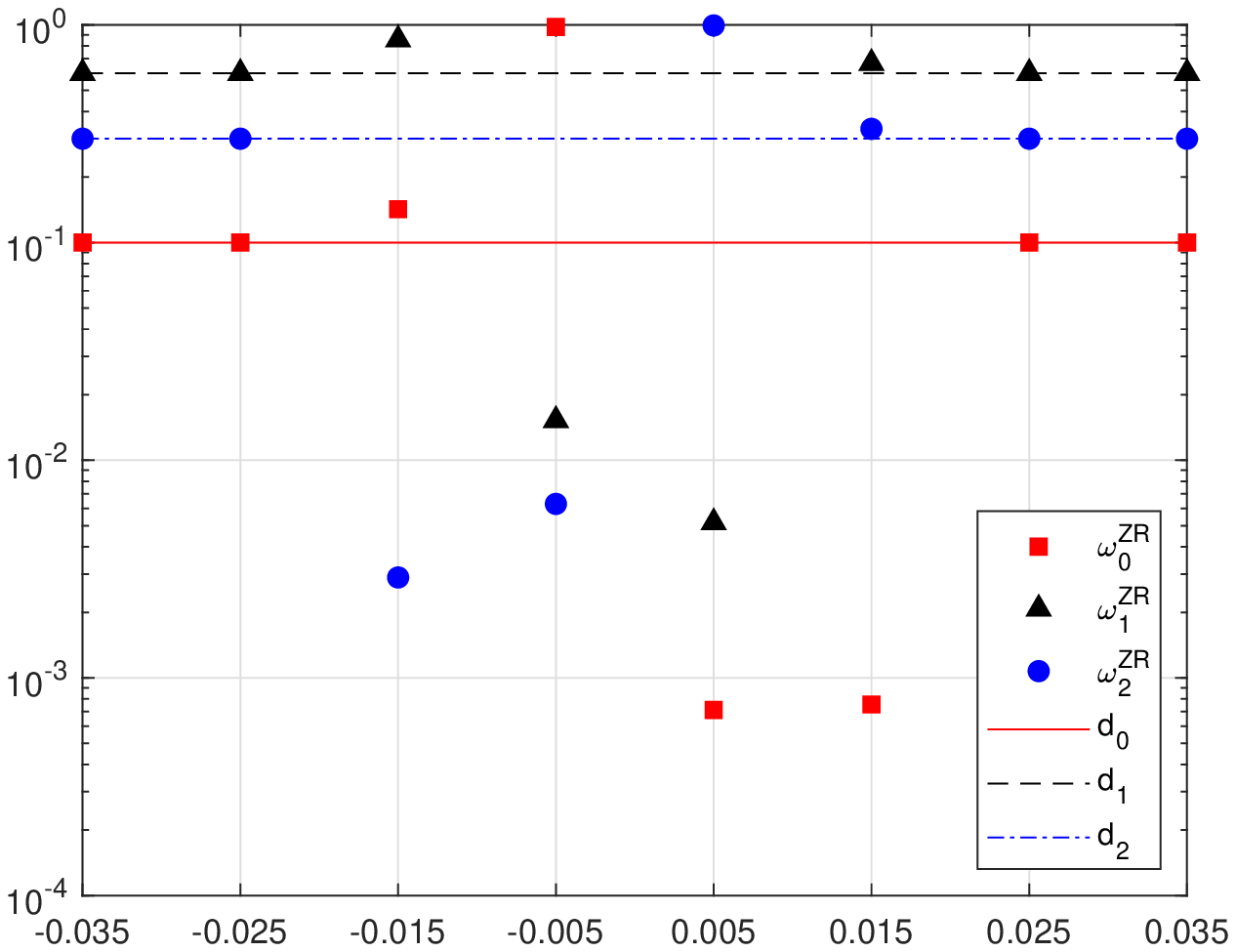}
 \caption{$p=6$}
 \end{subfigure}
\caption{The distribution of the nonlinear weights $\omega^{\ZR}_k$ near the discontinuity at $x=0$, compared with the linear weights $d_k$, for $p=1,3$ and $6$ at the first step of the numerical solutions of the test example \eqref{eq:weight_comparison} on a semilog plot with a $\log_{10}$ scale on the vertical axis. 
The linear weights $d_k$ are shown in lines while the nonlinear weights $\omega_k$ symbols.}
\label{fig:weight_comparison_ZR}
\end{figure}

\begin{table}[h!]
\renewcommand{\arraystretch}{1.2}
\centering
\begin{tabular}{ccccccccc|c} 
\hline
$x$ & -0.035 & -0.025 & -0.015 & -0.005 & 0.005 & 0.015 & 0.025 & 0.035 &  \\ 
\hline 
$\omega^{\ZR}_0$ ($p=1$) & 0.100000 & 0.100000 & 0.142660 & 0.991870 & 2.027e-4 & 1.604e-4 & 0.100000 & 0.100000 & \multirow{3}{*}{$d_0 = 0.1$} \\  
$\omega^{\ZR}_0$ ($p=3$) & 0.100000 & 0.100000 & 0.142646 & 0.991246 & 2.262e-4 & 1.864e-4 & 0.100000 & 0.100000 & \\  
$\omega^{\ZR}_0$ ($p=6$) & 0.100000 & 0.100000 & 0.142335 & 0.978451 & 7.119e-4 & 7.543e-4 & 0.100000 & 0.100000 & \\
\hline 
$\omega^{\ZR}_1$ ($p=1$) & 0.600000 & 0.600000 & 0.856724 & 6.318e-3 & 2.120e-3 & 0.666758 & 0.600000 & 0.600000 & \multirow{3}{*}{$d_1 = 0.6$} \\  
$\omega^{\ZR}_1$ ($p=3$) & 0.600000 & 0.600000 & 0.856638 & 6.734e-3 & 2.261e-3 & 0.666741 & 0.600000 & 0.600000 & \\  
$\omega^{\ZR}_1$ ($p=6$) & 0.600000 & 0.600000 & 0.854772 & 1.525e-2 & 5.171e-3 & 0.666361 & 0.600000 & 0.600000 & \\
\hline 
$\omega^{\ZR}_2$ ($p=1$) & 0.300000 & 0.300000 & 6.166e-4 & 1.812e-3 & 0.997677 & 0.333082 & 0.300000 & 0.300000 & \multirow{3}{*}{$d_2 = 0.3$} \\  
$\omega^{\ZR}_2$ ($p=3$) & 0.300000 & 0.300000 & 7.164e-4 & 2.021e-3 & 0.997513 & 0.333073 & 0.300000 & 0.300000 & \\  
$\omega^{\ZR}_2$ ($p=6$) & 0.300000 & 0.300000 & 2.892e-3 & 6.299e-3 & 0.994117 & 0.332885 & 0.300000 & 0.300000 & \\  
\hline
\end{tabular}
\caption{The values of the nonlinear weights $\omega^{\ZR}_k$ near the discontinuity at $x=0$, compared with the linear weights $d_k$, for $p=1,3$ and $6$ at the first step of the numerical solutions of the test example \eqref{eq:weight_comparison}.}
\label{tab:weight_comparison_ZR}
\end{table}

\section{Numerical experiments} \label{sec:numex}
In this section, we present some numerical experiments to compare the proposed WENO scheme, referred as WENO-ZR, with the WENO-JS, WENO-M and WENO-Z schemes.
We take $p=3$ here for the nonlinear weights $\omega^{\ZR}_k$ in \eqref{eq:weights_ZR} as the sufficient condition \eqref{eq:WENO_condition} is satisfied for at least non-critical and first-order critical points. 
We use the 1D linear advection equation to verify the order of accuracy of the WENO schemes in terms of $L^1, L^2$ and $L^{\infty}$ error norms:
\begin{align*}
 & L^1 = \frac{1}{N+1} \sum_{i=0}^N \left| u(x_i, T) - u_i(T) \right|, \\
 & L^2 = \sqrt{\frac{1}{N+1} \sum_{i=0}^N \left( u(x_i, T) - u_i(T) \right)^2}, \\
 & L^{\infty} = \max_{0 \leqslant i \leqslant N} \left| u(x_i, T)- u_i(T) \right|,
\end{align*}
where $u(x_i, T)$ is the exact solution and $u_i(T)$ denotes the numerical approximation at the final time $t = T$.
The rest examples show the numerical results from WENO-ZR, in comparison with WENO-JS, WENO-M and WENO-Z.
We choose $\epsilon = 10^{-40}$ for the WENO-M, WENO-Z and WENO-ZR schemes whereas $\epsilon = 10^{-6}$ remains for WENO-JS as in \cite{Jiang}.
For 1D scalar problems, we use the Lax-Friedrich flux splitting. 
For 1D system problems, we apply the characteristic-wise Lax-Friedrich flux splitting. 
For 2D system problems, we implement the schemes with the characteristic-wise Lax-Friedrich flux splitting in a dimension-by-dimension fashion. 
We employ the explicit third-order total variation diminishing Runge-Kutta method in \cite{ShuOsherI} for time discretization
\begin{align*}
 u^{(1)} &= u^n + \Delta t L(u^n), \\
 u^{(2)} &= \frac{3}{4} u^n + \frac{1}{4} u^{(1)} + \frac{1}{4} \Delta t L \left( u^{(1)} \right), \\
 u^{n+1} &= \frac{1}{3} u^n + \frac{2}{3} u^{(2)} + \frac{2}{3} \Delta t L \left( u^{(2)} \right),
\end{align*}
where $L$ is the spatial operator.

\subsection{1D linear advection equation} 
\begin{example} \label{ex:advection_1d}
We start with checking the order of accuracy for the 1D linear advection equation
$$
   u_t + u_x = 0,~-1 \leqslant x \leqslant 1,
$$
with the initial condition 
$$
   u(x,0) = \sin(\pi x)
$$
and the periodic boundary condition. 
The exact solution is given by
$$
   u(x,t)= \sin \left( \pi (x-t) \right).
$$
The numerical solution is computed up to the time $T=2$ with the time step $\Delta t = 0.4 \Delta x^{5/3}$. 

The $L_1, L_2$ and $L_{\infty}$ errors versus $N$, as well as the order of accuracy, for the WENO-JS, WENO-M, WENO-Z and WENO-ZR schemes are listed in Tables \ref{tab:advection_1d_L1}, \ref{tab:advection_1d_L2} and \ref{tab:advection_1d_Linf}. 
It is clear that the expected order of accuracy is achieved for all schemes since the exact solution is smooth.
The WENO-ZR performs slightly better than the other WENO schemes in terms of accuracy.  
\end{example}

\begin{table}[h!]
\renewcommand{\arraystretch}{1.1}
\scriptsize
\centering
\caption{$L_1$ error and order of accuracy for Example \ref{ex:advection_1d}.}      
\begin{tabular}{clcrlcrlcrlc} 
\hline  
N & \multicolumn{2}{l}{WENO-JS} & & \multicolumn{2}{l}{WENO-M} & & \multicolumn{2}{l}{WENO-Z} & & \multicolumn{2}{l}{WENO-ZR} \\ 
    \cline{2-3}                     \cline{5-6}                    \cline{8-9}                    \cline{11-12}  
  & Error & Order               & & Error & Order              & & Error & Order              & & Error & Order \\
\hline 
10  & 2.81e-2 & --     & & 8.57e-3  & --     & & 7.40e-3  & --     & & 5.94e-3  & -- \\  
20  & 1.44e-3 & 4.2923 & & 2.06e-4  & 5.3785 & & 2.09e-4  & 5.1461 & & 1.97e-4  & 4.9135 \\  
40  & 4.39e-5 & 5.0301 & & 6.32e-6  & 5.0252 & & 6.33e-6  & 5.0461 & & 6.31e-6  & 4.9646 \\
80  & 1.38e-6 & 4.9897 & & 2.00e-7  & 4.9832 & & 2.00e-7  & 4.9841 & & 2.00e-7  & 4.9799 \\ 
160 & 4.34e-8 & 4.9938 & & 6.29e-9  & 4.9911 & & 6.29e-9  & 4.9912 & & 6.29e-9  & 4.9909 \\ 
320 & 1.36e-9 & 4.9975 & & 1.98e-10 & 4.9892 & & 1.98e-10 & 4.9892 & & 1.98e-10 & 4.9892 \\  
\hline
\end{tabular}
\label{tab:advection_1d_L1}
\end{table}

\begin{table}[h!]
\renewcommand{\arraystretch}{1.1}
\scriptsize
\centering
\caption{$L_2$ error and order of accuracy for Example \ref{ex:advection_1d}.}      
\begin{tabular}{clcrlcrlcrlc} 
\hline  
N & \multicolumn{2}{l}{WENO-JS} & & \multicolumn{2}{l}{WENO-M} & & \multicolumn{2}{l}{WENO-Z} & & \multicolumn{2}{l}{WENO-ZR} \\ 
    \cline{2-3}                     \cline{5-6}                    \cline{8-9}                    \cline{11-12}  
  & Error & Order               & & Error & Order              & & Error & Order              & & Error & Order \\
\hline 
10  & 3.05e-2 & --     & & 9.19e-3  & --     & & 8.12e-3  & --     & & 6.60e-3  & -- \\  
20  & 1.64e-3 & 4.2147 & & 2.29e-4  & 5.3257 & & 2.41e-4  & 5.0710 & & 2.21e-4  & 4.8986 \\  
40  & 5.19e-5 & 4.9851 & & 7.09e-6  & 5.0144 & & 7.21e-6  & 5.0663 & & 7.07e-6  & 4.9677 \\
80  & 1.59e-6 & 5.0263 & & 2.23e-7  & 4.9886 & & 2.24e-7  & 5.0087 & & 2.23e-7  & 4.9856 \\ 
160 & 4.91e-8 & 5.0195 & & 7.00e-9  & 4.9946 & & 7.01e-9  & 4.9979 & & 7.00e-9  & 4.9944 \\
320 & 1.53e-9 & 5.0065 & & 2.20e-10 & 4.9912 & & 2.20e-10 & 4.9916 & & 2.20e-10 & 4.9912 \\   
\hline
\end{tabular}
\label{tab:advection_1d_L2}
\end{table}

\begin{table}[h!]
\renewcommand{\arraystretch}{1.1}
\scriptsize
\centering
\caption{$L_{\infty}$ error and order of accuracy for Example \ref{ex:advection_1d}.}      
\begin{tabular}{clcrlcrlcrlc} 
\hline  
N & \multicolumn{2}{l}{WENO-JS} & & \multicolumn{2}{l}{WENO-M} & & \multicolumn{2}{l}{WENO-Z} & & \multicolumn{2}{l}{WENO-ZR} \\ 
    \cline{2-3}                     \cline{5-6}                    \cline{8-9}                    \cline{11-12}  
  & Error & Order               & & Error & Order              & & Error & Order              & & Error & Order \\
\hline 
10  & 4.73e-2 & --     & & 1.25e-2  & --     & & 1.12e-2  & --     & & 9.71e-3  & -- \\  
20  & 2.58e-3 & 4.1952 & & 3.20e-4  & 5.2867 & & 3.47e-4  & 5.0151 & & 3.18e-4  & 4.9335 \\ 
40  & 9.00e-5 & 4.8409 & & 1.01e-5  & 4.9815 & & 1.03e-5  & 5.0688 & & 1.01e-5  & 4.9743 \\
80  & 2.79e-6 & 5.0116 & & 3.18e-7  & 4.9933 & & 3.19e-7  & 5.0187 & & 3.18e-7  & 4.9926 \\  
160 & 8.64e-8 & 5.0136 & & 9.93e-9  & 4.9984 & & 9.95e-9  & 5.0036 & & 9.93e-9  & 4.9983 \\ 
320 & 2.56e-9 & 5.0753 & & 3.12e-10 & 4.9936 & & 3.12e-10 & 4.9950 & & 3.12e-10 & 4.9935 \\      
\hline
\end{tabular}
\label{tab:advection_1d_Linf}
\end{table}

\subsection{1D Burgers' equation}
\begin{example} \label{ex:burgers_1d}
Consider the Riemann problem for the nonlinear Burgers' equation 
\begin{align*}
 u_t + \left(\frac{1}{2} u^2 \right)_x &= 0, \\
                                u(x,0) &= \left\{ 
                                           \begin{array}{ll} 
                                            1, & x \leqslant 0, \\ 
                                            0, & x > 0.
                                           \end{array} 
                                          \right.
\end{align*}
Then the exact solution is a shock wave of the form
$$
   u(x,0) = \left\{ 
             \begin{array}{ll} 
              1, & x - \frac{1}{2} t \leqslant 0, \\ 
              0, & x - \frac{1}{2} t > 0.
             \end{array} 
            \right.
$$
The shock moves to the right at the position $x = \frac{1}{2} t$ for any time $t>0$.
We divide the computational domain $[-1, 1]$ into $N = 40$ uniform cells.
The final time is $T = 1$ and the time step is $\Delta t = 0.4 \Delta x$.

Figure \ref{fig:burgers} depicts the numerical solutions for Example \ref{ex:burgers_1d} at the final time, compared to the exact solution.
We can see that WENO-JS, WENO-M, WENO-Z and WENO-ZR, in order of less dissipation, yield a sharper and sharper approximation around the shock, while keeping the smooth regions without noticeable oscillations.
\end{example}

\begin{figure}[h!]
\centering
\includegraphics[width=0.45\textwidth]{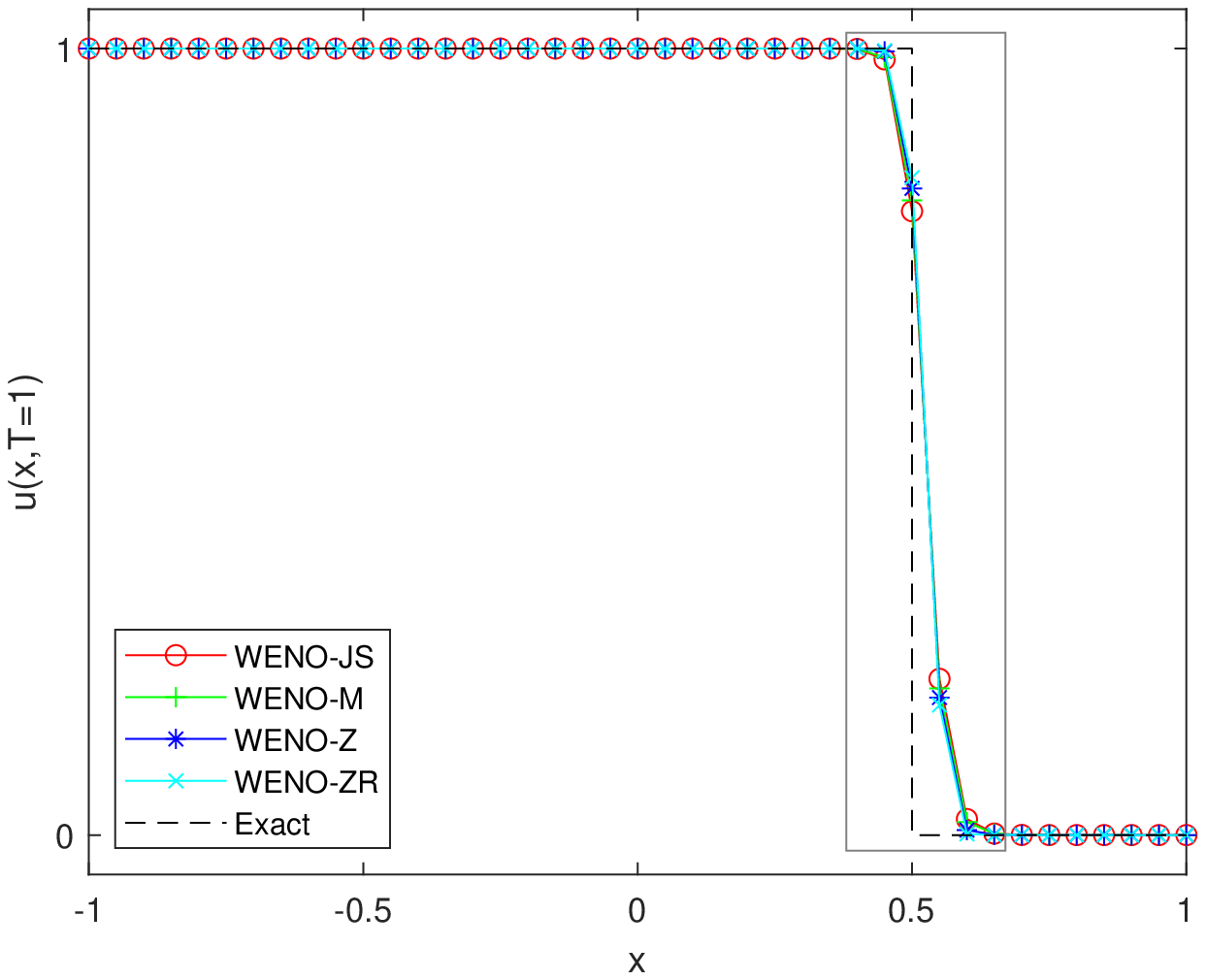}
\includegraphics[width=0.45\textwidth]{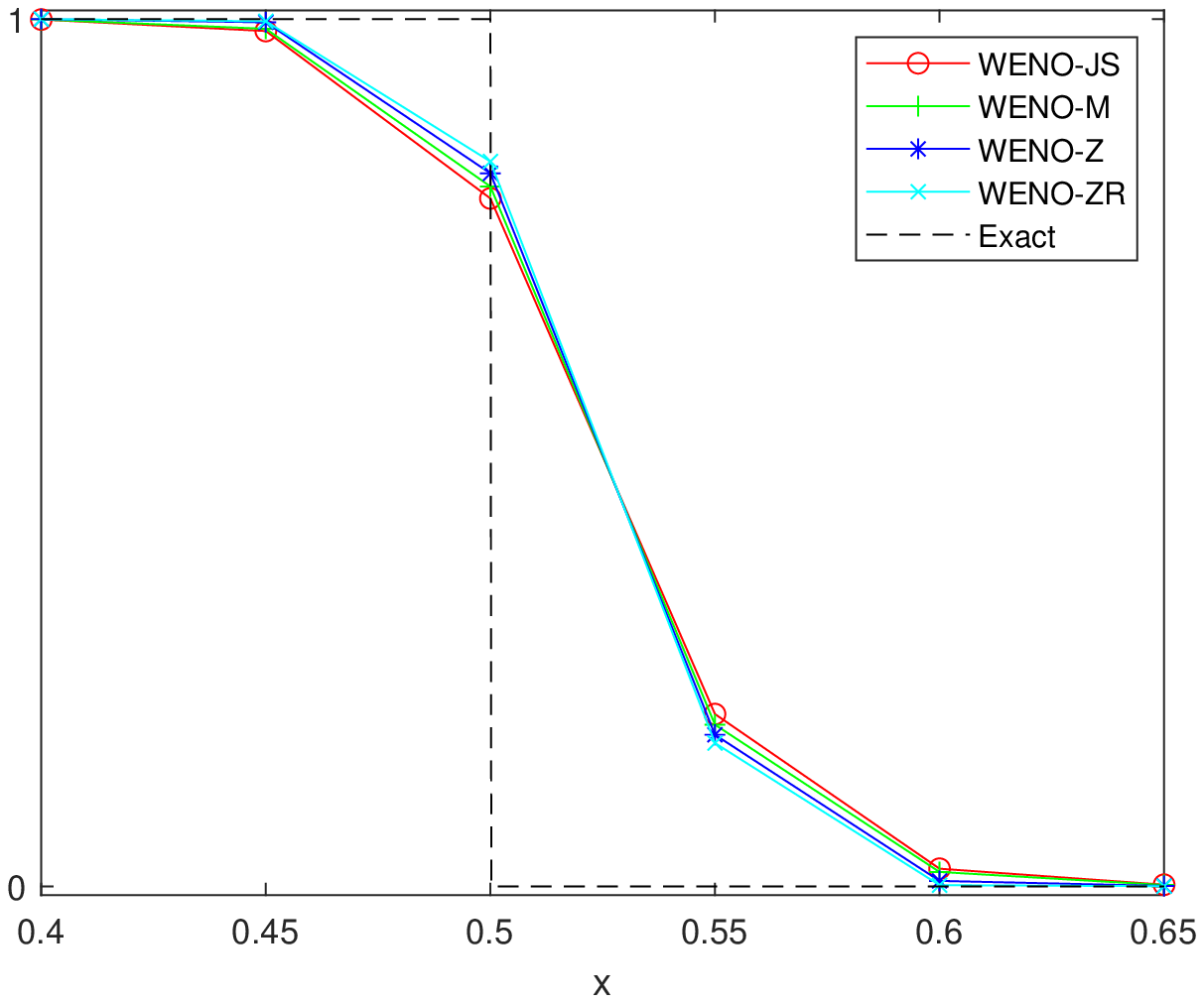}
\caption{Solution profiles for Example \ref{ex:burgers_1d} at $T = 1$ (left), close-up view of the solutions in the box (right) computed by WENO-JS (red), WENO-M (green), WENO-Z (blue) and WENO-ZR (cyan) with $N = 40$.
The dashed black lines are the exact solution.}
\label{fig:burgers}
\end{figure}

\subsection{1D Euler equations}
Next, we solve the 1D Euler equations of gas dynamics
\begin{equation} \label{eq:euler_1d}
 U_t + F(U)_x = 0, 
\end{equation}
where the column vector $U$ of conserved variables and the flux vector $F$ in the $x$ direction are given by 
$$
   U = \left[ \rho, ~\rho u, ~E \right]^T, \quad F(U) = \left[ \rho u, ~\rho u^2 + P, ~u(E+P) \right]^T.
$$
Here $\rho, u$ and $P$ are primitive variables representing density, velocity and pressure, respectively.
The expression for the specific kinetic energy $E$ is 
$$
   E = \frac{P}{\gamma - 1} + \frac{1}{2} \rho u^2 
$$
with $\gamma = 1.4$ for the ideal gas. 

\begin{example} \label{ex:sod_lax_123}
We test the performance of the proposed WENO-ZR scheme by means of the standard shock tube problems with the Riemann-type initial data.
Here the exact solution is obtained with the method in \cite{Toro} for the Riemann problem of the Euler equations \eqref{eq:euler_1d}. 
The computational domain $[-5, 5]$ is divided into $N = 200$ uniform cells and the time step is $\Delta t = 0.2 \Delta x$.

We first consider the Sod's problem with the initial condition in the form of primitive variables
\begin{equation} \label{eq:sod}
 (\rho, u, P ) = \left\{ 
                  \begin{array}{ll} 
                   (1,~0,~1),       & x \leqslant 0, \\ 
                   (0.125,~0,~0.1), & x > 0.
                  \end{array} 
                 \right.
\end{equation}
The final time is $T = 2$.
We present the numerical solutions of the density $\rho$ at the final time, as shown in Figure \ref{fig:sod}.

Now we consider the the Lax's problem. 
The initial condition in this case is
\begin{equation} \label{eq:lax}
 (\rho, u, P ) = \left\{ 
                  \begin{array}{ll} 
                   (0.445,~0.698,~3.528), & x \leqslant 0, \\ 
                   (0.5,~0,~0.571),       & x > 0. 
                  \end{array} 
                 \right. 
\end{equation}
Figure \ref{fig:lax} shows the approximate density obtained from each WENO scheme when solving the Euler equations up to the final time $T = 1.3$.  

We conclude with the $123$ problem, where the initial condition is given by
\begin{equation} \label{eq:123}
 (\rho, u, P ) = \left\{ 
                  \begin{array}{ll} 
                   (1,-2,~0.4), & x \leqslant 0, \\ 
                   (1,~2,~0.4), & x > 0.
                  \end{array} 
                 \right. 
\end{equation}
The simulations of the density $\rho$ at the final time $T=1$ are plotted in Figure \ref{fig:one23}.
 
The top left figure displays the numerical solutions over the entire domain while the rest figures show the solution profiles in the boxes specified in the top left figure.
The sharper numerical results from WENO-ZR are more accurate than WENO-JS, WENO-M and WENO-Z except the approximations in the smooth area near the origin in the bottom left corner of Figure \ref{fig:one23}.
\end{example}

\begin{figure}[h!]
\centering
\includegraphics[width=0.45\textwidth]{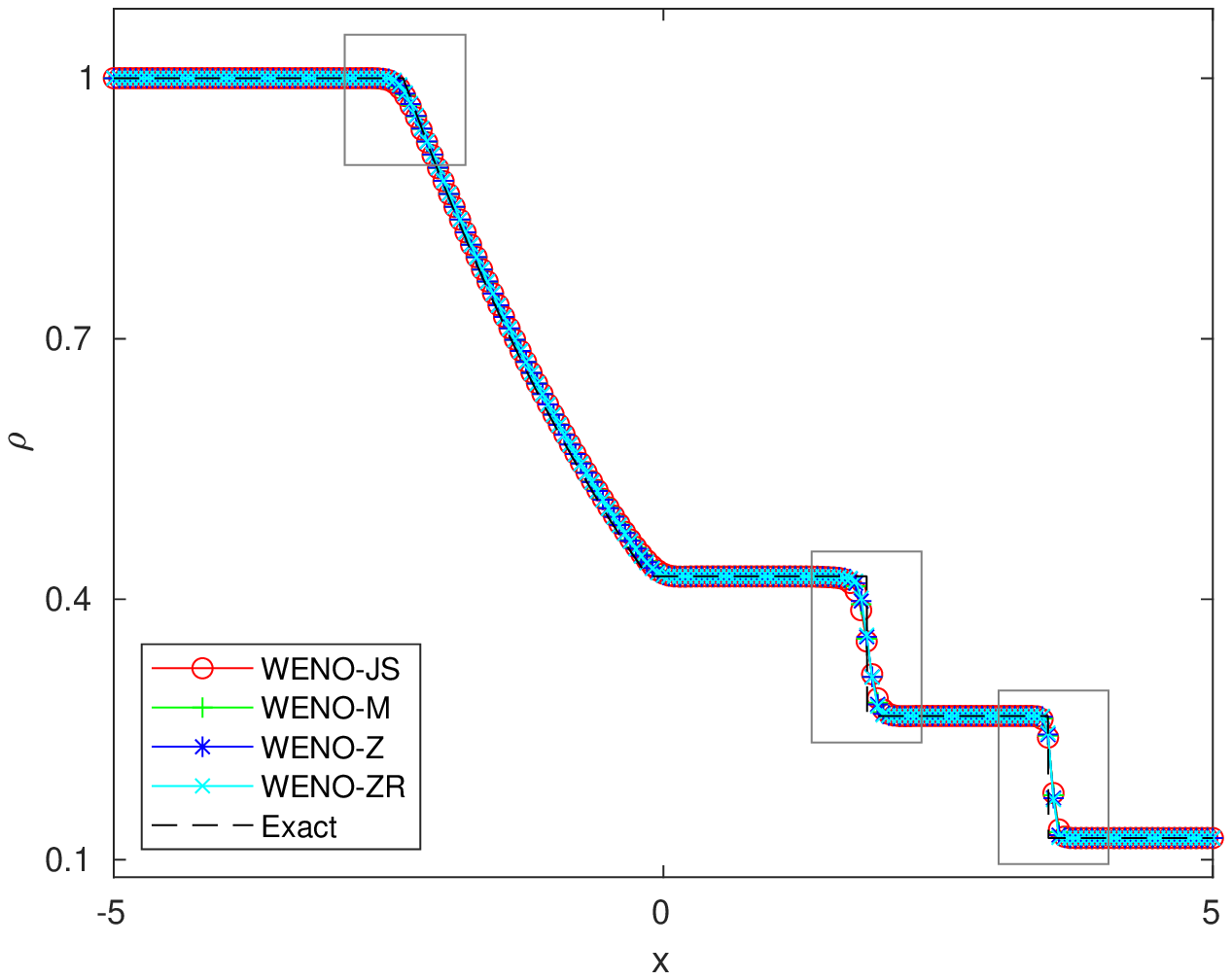}
\includegraphics[width=0.45\textwidth]{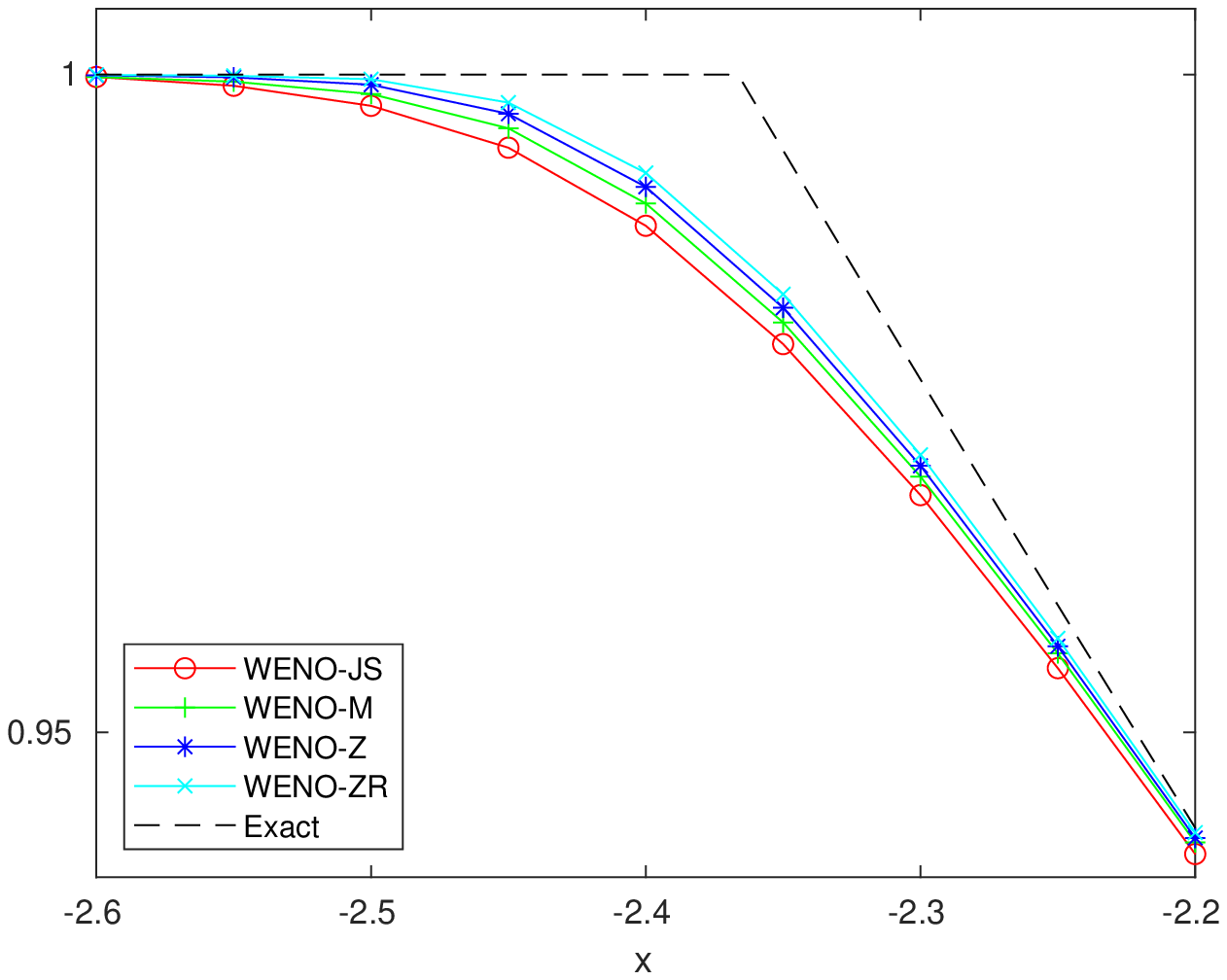}
\includegraphics[width=0.45\textwidth]{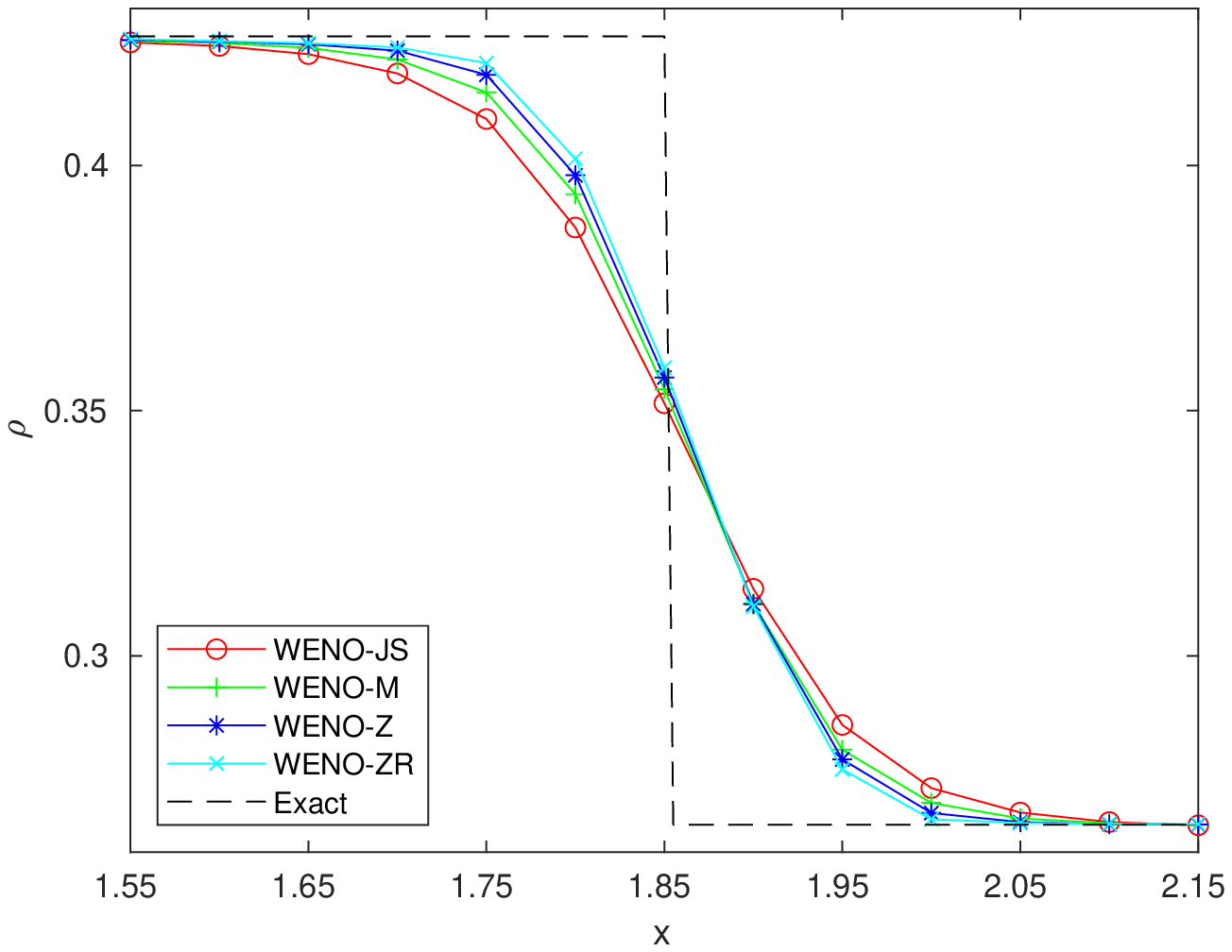}
\includegraphics[width=0.45\textwidth]{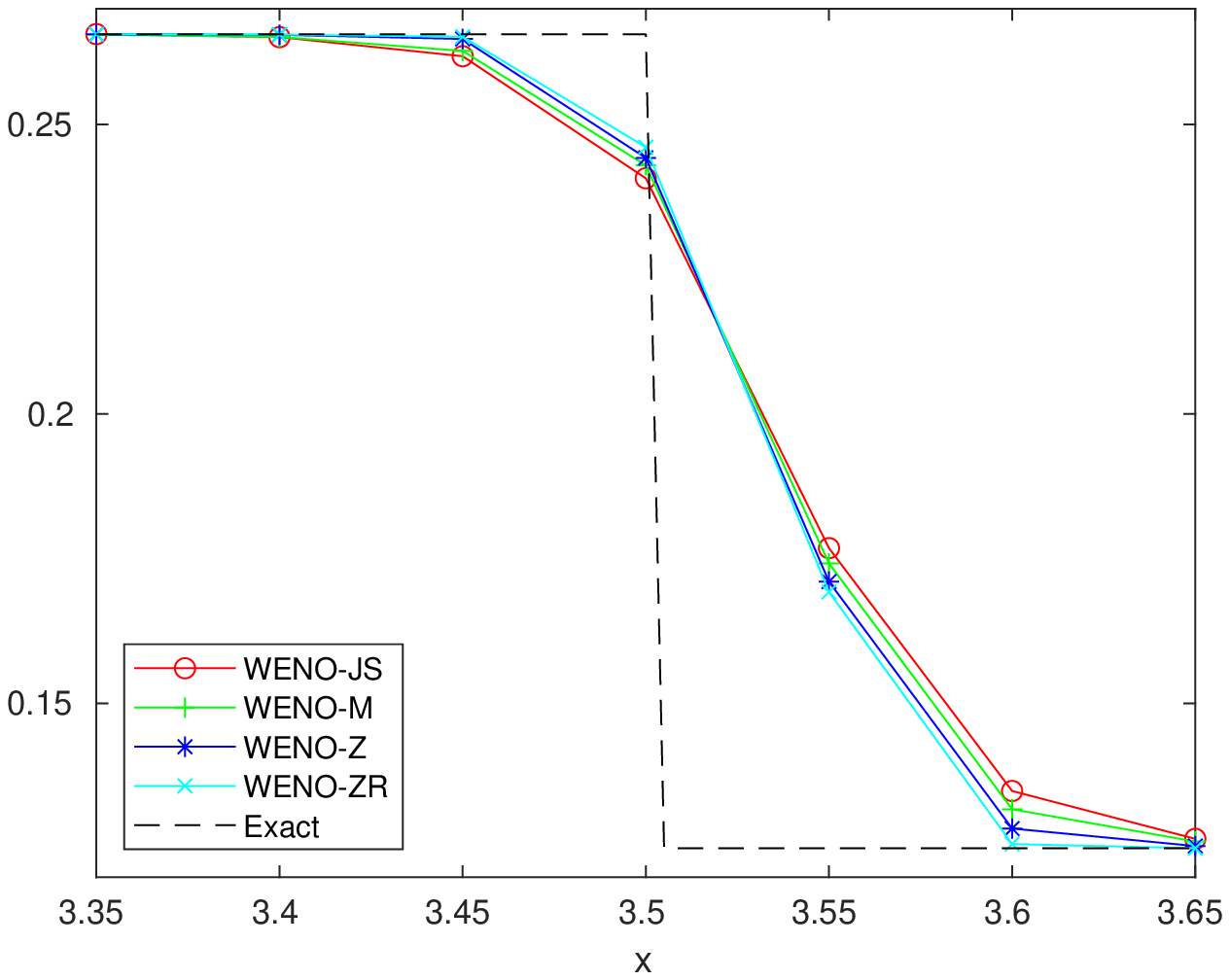}
\caption{Density profiles for the Sod problem \eqref{eq:euler_1d} and \eqref{eq:sod} at $T=2$ (top left), close-up view of the solutions in the boxes from left to right (top right, bottom left, bottom right) approximated by WENO-JS (red), WENO-M (green), WENO-Z (blue) and WENO-ZR (cyan) with $N = 200$. 
The dashed black lines are the exact solution.}
\label{fig:sod}
\end{figure}

\begin{figure}[h!]
\centering
\includegraphics[width=0.45\textwidth]{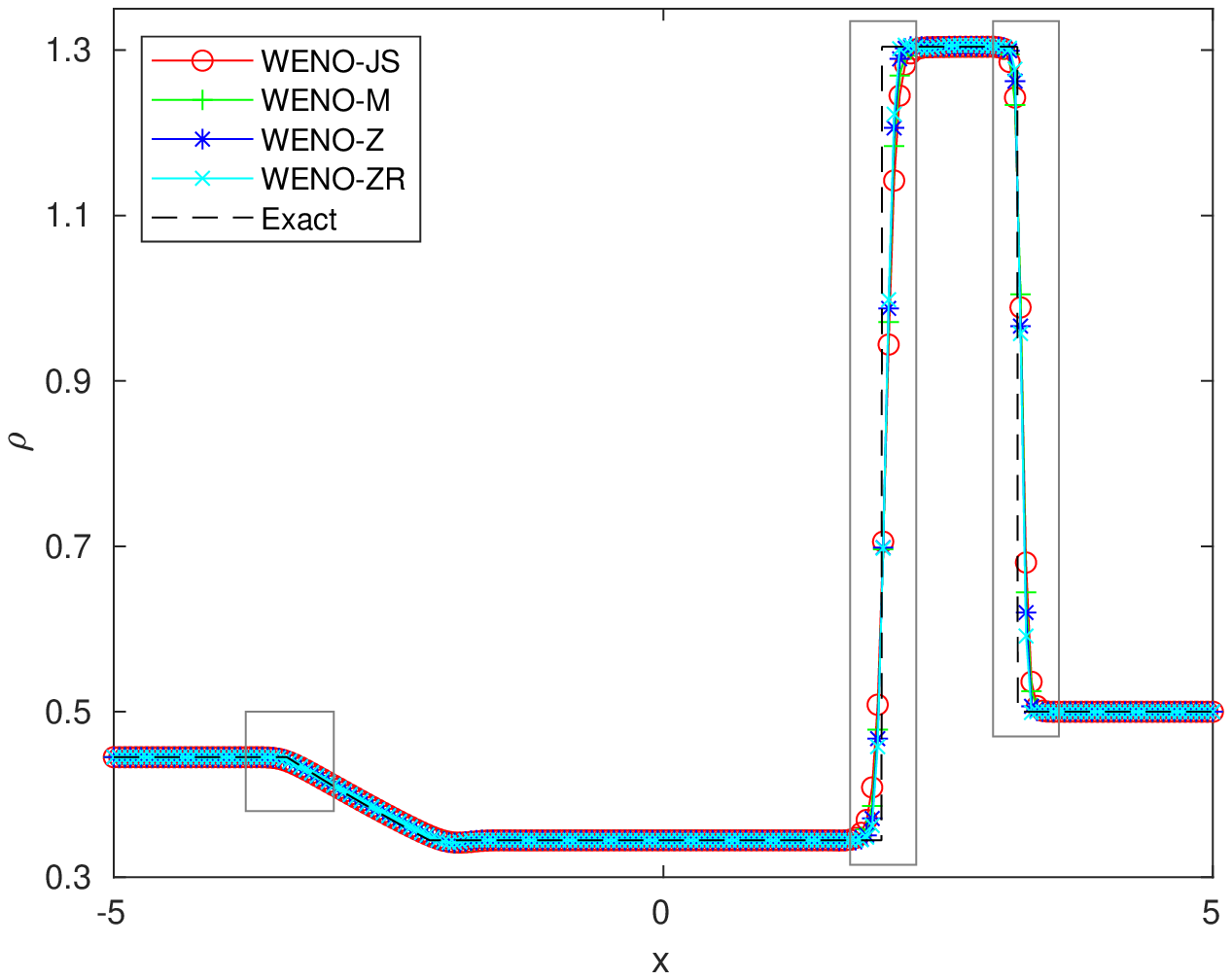}
\includegraphics[width=0.45\textwidth]{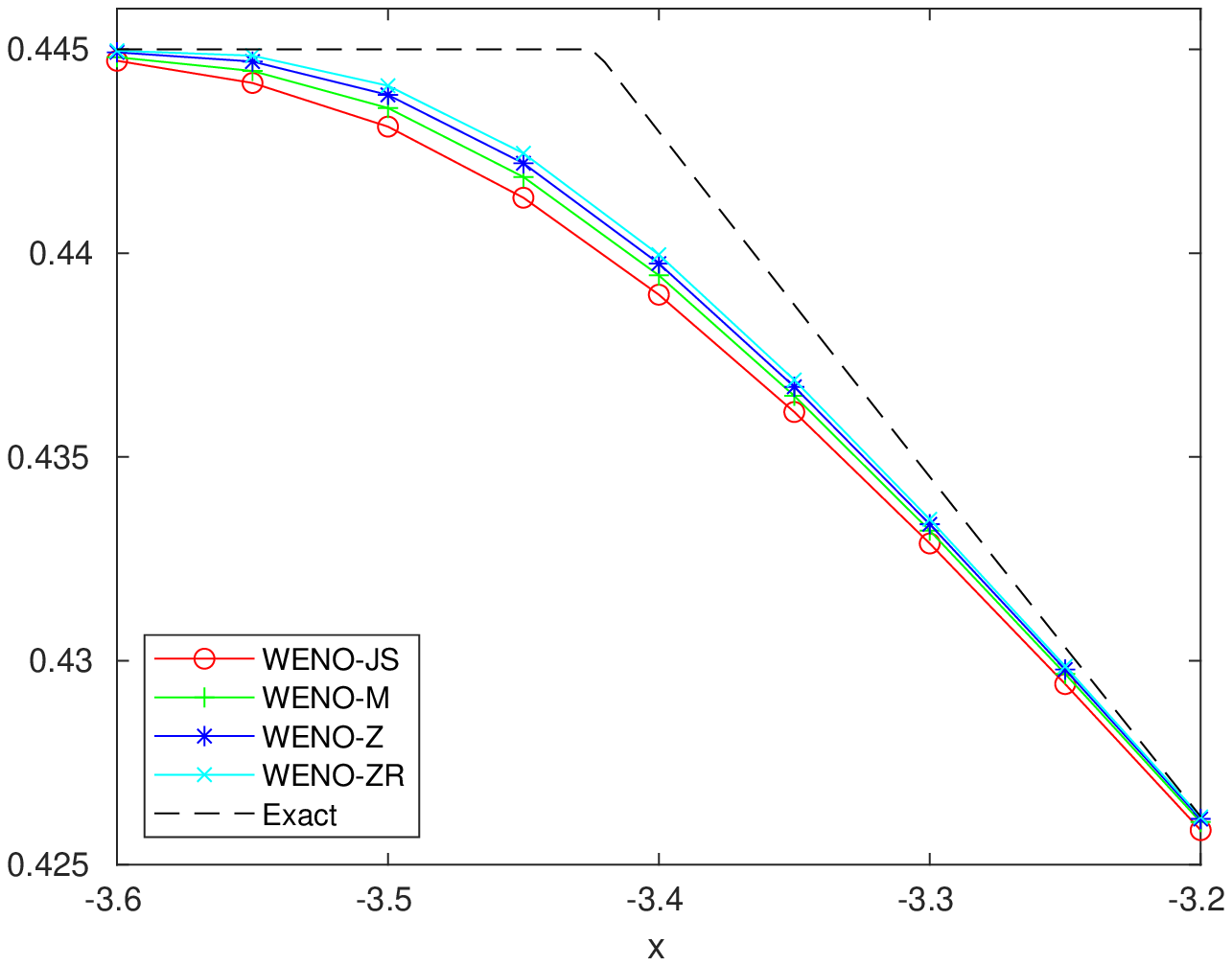}
\includegraphics[width=0.45\textwidth]{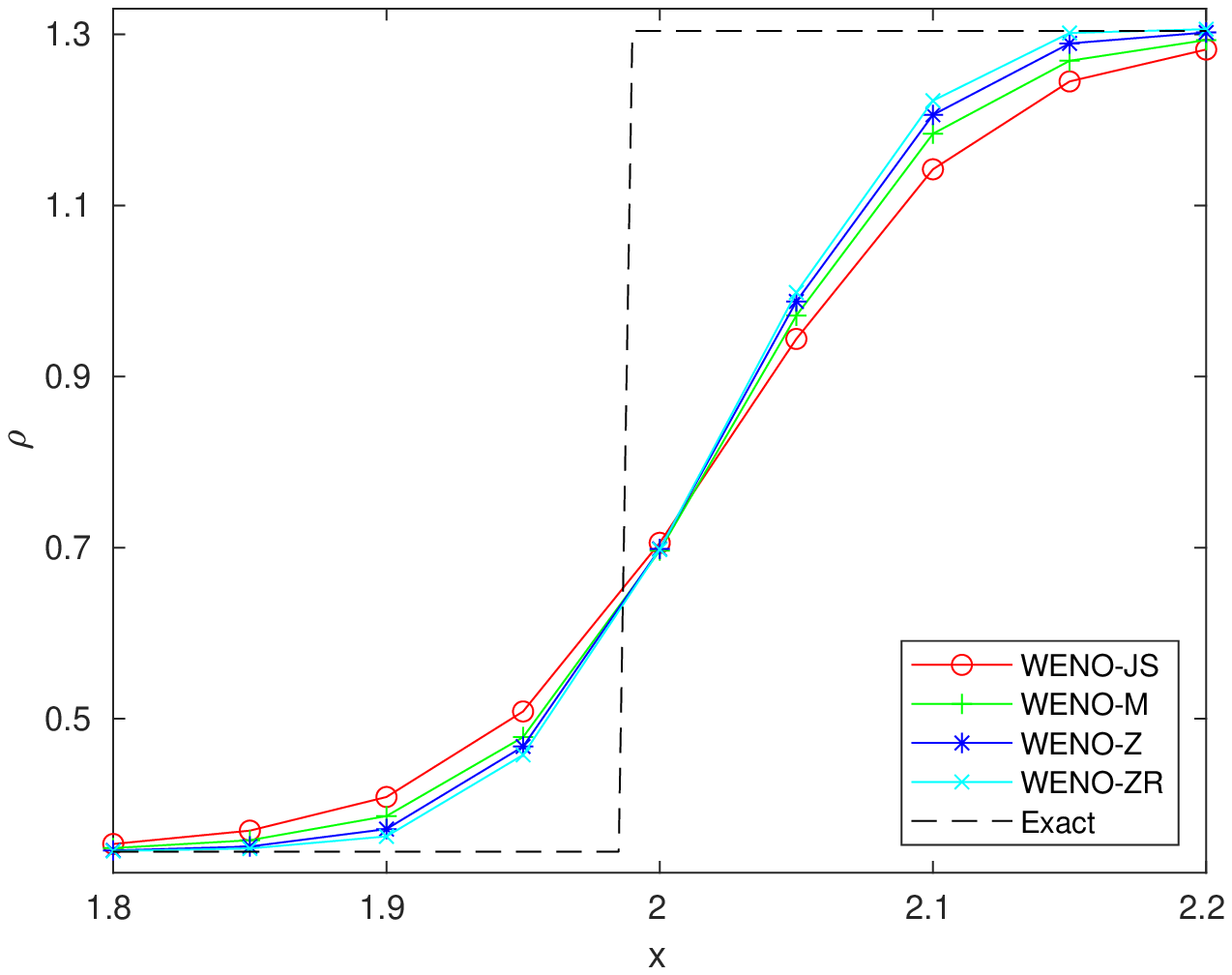}
\includegraphics[width=0.45\textwidth]{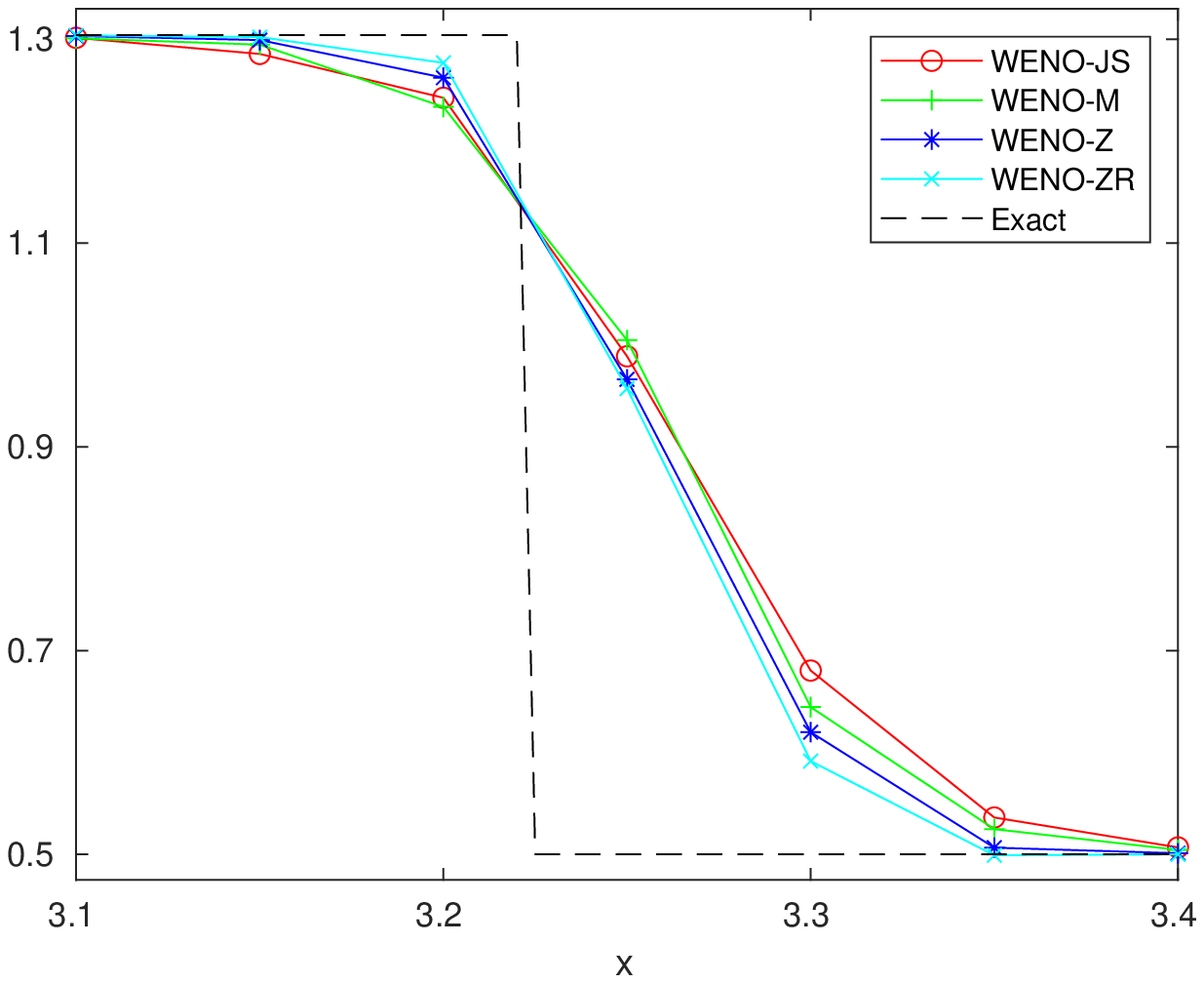}
\caption{Density profiles for the Lax problem \eqref{eq:euler_1d} and \eqref{eq:lax} at $T=1.3$ (top left), close-up view of the solutions in the boxes from left to right (top right, bottom left, bottom right) solved by WENO-JS (red), WENO-M (green), WENO-Z (blue) and WENO-ZR (cyan) with $N = 200$.
The dashed black lines are the exact solution.}
\label{fig:lax}
\end{figure}

\begin{figure}[h!]
\centering
\includegraphics[width=0.45\textwidth]{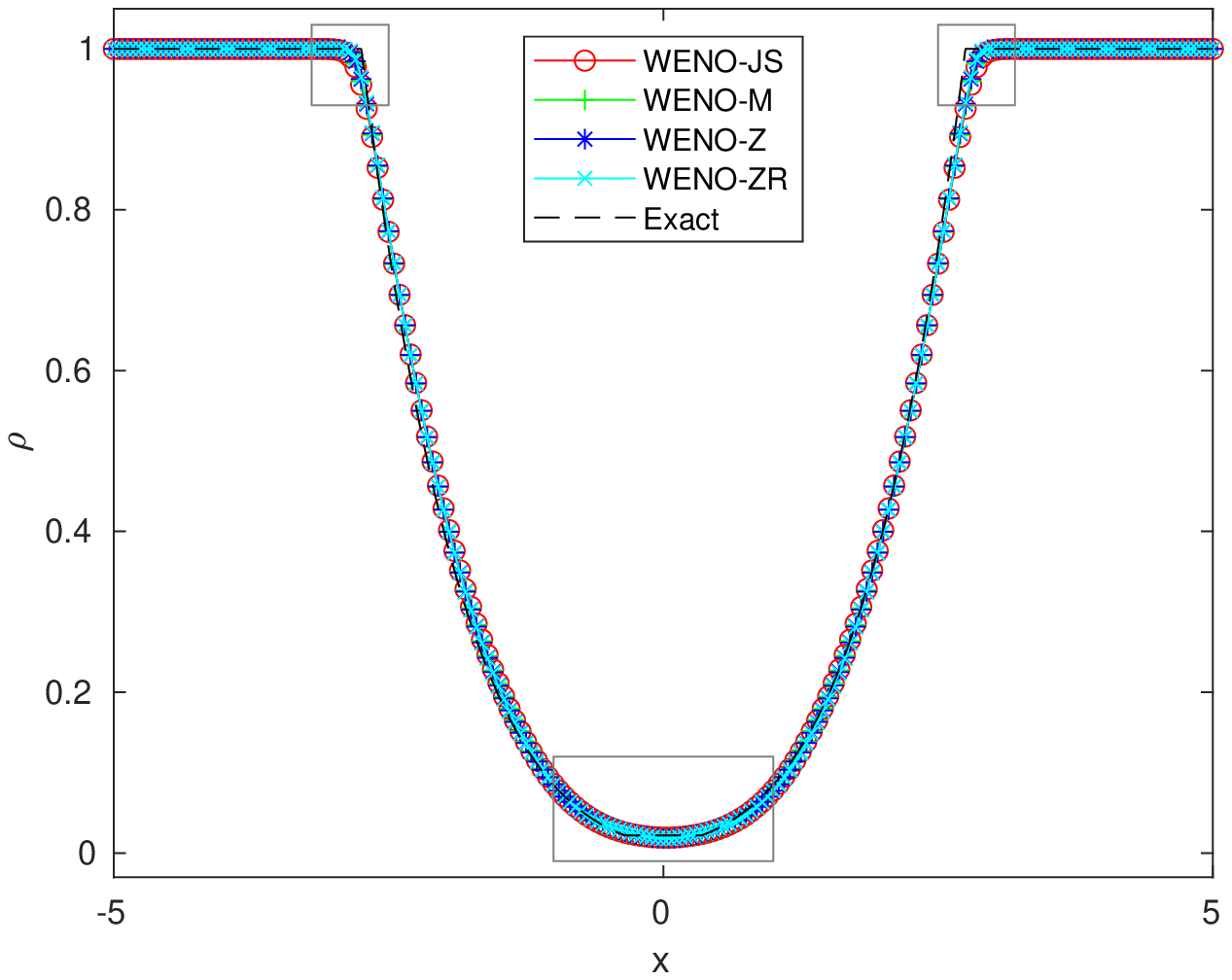}
\includegraphics[width=0.45\textwidth]{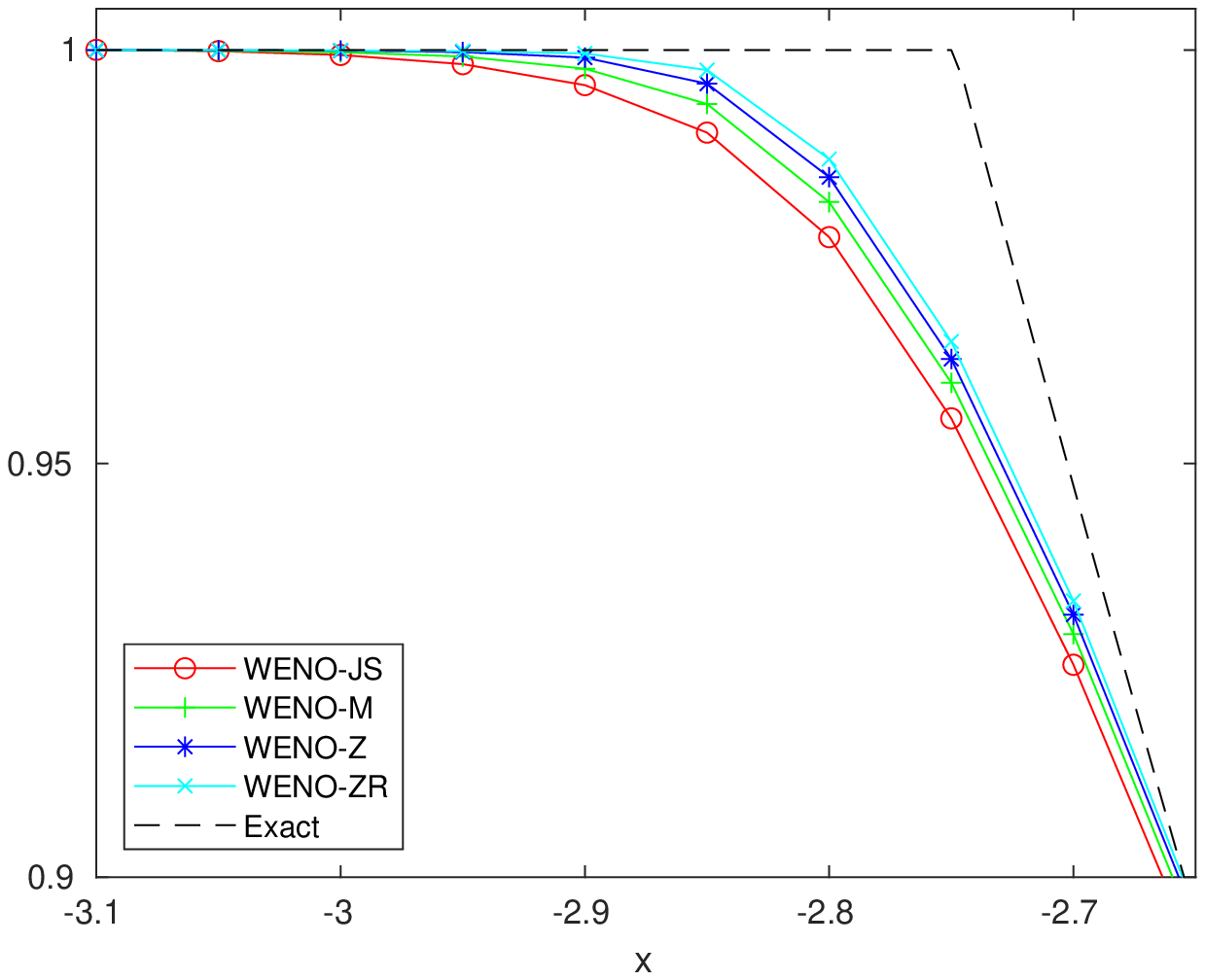}
\includegraphics[width=0.45\textwidth]{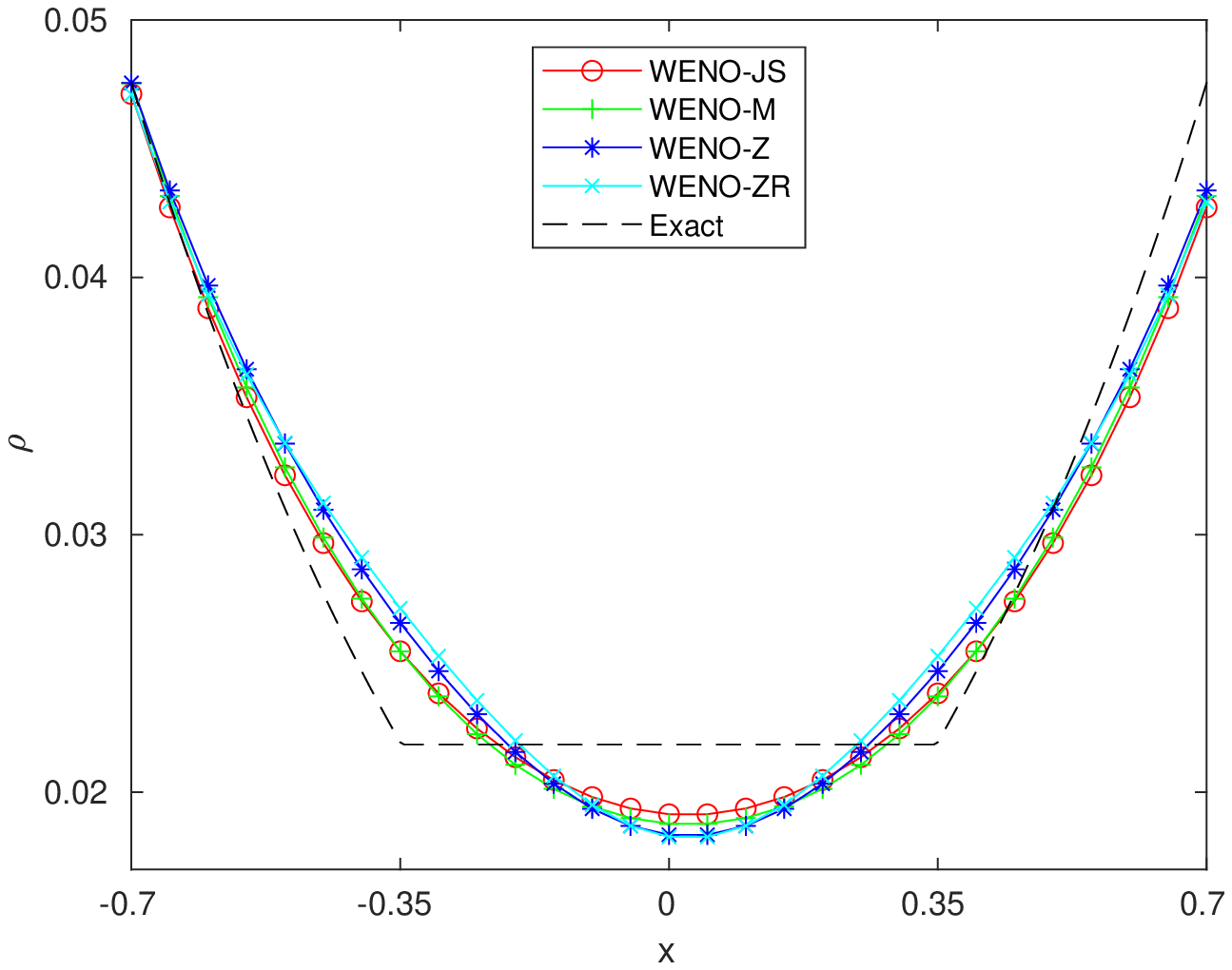}
\includegraphics[width=0.45\textwidth]{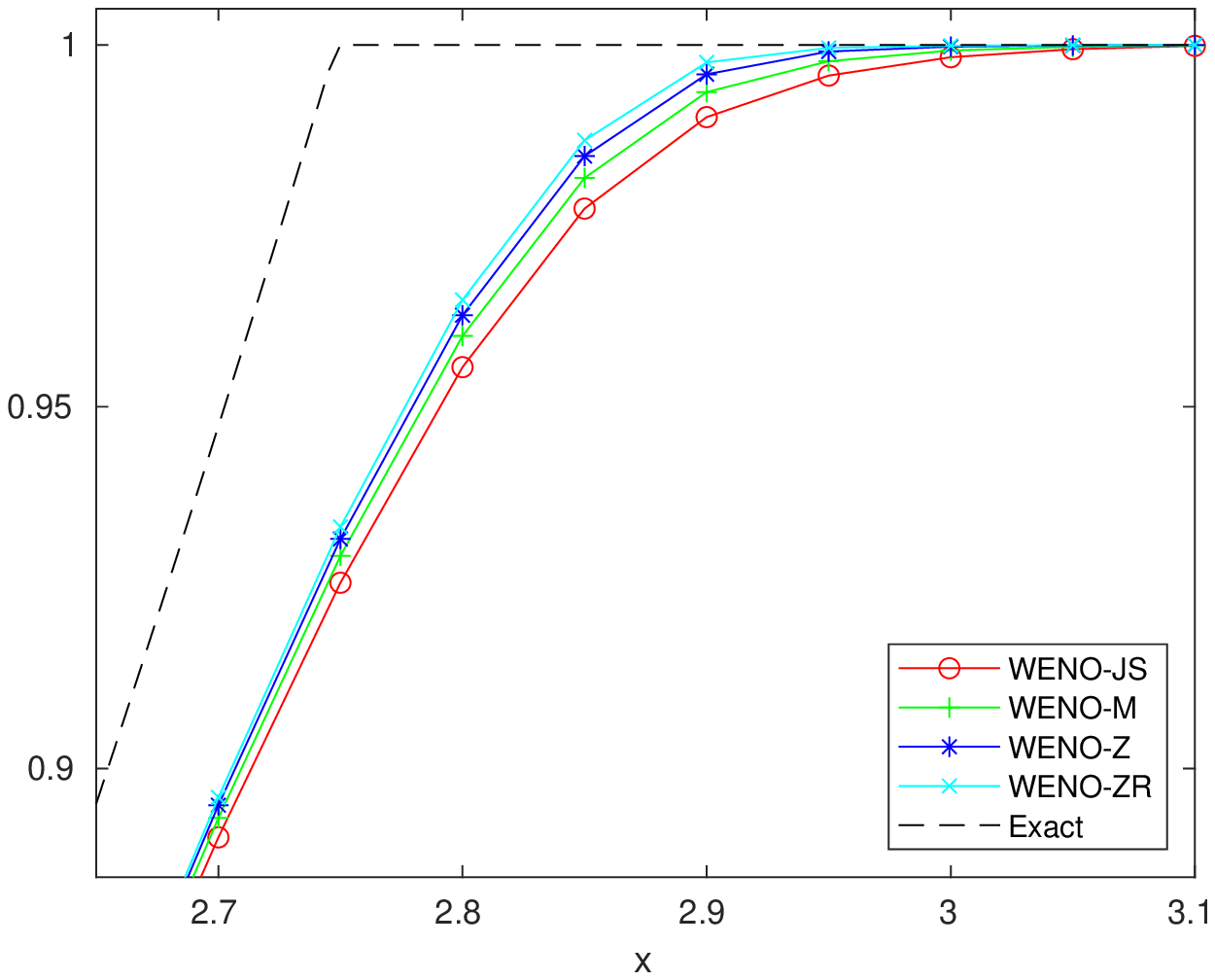}
\caption{Density profiles for the $123$ problem \eqref{eq:euler_1d} and \eqref{eq:123} at $T=1$ (top left), close-up view of the solutions in the boxes from left to right (top right, bottom left, bottom right) simulated by WENO-JS (red), WENO-M (green), WENO-Z (blue) and WENO-ZR (cyan) with $N = 200$.
The dashed black lines are the exact solution.}
\label{fig:one23}
\end{figure}

\begin{example} \label{ex:shock_entropy_wave}
The shock entropy wave interaction problem \cite{ShuOsherII} consists of a right moving Mach $3$ shock and an entropy wave in density.
The initial condition is specified as
$$
   (\rho, u, P ) = \left\{ 
                    \begin{array}{ll} 
                     (3.857143,~2.629369,~10.333333), & x < -4, \\ 
                     (1 + 0.2 \sin(kx),~0,~1),        & x \geqslant -4,
                    \end{array} 
                   \right. 
$$
where $k$ is the wave number of the entropy wave. 

For $k=5$, we take a uniform grid with $N = 200$ cells on the computational domain $[-5, 5]$ with the time step $\Delta t = 0.05 \Delta x$.
The numerical solution, computed by WENO-M with a high resolution of $N = 2000$ points, will be referred to as the ``exact'' solution.
Figure \ref{fig:shock_entropy_wave_k5} shows the the density profile approximated by those WENO schemes at $T = 2$.

If we increase the wave number $k$ to $10$, the computational domain $[-5, 5]$ is divided into $N = 500$ uniform cells and the time step is kept as $\Delta t = 0.05 \Delta x$.
We still take the solution computed by WENO-M with $N = 2000$ points as the ``exact'' solution.
The numerical solutions of the density at $T = 2$ are given in Figure \ref{fig:shock_entropy_wave_k10}.

We observe that WENO-ZR performs better in capturing the fine structure in the density profile than WENO-JS, WENO-M and WENO-Z and hence improves the resolution in this example, because of the sharper approximations provided by WENO-ZR.
\end{example}

\begin{figure}[h!]
\centering
\includegraphics[width=0.45\textwidth]{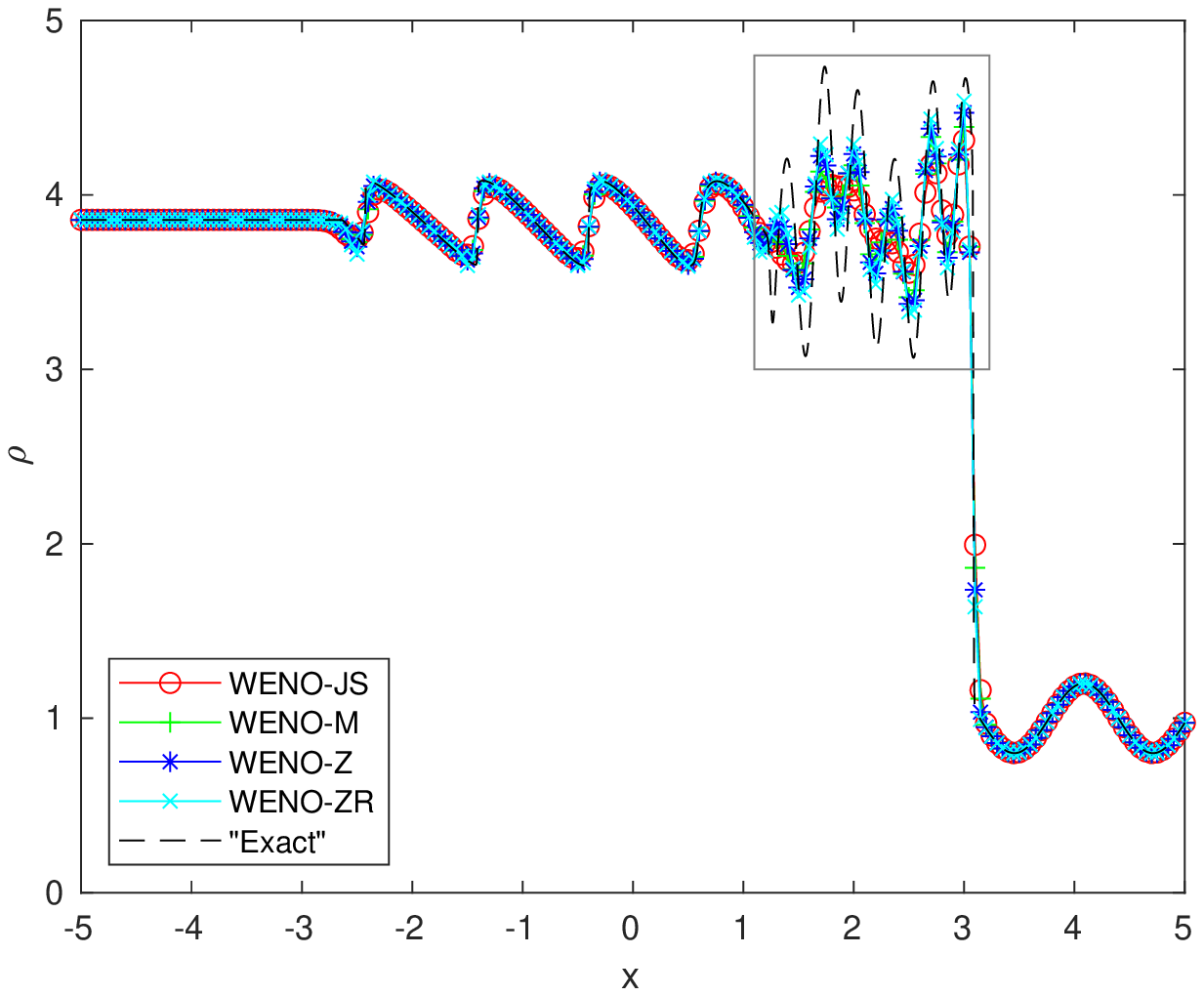}
\includegraphics[width=0.45\textwidth]{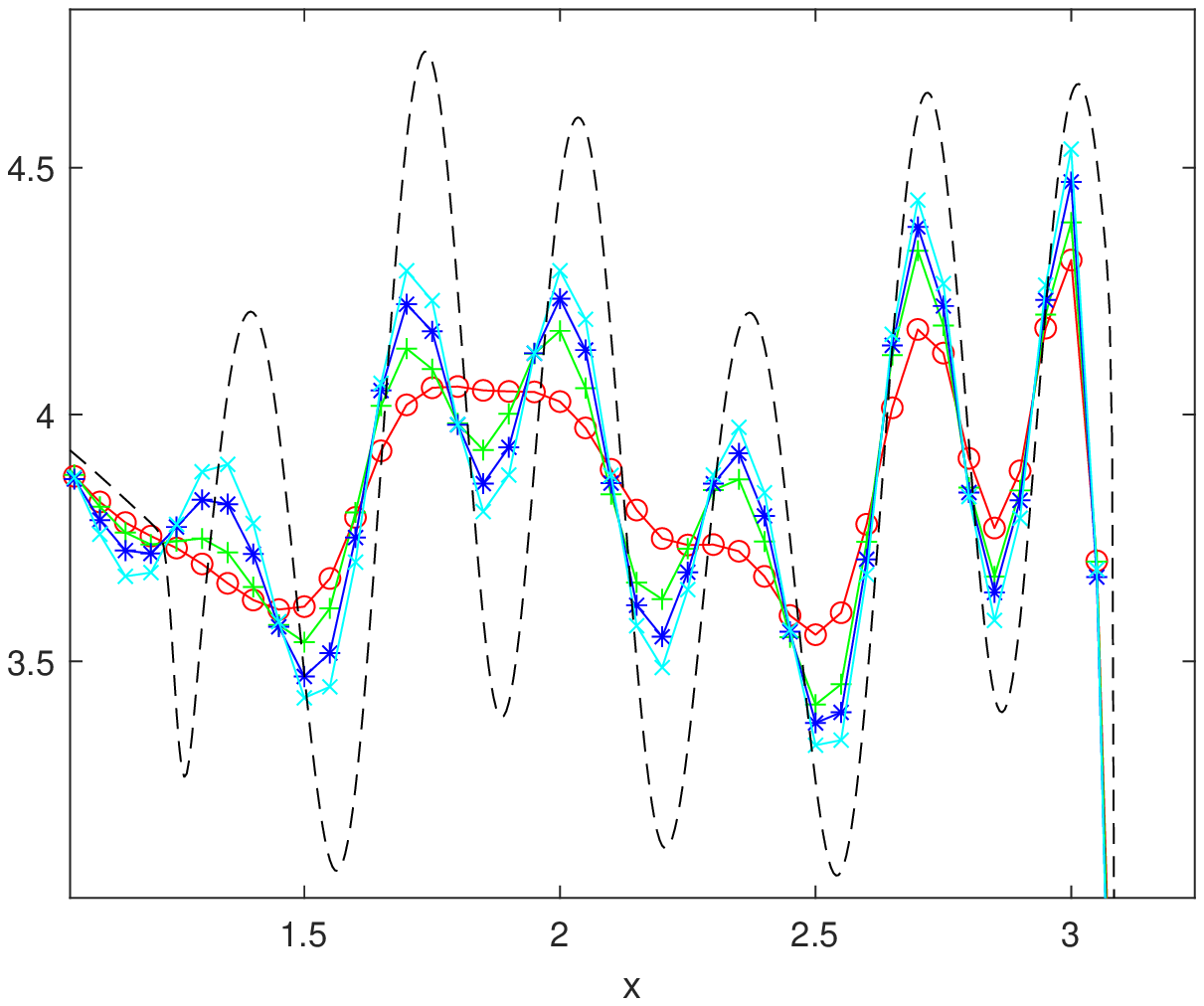}
\caption{Solution profiles for Example \ref{ex:shock_entropy_wave} with $k=5$ at $T = 2$ (left), close-up view of the solutions in the box (right) computed by WENO-JS (red), WENO-M (green), WENO-Z (blue) and WENO-ZR (cyan) with $N = 200$.
The dashed black lines are generated by WENO-M with $N = 2000$.}
\label{fig:shock_entropy_wave_k5}
\end{figure}

\begin{figure}[h!]
\centering
\includegraphics[width=0.45\textwidth]{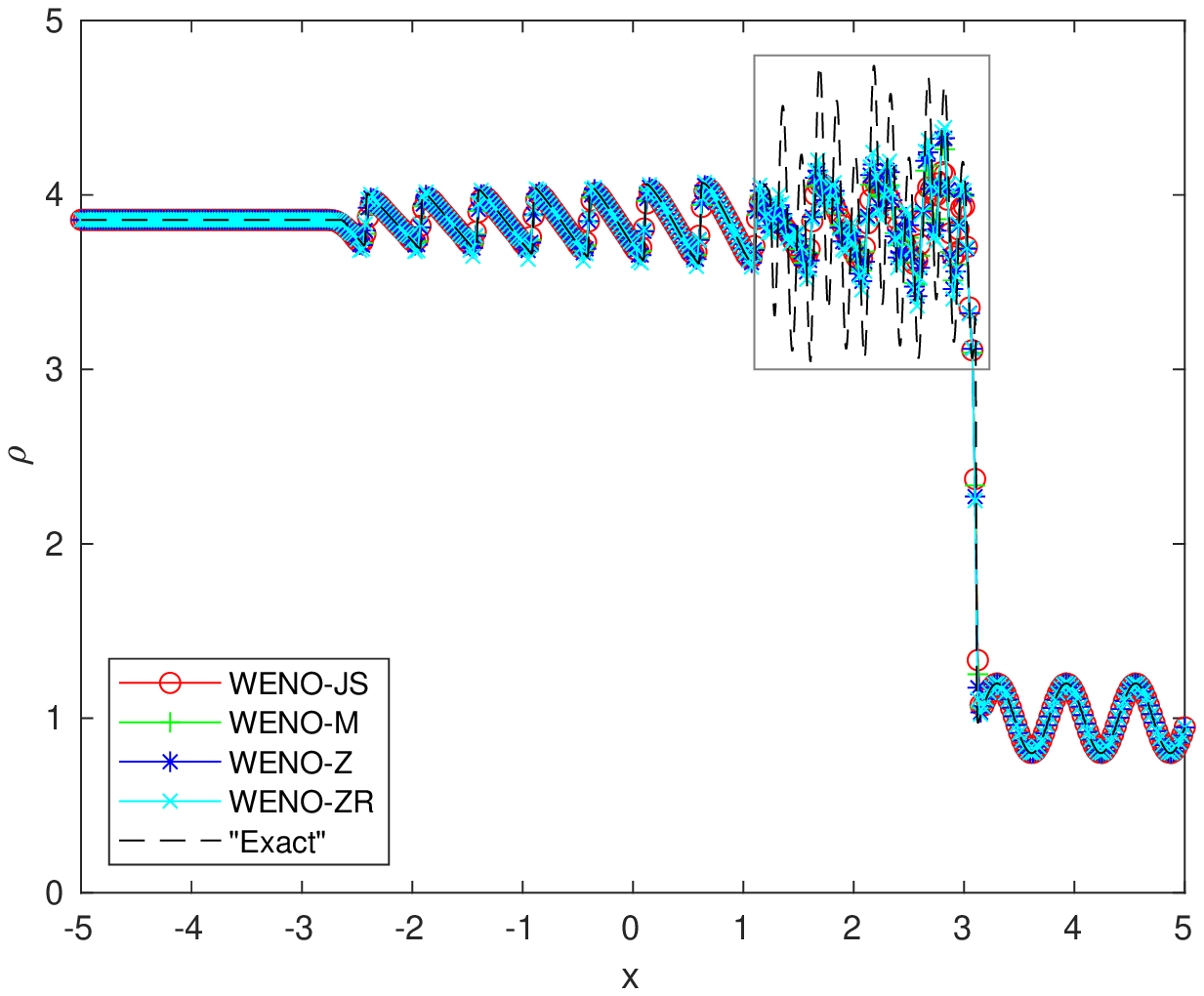}
\includegraphics[width=0.45\textwidth]{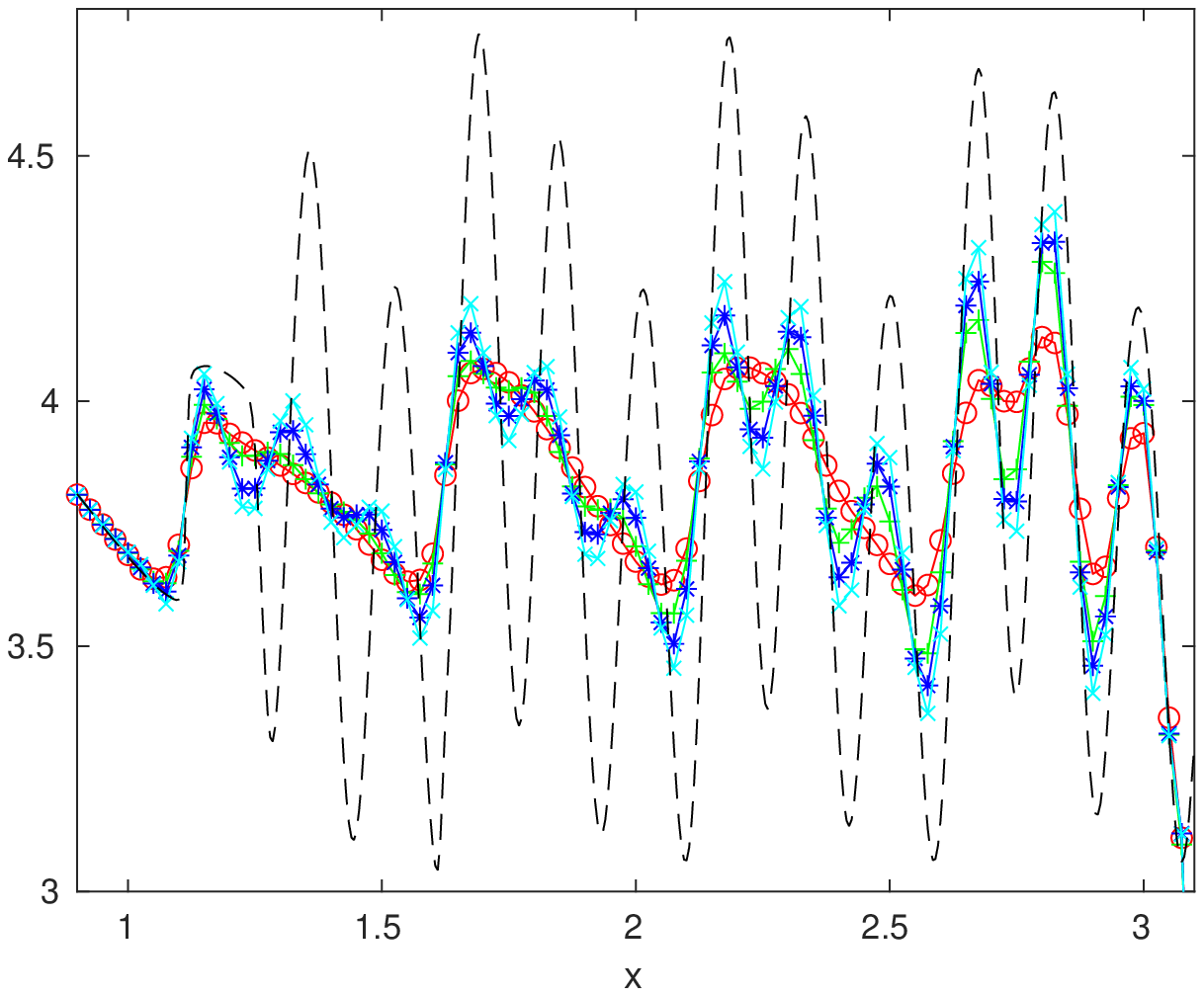}
\caption{Solution profiles for Example \ref{ex:shock_entropy_wave} with $k=10$ at $T = 2$ (left), close-up view of the solutions in the box (right) computed by WENO-JS (red), WENO-M (green), WENO-Z (blue) and WENO-ZR (cyan) with $N = 400$.
The dashed black lines are generated by WENO-M with $N = 2000$.}
\label{fig:shock_entropy_wave_k10}
\end{figure}

\subsection{2D Euler equations} 
The 2D Euler equations of gas dynamics are of the form
\begin{equation} \label{eq:euler_2d}
 U_t + F(U)_x + G(U)_y = 0, 
\end{equation}
where the conserved vector $U$ and the flux functions $F,\, G$ in the $x,\, y$ directions, respectively, are 
\begin{align*}
    U &= \left[ \rho, ~\rho u, ~\rho v, ~E \right]^T, \\
 F(U) &= \left[ \rho u, ~\rho u^2 + P, ~\rho u v, ~u(E+P) \right]^T, \\ 
 G(U) &= \left[ \rho v, ~\rho u v, ~\rho v^2 + P, ~v(E+P) \right]^T.
\end{align*}
As in 1D case, $\rho$ is the density and $P$ is the pressure. 
The primitive variables $u$ and $v$ denote $x$-component and $y$-component velocity, respectively.
The specific kinetic energy $E$ has the formula
$$
   E = \frac{P}{\gamma - 1} + \frac{1}{2} \rho ( u^2 + v^2 ) 
$$
with $\gamma = 1.4$ for the ideal gas. 

\begin{example} \label{ex:euler_2d_riemann}
We consider the Riemann problem in \cite{Kurganov} for the 2D Euler equations \eqref{eq:euler_2d}, along with the initial condition,
$$
   (\rho, u, v, P ) = \left\{ 
                       \begin{array}{ll} 
                        (1.5,~0,~0,~1.5),             & x > 0.8,~y > 0.8, \\
                        (0.5323,~1.206,~0,~0.3),      & x < 0.8,~y > 0.8, \\
                        (0.138,~1.206,~1.206,~0.029), & x < 0.8,~y < 0.8, \\
                        (0.5323,~0,~1.206,~0.3),      & x > 0.8,~y < 0.8,
                       \end{array}
                      \right.
$$
We divide the square computational domain $[0,1] \times [0,1]$ into $N_x \times N_y = 200 \times 200$ uniform cells with the time step $\Delta t = 0.2 \min (\Delta x, \Delta y)$.
The free stream boundary conditions are imposed in both $x$- and $y$-directions.
The numerical performance of density computed by WENO-ZR at the final time $T=0.8$, compared with those by WENO-JS, WENO-M and WENO-Z, is presented in Figure \ref{fig:eer2d}.
It is observed that WENO-ZR achieves a better resolution of the jet in the zones of strong velocity shear and the instability along the jet's neck.
\end{example}

\begin{figure}[h!]
\centering
\includegraphics[width=0.45\textwidth]{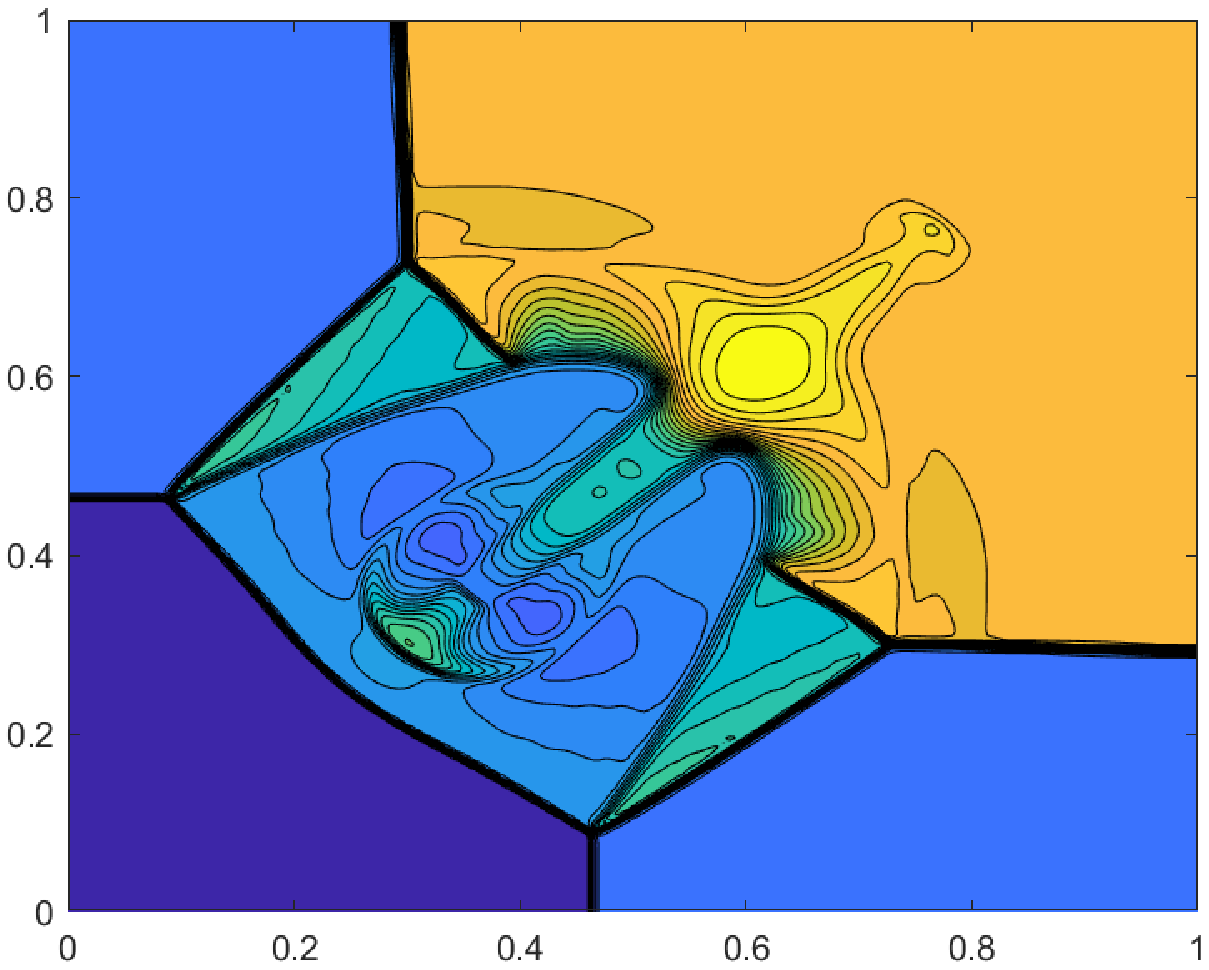}
\includegraphics[width=0.45\textwidth]{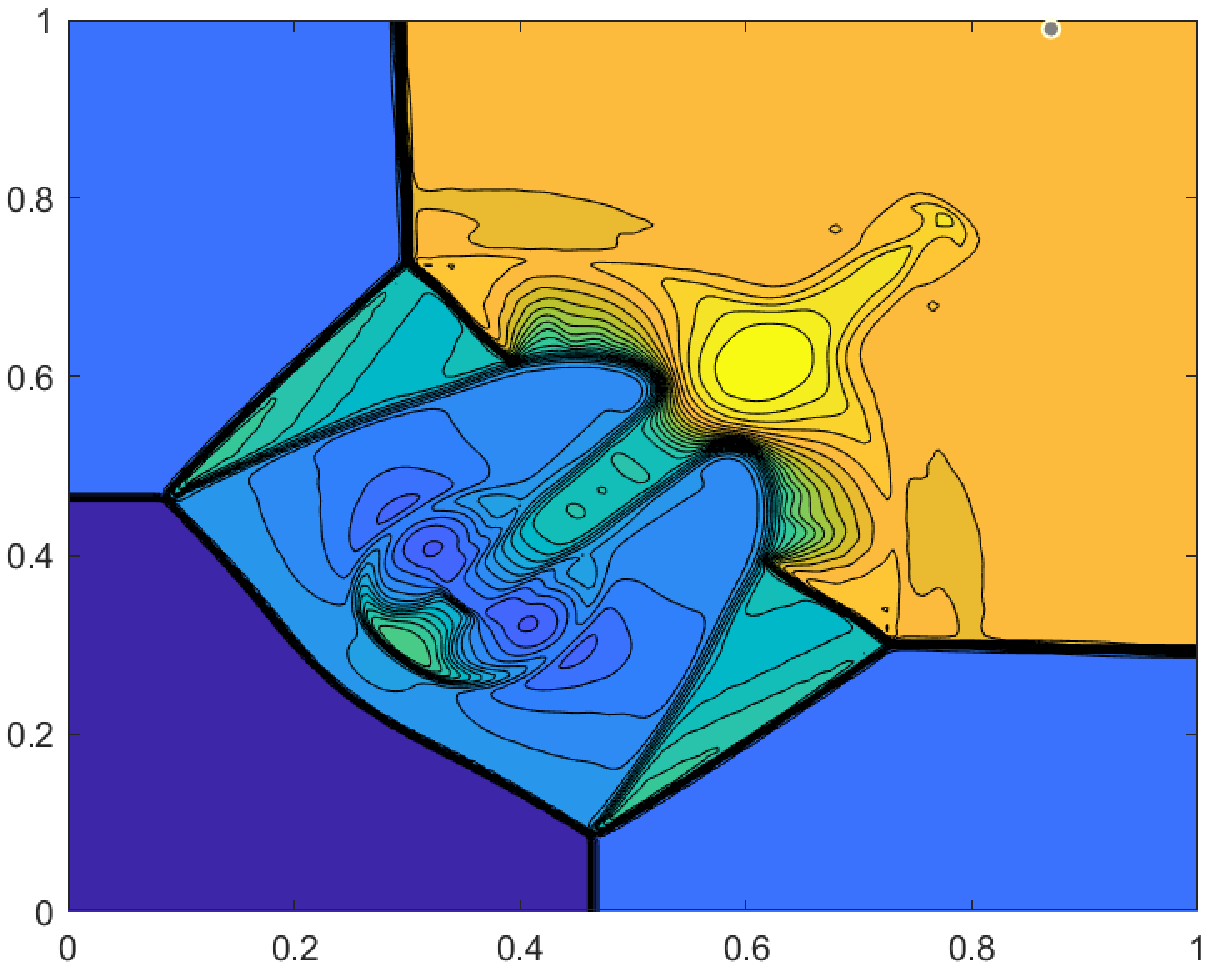}
\includegraphics[width=0.45\textwidth]{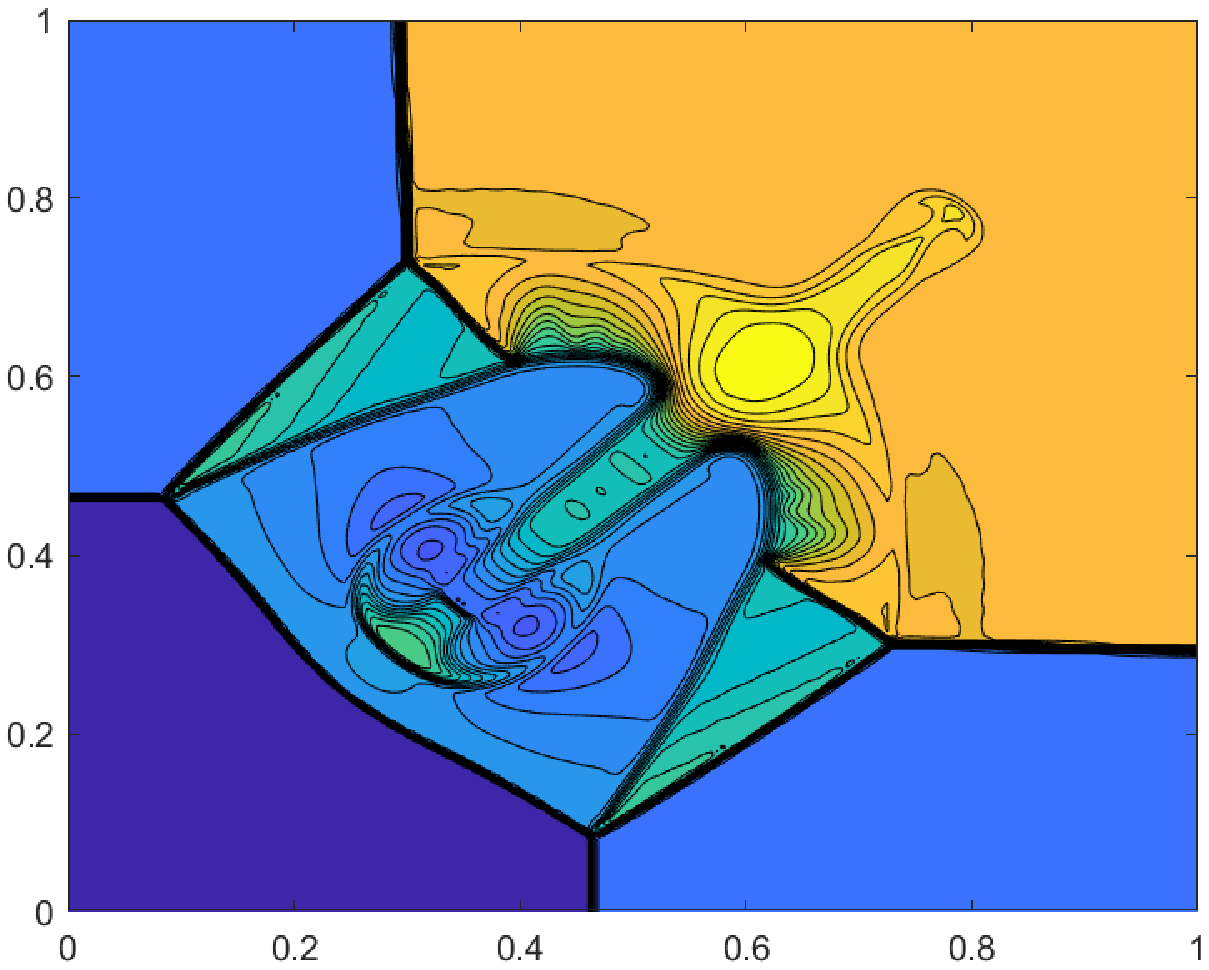}
\includegraphics[width=0.45\textwidth]{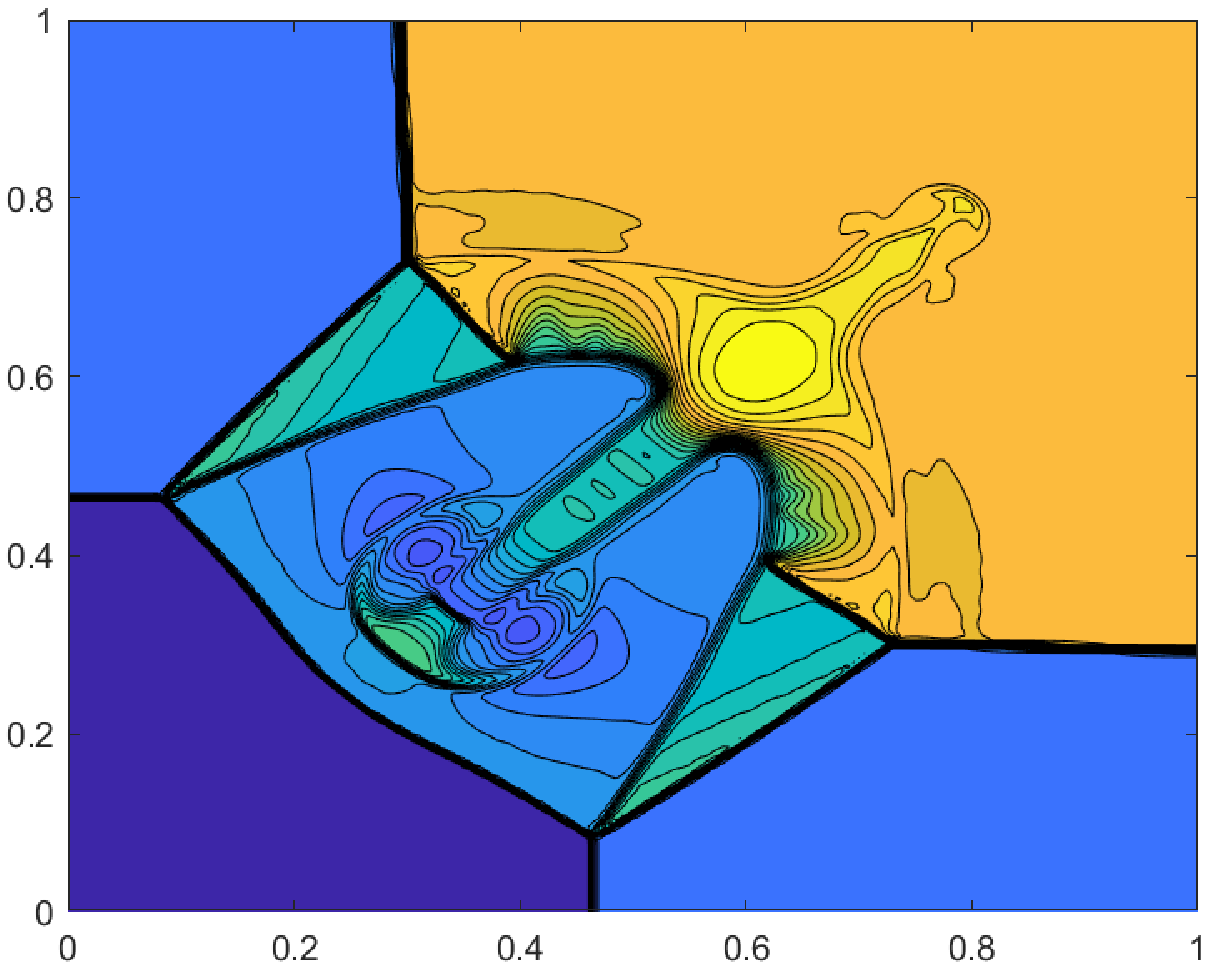}
\caption{Density in the filled contour plot for Example \ref{ex:euler_2d_riemann} at $T=0.8$ by WENO-JS (top left), WENO-M (top right), WENO-Z (bottom left) and WENO-ZR (bottom right) with $N_x = N_y = 200$.
Each contour plot displays contours at $30$ levels of the density.}
\label{fig:eer2d}
\end{figure}

\begin{example} \label{ex:double_mach_reflection}
We end this section with the double Mach reflection problem \cite{Woodward} that is again a standard test case for high resolution schemes.
The initial condition is given by 
$$
   (\rho, u, v, P ) = \left\{ 
                       \begin{array}{ll} 
                        (8,~8.25 \cos \theta,~-8.25 \sin \theta,~116.5), & x < \frac{1}{6} + \frac{y}{\sqrt{3}}, \\ 
                        (1.4,~0,~0,~1), & x \geqslant \frac{1}{6} + \frac{y}{\sqrt{3}}, 
                       \end{array} 
                      \right. 
$$
with $\theta = \frac{\pi}{6}$.
The computational domain is chosen as $[0,4] \times [0,1]$ with $N_x \times N_y = 480 \times 119$ uniform cells. 
We take the time step $\Delta t = 0.005 \min (\Delta x, \Delta y)$.
For the bottom boundary $y=0$, the left initial condition is imposed for the interval $\left[ 0, \frac{1}{6} \right)$ on the $x$-axis whereas the reflective boundary condition is applied to the interval $\left[ \frac{1}{6}, 4 \right]$.
At the top boundary $y=1$, the exact motion of the right-moving Mach $10$ oblique shock is used.
The free stream boundary conditions are imposed for the left boundary $x=0$ and the right boundary $x=4$.
The simulation is conducted until the final time $T = 0.2$, when the strong shock, joining the contact surface and transverse wave, sharpens.
In Figure \ref{fig:dmr}, the numerical density of each WENO scheme for the domain $[0,3] \times [0,1]$ at the final time is plotted.
We can see that WENO-ZR captures the structure, which exists in all experimented schemes, with a higher resolution.
\end{example}

\begin{figure}[h!]
\centering
\includegraphics[width=0.45\textwidth]{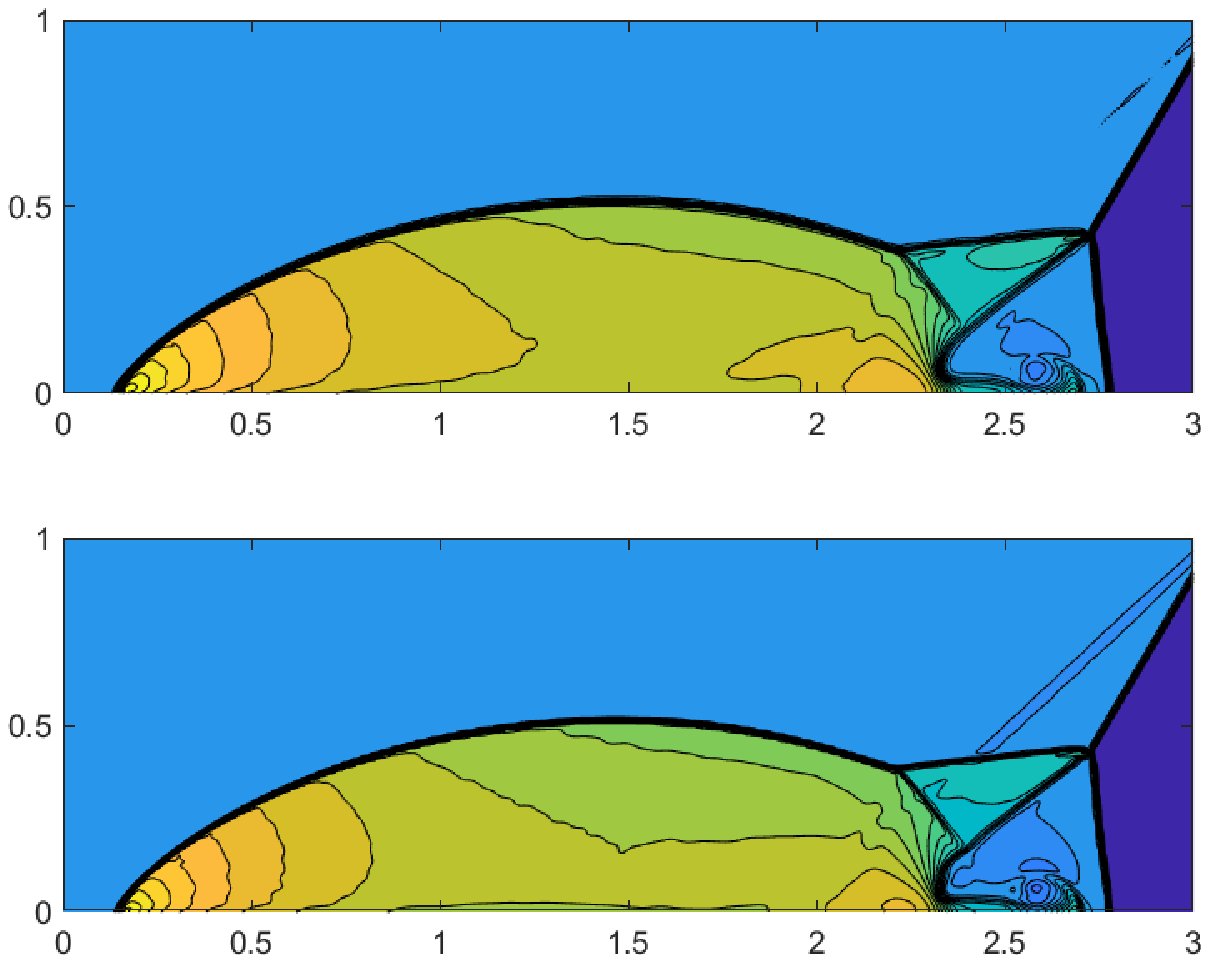}
\includegraphics[width=0.45\textwidth]{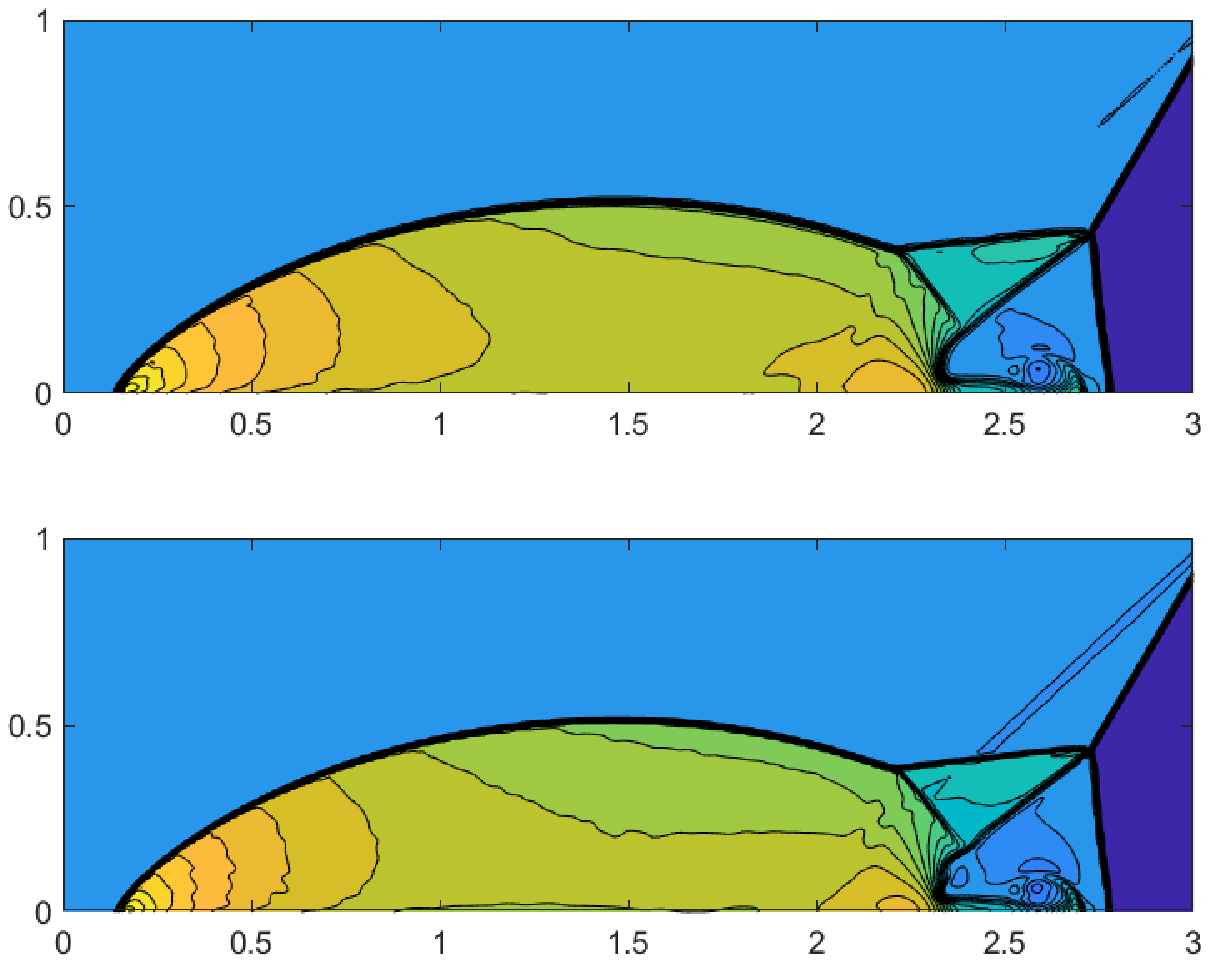}
\caption{Density in the filled contour plot for Example \ref{ex:double_mach_reflection} at $T=0.2$ by WENO-JS (top left), WENO-M (top right), WENO-Z (bottom left) and WENO-ZR (bottom right) with $N_x = 480 $ and $N_y = 119$.
Each contour plot displays contours at $30$ levels of the density.}
\label{fig:dmr}
\end{figure}


\section{Concluding remarks} 
In this paper, we have designed a new set of Z-tyoe nonlinear weights for the fifth-order finite difference WENO scheme.
It is shown that the new nonlinear weights satisfy the sufficient condition for the fifth order accuracy, even at critical points.
We also have proved that the nonlinear weights corresponding to the discontinuity increase as the value of $p$ increases, which sharpens the discontinuities.
The relations between the nonlinear weights and the dissipation around discontinuities has been discussed briefly and will be explored further in the future work.
Numerical results indicate the improved behavior of the decrease in dissipation around discontinuities while preserving the essentially non-oscillatory behavior in smooth regions.

\section*{Acknowledgments}
We would like to thank Tsz Fun Hung, Xifeng Jin and Biao Wang for the fruitful discussions of nonlinear weight analysis. 
This work was supported in part by Samsung Science \& $\&$ Technology Foundation under the grant number 2021R1A2C3009648.

\section*{Appendix} \label{sec:appendix}
In this appendix, we will prove the inequalities that are used in the Section \ref{sec:new_Z_weight}. 

\begin{theorem} \label{thm:pthroot}
If $a > b > 0$, then
\begin{equation} \label{eq:pthroot}
 \left( \sqrt[p]{a} - \sqrt[p]{b} \right)^p < a-b
\end{equation}
for $p>1$.
\end{theorem} 
\begin{proof}
Let $\eta(x) = \sqrt[p]{x}$ for $x \geqslant 0$. 
Then $\eta$ is strictly concave since $\eta''(x)$ is always negative for $p>1$. 
Set $s \in (0,1)$.
If $x \ne 0$, we have, by the definition of strict concavity,
$$
   \eta(sx) = \eta \left( sx + (1-s) \cdot 0 \right) > s \eta(x) + (1-s) \eta(0) = s \eta(x).
$$
So for $a > b > 0$,
$$
   \eta(a-b) + \eta(b) = \eta \left( a \frac{a-b}{a} \right) + \eta \left( a \frac{b}{a} \right) > \frac{a-b}{a} \eta(a) + \frac{b}{a} \eta(a) = \eta(a).
$$
Thus $\sqrt[p]{a} - \sqrt[p]{b} < \sqrt[p]{a-b}$.
Taking the $p$th power of both sides, we have the inequality \eqref{eq:pthroot}.
\end{proof}

\begin{theorem} \label{thm:phi}
Let $c>1$ and define $\varphi: (0, \infty) \to \R$ by $\varphi(x) = \left( \sqrt[x]{c} - 1 \right)^x$. Then
\begin{enumerate}[label=\alph*)]
\item $\varphi$ is strictly decreasing on $(0, \infty)$.
\item $\displaystyle \lim_{x \to 0^+} \varphi(x) = c$ and $\displaystyle \lim_{x \to \infty} \varphi(x) = 0$.
\end{enumerate}
\end{theorem} 
\begin{proof}
Let $g(x) = \ln \varphi(x) = x \ln \left( \sqrt[x]{c} - 1 \right)$ for $x > 0$. 
Since 
\begin{align*}
 g'(x) &= \ln \left( \sqrt[x]{c} - 1 \right) + x \frac{1}{\sqrt[x]{c} - 1} \sqrt[x]{c} \ln c \left( - \frac{1}{x^2} \right) \\
       &= \ln \left( \sqrt[x]{c} - 1 \right) - \frac{\sqrt[x]{c}}{\sqrt[x]{c} - 1} \frac{1}{x} \ln c  \\
       &< \ln \left( \sqrt[x]{c} - 1 \right) - \ln \sqrt[x]{c} < 0,
\end{align*}
then $g$ is strictly decreasing on $(0, \infty)$.
So is $\varphi(x) = \e^{g(x)}$. 
It is easy to verify that for $x \in (0, \infty),~0 < \varphi(x) < c$.
We show that the upper and lower bounds are the limits of the function $\varphi(x)$ as $x \to \infty$ and $x \to 0^+$, respectively. 
Since
\begin{align*}
 \lim_{x \to 0^+} x \ln \left( \sqrt[x]{c} - 1 \right) &= \lim_{x \to 0^+} \frac{\ln \left( \sqrt[x]{c} - 1 \right)}{\frac{1}{x}} \\
                           &= \lim_{x \to 0^+} \frac{\frac{1}{\sqrt[x]{c} - 1} \sqrt[x]{c} \ln c \left( - \frac{1}{x^2} \right)}{- \frac{1}{x^2}} \\
                                                       &= \lim_{x \to 0^+} \frac{\sqrt[x]{c}}{\sqrt[x]{c} - 1} \ln c \\
                                                       &= \lim_{x \to 0^+} \frac{1}{1 - c^{-1/x}} \ln c \\
                                                       &= \ln c,
\end{align*}
then 
$$
   \lim_{x \to 0^+} \left( \sqrt[x]{c} - 1 \right)^x = \lim_{x \to 0^+} \e^{\ln \left( \sqrt[x]{c} - 1 \right)^x}
                                                     = \exp \left( \lim_{x \to 0^+} x \ln \left( \sqrt[x]{c} - 1 \right) \right)
                                                     = \e^{\ln c} = c.
$$
We first take a look at the limit
$$
   \lim_{x \to \infty} \sqrt[x]{c} = \lim_{x \to \infty} \e^{\ln c^{1/x}} = \exp \left( \lim_{x \to \infty} \frac{\ln c}{x} \right) = 1.
$$
Since $\displaystyle \lim_{x \to \infty} x \ln \left( \sqrt[x]{c} - 1 \right) = - \infty$, then 
$$
   \lim_{x \to \infty} \left( \sqrt[x]{c} - 1 \right)^x = \lim_{x \to \infty} \e^{\ln \left( \sqrt[x]{c} - 1 \right)^x}
                                                        = \exp \left( \lim_{x \to \infty} x \ln \left( \sqrt[x]{c} - 1 \right) \right)
                                                        = 0.
$$
\end{proof}

\begin{corollary} \label{cor:phi}
If $a > b > 0$, then the function 
$$
   \Phi(x) = \left( \sqrt[x]{a} - \sqrt[x]{b} \right)^x = b \left( \sqrt[x]{\frac{a}{b}} - 1 \right)^x
$$
is strictly decreasing on $(0, \infty)$ with 
$$
   \lim_{x \to 0^+} \Phi(x) = a, \; \lim_{x \to \infty} \Phi(x) = 0 \; \text{ and } \: 0 < \Phi(x) < a.
$$
\end{corollary}

%
%
%


\end{document}